\documentclass[12pt]{amsart}
\usepackage{amssymb}
\usepackage{amscd}
\usepackage{euscript}
\usepackage{epsfig}
\usepackage{epic}
\usepackage{eepic}
\def\smallcoxintext#1{\strut\smash{\kern3pt\raise3pt\hbox{#1}}\kern51pt}%
\def\bigtextdot#1{\strut\smash{\kern0pt\raise4pt\hbox{#1}\kern1.5pt}}
\def\smalltextdot#1{\strut\smash{\kern3pt\raise3pt\hbox{#1}\kern2.5pt}}

\def\input{\relax}

\theoremstyle{plain}
\newtheorem{theorem}{Theorem}
\newtheorem{lemma}[theorem]{Lemma}
\newtheorem{corollary}[theorem]{Corollary}

\theoremstyle{remark}
\newtheorem*{caution}{Caution}
\newtheorem*{remark}{Remark}
\newtheorem*{remarks}{Remarks}

\numberwithin{theorem}{section}
\numberwithin{equation}{section}
\numberwithin{table}{section}
\numberwithin{figure}{section}



\newcommand\forms{\EuScript{C}}  
\newcommand\D{\Delta}           
\newcommand\framed{\EuScript{F}} 
\newcommand\moduli{\EuScript{M}} 
\renewcommand\H{\EuScript{H}}    
\newcommand\PGamma{P\Gamma}     
\newcommand\w{\omega}           
\newcommand\wbar{{\bar\omega}}  
\newcommand\x{\chi} 
\newcommand\WEck{W^{\rm\scriptstyle Eck}} 
\newcommand\e{\varepsilon}
\newcommand\F{\mathbb{F}}       
\newcommand\R{\mathbb{R}}       
\newcommand\C{\mathbb{C}}       
\newcommand\Z{\mathbb{Z}}       
\newcommand\Q{\mathbb{Q}}       
\newcommand\E{\EuScript{E}}      
\newcommand\A{\EuScript{A}}      
\newcommand\ch{\C H}            
\newcommand\cp{\C P}            
\newcommand\rp{\R P}            
\newcommand\GL{{\rm GL}}

\newcommand\PO{{\rm PO}}
\newcommand\SO{{\rm SO}}
\renewcommand\O{{\rm O}}
\DeclareMathOperator{\Isom}{Isom}
\DeclareMathOperator{\Aut}{Aut}
\DeclareMathOperator{\PAut}{\hbox{$P$\kern-1pt Aut}}
\DeclareMathOperator{\Trace}{\hbox{Tr}}
\DeclareMathOperator{\Ad}{\hbox{Ad}}

\DeclareMathOperator{\dimension}{\hbox{dim}}
\DeclareMathOperator{\diag}{\hbox{diag}}
\newcommand\isomorphism{\cong}
\newcommand\tensor{\otimes}
\newcommand\sset{\subseteq}
\newcommand\semidirect{\rtimes}
\newcommand\orbpi{\pi^{\rm orb}}
\newcommand\breakok{\discretionary{}{}{}}
\def\mathllap#1{\mathchoice
{\llap{$\displaystyle #1$}}%
{\llap{$\textstyle #1$}}%
{\llap{$\scriptstyle #1$}}%
{\llap{$\scriptscriptstyle #1$}}}
\def\set#1#2{\left\{\,#1\mathllap{\phantom{#2}}\mathrel{\relax}\right|\left.#2\mathllap{\phantom{#1}}\,\right\}}
\def\centeroverfull#1{\setbox0=\hbox{#1}\newdimen\foo\foo=\wd0\advance\foo
by -\hsize\divide\foo by 2\advance\foo by\hsize\rlap{\kern\foo\llap{\box0}}}
\newcommand\0{\phantom{0}}

\newcommand{\phantomi}{\phantom{i}}
\let\originaldash\-

\newcount\edgelengthC
\newcount\noderadiusC
\newcount\doublebondoffsetC
\newcount\triplebondoffsetC

\newdimen\edgelengthD
\newdimen\noderadiusD

\newcount\nodediameterC

\def\largediagrams{%
  \edgelengthC=70
  \noderadiusC=5
  \doublebondoffsetC=180
  \triplebondoffsetC=280
  \edgelengthD=\edgelengthC pt
  \noderadiusD=\noderadiusC pt
  \nodediameterC=\noderadiusD
  \multiply\nodediameterC by 2
}
\def\smalldiagrams{%
  \edgelengthC=50
  \noderadiusC=3
  \doublebondoffsetC=120
  \triplebondoffsetC=200
  \edgelengthD=\edgelengthC pt
  \noderadiusD=\noderadiusC pt
  \nodediameterC=\noderadiusD
  \multiply\nodediameterC by 2
}
\largediagrams

\newcount\UxC
\newcount\UyC
\newcount\VxC
\newcount\VyC
\newdimen\dashlengthD
\newcount\dashlengthC
\def\bond#1#2#3#4{%
  \thinlines
  \UxC=#1
  \UyC=#2
  \VxC=#3
  \VyC=#4
  \drawline(\UxC,\UyC)(\VxC,\VyC)
}
\def\heavybond#1#2#3#4{%
  \Thicklines
  \UxC=#1
  \UyC=#2
  \VxC=#3
  \VyC=#4
  \drawline(\UxC,\UyC)(\VxC,\VyC)
}
\def\dashedbond#1#2#3#4{%
  \Thicklines
  \UxC=#1
  \UyC=#2
  \VxC=#3
  \VyC=#4
  \dashlengthD=8pt
  \dashlengthC=\dashlengthD
  \dashline[60]{\dashlengthC}(\UxC,\UyC)(\VxC,\VyC)
}

\newdimen\UxD
\newdimen\UyD
\newdimen\VxD
\newdimen\VyD
\newdimen\XoffsetD
\newdimen\YoffsetD
\def\doublebond#1#2#3#4#5#6{%
  \thinlines
  \UxD=#1
  \UyD=#2
  \VxD=#3
  \VyD=#4
  \XoffsetD=#5
  \YoffsetD=#6
  \multiply\XoffsetD by \doublebondoffsetC
  \multiply\YoffsetD by \doublebondoffsetC
  \divide\XoffsetD by 100
  \divide\YoffsetD by 100
  \advance\UxD by \XoffsetD
  \advance\UyD by \YoffsetD
  \advance\VxD by \XoffsetD
  \advance\VyD by \YoffsetD
  \bond\UxD\UyD\VxD\VyD
  \advance\UxD by -\XoffsetD
  \advance\UyD by -\YoffsetD
  \advance\VxD by -\XoffsetD
  \advance\VyD by -\YoffsetD
  \advance\UxD by -\XoffsetD
  \advance\UyD by -\YoffsetD
  \advance\VxD by -\XoffsetD
  \advance\VyD by -\YoffsetD
  \bond\UxD\UyD\VxD\VyD
}
\def\triplebond#1#2#3#4#5#6{%
  \thinlines
  \UxD=#1
  \UyD=#2
  \VxD=#3
  \VyD=#4
  \XoffsetD=#5
  \YoffsetD=#6
  \multiply\XoffsetD by \triplebondoffsetC
  \multiply\YoffsetD by \triplebondoffsetC
  \divide\XoffsetD by 100
  \divide\YoffsetD by 100
  \bond\UxD\UyD\VxD\VyD
  \advance\UxD by \XoffsetD
  \advance\UyD by \YoffsetD
  \advance\VxD by \XoffsetD
  \advance\VyD by \YoffsetD
  \bond\UxD\UyD\VxD\VyD
  \advance\UxD by -\XoffsetD
  \advance\UyD by -\YoffsetD
  \advance\VxD by -\XoffsetD
  \advance\VyD by -\YoffsetD
  \advance\UxD by -\XoffsetD
  \advance\UyD by -\YoffsetD
  \advance\VxD by -\XoffsetD
  \advance\VyD by -\YoffsetD
  \bond\UxD\UyD\VxD\VyD
}

\def\hollownode#1#2{%
  \thinlines
  \UxC=#1
  \UyC=#2
  \filltype{white}\put(\UxC,\UyC){\circle*{\nodediameterC}}
}
\def\solidnode#1#2{%
  \thinlines
  \UxC=#1
  \UyC=#2
  \filltype{black}\put(\UxC,\UyC){\circle*{\nodediameterC}}
}
\def\spokeup#1#2{%
  \thinlines
  \UxD=#1
  \UyD=#2
  \VxD=0pt
  \VyD=\noderadiusD
  \advance\VxD by \UxD
  \advance\VyD by \UyD
  \bond\UxD\UyD\VxD\VyD
}
\def\spokedown#1#2{%
  \thinlines
  \UxD=#1
  \UyD=#2
  \VxD=0pt
  \VyD=-\noderadiusD
  \advance\VxD by \UxD
  \advance\VyD by \UyD
  \bond\UxD\UyD\VxD\VyD
}
\def\spokeupleft#1#2{%
  \thinlines
  \UxD=#1
  \UyD=#2
  \VxD=\noderadiusD
  \VyD=\noderadiusD
  \multiply\VxD by -1732
  \divide\VxD by 2000
  \divide\VyD by 2
  \advance\VxD by \UxD
  \advance\VyD by \UyD
  \bond\UxD\UyD\VxD\VyD
}
\def\spokeupright#1#2{%
  \thinlines
  \UxD=#1
  \UyD=#2
  \VxD=\noderadiusD
  \VyD=\noderadiusD
  \multiply\VxD by 1732
  \divide\VxD by 2000
  \divide\VyD by 2
  \advance\VxD by \UxD
  \advance\VyD by \UyD
  \bond\UxD\UyD\VxD\VyD
}
\def\spokedownleft#1#2{%
  \thinlines
  \UxD=#1
  \UyD=#2
  \VxD=\noderadiusD
  \VyD=\noderadiusD
  \multiply\VxD by -1732
  \divide\VxD by 2000
  \divide\VyD by -2
  \advance\VxD by \UxD
  \advance\VyD by \UyD
  \bond\UxD\UyD\VxD\VyD
}
\def\spokedownright#1#2{%
  \thinlines
  \UxD=#1
  \UyD=#2
  \VxD=\noderadiusD
  \VyD=\noderadiusD
  \multiply\VxD by 1732
  \divide\VxD by 2000
  \divide\VyD by -2
  \advance\VxD by \UxD
  \advance\VyD by \UyD
  \bond\UxD\UyD\VxD\VyD
}
\def\twonode#1#2{%
  \hollownode{#1}{#2}
  \spokeup{#1}{#2}
  \spokedown{#1}{#2}
}
\def\threenode#1#2{%
  \hollownode{#1}{#2}
  \spokeup{#1}{#2}
  \spokedownright{#1}{#2}
  \spokedownleft{#1}{#2}
}
\def\sixnode#1#2{%
  \hollownode{#1}{#2}
  \spokeup{#1}{#2}
  \spokedown{#1}{#2}
  \spokeupright{#1}{#2}
  \spokedownright{#1}{#2}
  \spokeupleft{#1}{#2}
  \spokedownleft{#1}{#2}
}

\def\myput#1#2#3{%
  \UxC=#1
  \UyC=#2
  \put(\UxC,\UyC){#3}
}
\newcount\XoffsetC
\newcount\YoffsetC
\def\nearnode#1#2#3#4#5#6#7#8{%
  \UxD=#1
  \UyD=#2
  \XoffsetC=#3
  \YoffsetC=#4
  \XoffsetD=\noderadiusD
  \YoffsetD=\noderadiusD
  \multiply\XoffsetD by \XoffsetC
  \multiply\YoffsetD by \YoffsetC
  \divide\XoffsetD by 100
  \divide\YoffsetD by 100
  \advance\UxD by \XoffsetD
  \advance\UyD by \YoffsetD
  \myput\UxD\UyD{\kern#5\makebox(0,0)[#7]{\raise#6\hbox{#8}}}
}

\newdimen\Ax
\newdimen\Ay
\newdimen\Bx
\newdimen\By
\newdimen\Cx
\newdimen\Cy
\newdimen\Dx
\newdimen\Dy
\newdimen\Ex
\newdimen\Ey
\newdimen\Fx
\newdimen\Fy
\newdimen\Gx
\newdimen\Gy
\newdimen\ABperpX
\newdimen\ABperpY
\newdimen\BCperpX
\newdimen\BCperpY
\newdimen\CDperpX
\newdimen\CDperpY

\begin{document}

\title[Moduli of Real Cubic Surfaces]{Hyperbolic geometry and moduli of real cubic surfaces }
\author{Daniel Allcock}
\address{Department of Mathematics\\University of Texas at Austin\\Austin, TX 78712}
\email{allcock@math.utexas.edu}
\urladdr{http://www.math.utexas.edu/\textasciitilde allcock}
\author{James A. Carlson}
\address{Department of Mathematics\\University of Utah\\Salt
Lake City, UT 84112}
\curraddr{Clay Mathematics Institute, One Bow Street,
Cambridge, Massachusetts 02138}
\email{carlson@claymath.org}
\urladdr{http://www.math.utah.edu/\textasciitilde carlson}
\author{Domingo Toledo}
\address{Department of Mathematics\\University of Utah\\Salt
Lake City, UT 84112}
\email{toledo@math.utah.edu}
\urladdr{http://www.math.utah.edu/\textasciitilde toledo}
\date{18 Apr 2009}
\thanks{First author partly supported by NSF grants DMS~0070930,
  DMS-0231585 and DMS-0600112.
Second and third authors partly supported by NSF grants DMS~9900543
, DMS-0200877 and DMS-0600816.  The second author thanks the Clay Mathematics
  Institute for its support.}
\begin{abstract}
Let  $\moduli_0^\R$ be the moduli space of smooth real cubic
surfaces. We show that each of its components admits a real hyperbolic structure.  More precisely, one can
 remove some lower-dimensional geodesic subspaces from a real hyperbolic
space $H^4$ and form the
quotient by an arithmetic group to obtain an 
orbifold isomorphic to a component of
the moduli space.  There are five components.  For each we describe the
corresponding lattices in $\PO(4,1)$.  We also derive several new
and several old results on the topology of $\moduli_0^\R$.  
Let $\moduli_s^\R$ be the moduli space of
real cubic surfaces that are stable in the sense of geometric
invariant theory.  We show that this space carries a hyperbolic structure 
whose restriction to $\moduli_0^\R$ is that just
mentioned. The corresponding lattice in $\PO(4,1)$, for which we find an explicit
fundamental domain, is
nonarithmetic. 
\end{abstract} 
\maketitle

\section{Introduction}
\label{sec-intro}
The purpose of this paper is to study the geometry and topology of the moduli space of real cubic surfaces in $\mathbb{R} P^3$.    It is a classical fact, going back to Schl\"afli \cite{schlafli-five, schlafli-italian} and Klein \cite{klein}, that the moduli space of smooth real cubic surfaces has five connected components.  We show in this paper that each of these components has a real hyperbolic structure that we  compute explicitly both in arithmetic and in geometric terms.  We use this geometric structure to compute, to a large extent,  the topology of each component.   These structures are not complete.  We also prove a more subtle result, that the moduli space of stable real cubic surfaces  has a real hyperbolic structure, which is complete, and that restricts, on each component of the moduli space of smooth surfaces, to the (incomplete) structures just mentioned.  The most surprising fact to us is that the resulting discrete group of isometries of hyperbolic space is not arithmetic.

To describe our results, we use the following notation.  We write $\forms$ for the space of non-zero cubic forms with complex coefficients in $4$ variables, $\D$ for the discriminant locus (forms where all partial derivatives have a common zero), $\forms_0$ for the space $\forms-\D$ of forms that define a smooth hypersurface in $\cp^3$, and $\forms_s$ for the space of forms that are stable in the sense of geometric invariant theory for the action of $GL(4,\C)$ on $\forms$.  It is classical that these are the forms that define a cubic surface which is either smooth or has only nodal singularities \cite[\S19]{Hilbert}.

We denote all the corresponding real objects with a superscript $\R$.
Thus $\forms^\R$ denotes the space of non-zero cubic forms with real
coefficients, and $\D^\R, \forms_0^\R$ and $\forms_s^\R$ the
intersection with $\forms^\R$ of the corresponding subspaces of
$\forms$.  We will also use the prefix $P$ for the corresponding
projective objects, thus $P\forms^\R\cong\rp^{19}$ is the projective
space of cubic forms with real coefficients, and $P\D^R$,
$P\forms_0^\R$, $P\forms_s^\R$ are the images of the objects just
defined.  The group $GL(4,\R)$ acts properly on $\forms_0^\R$ and
$\forms_s^\R$ (equivalently, $PGL(4,\R)$ acts properly on
$P\forms_0^\R$ and $P\forms_s^\R$) and we write $\moduli_0^\R$ and
$\moduli_s^\R$ for the corresponding quotient spaces, namely the
moduli spaces of smooth and of stable real cubic surfaces.

The space $P\D^\R$ has real codimension one in $P\forms$, its
complement $P\forms_0^\R$ has five connected components, and the
topology of a surface in each component is classically known
\cite{schlafli-italian, klein, segre}.  We label the components
$P\forms_{0,j}^\R$, for $j = 0, 1, 2, 3, 4$, choosing the indexing so
that a surface in $P\forms_{0,j}^\R$ is topologically a real
projective plane with $3 - j$ handles attached (see
table~\ref{tab-volumes-table}; the case of $-1$ many handles means the
disjoint union $\rp^2\sqcup S^2$.)  It follows that the moduli space
$\moduli_0^\R$ has five connected components, $\moduli_{0,j}^\R$, for
$j = 0, 1, 2, 3, 4$.  We can now state our first theorem:

\begin{theorem}
\label{thm-main-theorem-smooth}
For each $j = 0, \dots , 4$ there is a union $\H_j$ of two- and three-dimensional geodesic subspaces
of the four-dimensional real hyperbolic space 
$H^4$ and an isomorphism of real analytic orbifolds 
\[
   \moduli_{0,j}^\R \cong \PGamma^\R_j\backslash (H^4 - \H_j).
\]
Here $\PGamma^\R_j$ is the projectivized group of integer matrices
which are orthogonal with respect to the quadratic form obtained from
the diagonal form $[-1,1,1,1,1]$ by replacing the last $j$ of the
$1$'s by $3$'s.
\end{theorem}

The real hyperbolic structure on the component $\moduli_{0,0}^\R$
has been studied by Yoshida \cite{Yoshida-realcubics}.  The other
cases are new.

The space $P\forms_s^\R$ is connected, since it is obtained from the manifold $P\forms^\R$ by removing a subspace of codimension two (part of the singular set of $P\D^\R$).  Thus the moduli space $\moduli_s^\R$ is connected.  We have the following uniformization theorem for this space:

\begin{theorem}
\label{thm-main-theorem-stable}
There is a nonarithmetic lattice 
$\PGamma^\R\subset PO(4,1)$ and a homeomorphism
\[
  \moduli_s^\R \cong \PGamma^\R \backslash H^4.
\]
Moreover, there is a $\PGamma^\R$-invariant union of two- and
three-dimensional geodesic subspaces $\H'$ of $H^4$ so that this homeomorphism restricts to an isomorphism of real analytic
orbifolds,
\[
   \moduli_0^\R \cong \PGamma^\R \backslash (H^4 - \H').
\]
\end{theorem}

 To our knowledge this is the first appearance of a non-arithmetic lattice in a moduli problem for real varieties.  Observe that the group $\PGamma^\R$ uniformizes a space assembled from arithmetic pieces much in the spirit of the construction by Gromov and Piatetskii-Shapiro of non-arithmetic lattices in real hyperbolic space.  We thus view this theorem as an appearance ``in nature'' of their construction.
 
\newbox\Wzerobox
\setbox\Wzerobox=\hbox{%
\begin{picture}(0,0)%
\smalldiagrams
\setlength{\unitlength}{1sp}%
  \Ax=0pt    \Ay=0pt
  \Bx=\edgelengthD  \By=0pt
  \Cx=\edgelengthD  \Cy=0pt
  \Dx=\edgelengthD  \Dy=0pt
  \multiply\Cx by 2
  \multiply\Dx by 3
  \Ex=\edgelengthD  \Ey=-\edgelengthD
  \bond\Ax\Ay\Bx\By
  \bond\Bx\By\Cx\Cy
  \bond\Cx\Cy\Dx\Dy
  \doublebond\Bx\By\Ex\Ey{1pt}{0pt}%
  \hollownode\Ax\Ay
  \hollownode\Bx\By
  \hollownode\Cx\Cy
  \hollownode\Dx\Dy
  \solidnode\Ex\Ey
  \UxD=\edgelengthD \UyD=-\edgelengthD
  \divide\UxD by 2
  \divide\UyD by 2
  \nearnode\UxD\UyD{0}{0}{0pt}{0pt}{c}{$W_0$}%
\end{picture}%
}
\newbox\Wonebox
\setbox\Wonebox=\hbox{%
\begin{picture}(0,0)%
\smalldiagrams
\setlength{\unitlength}{1sp}%
%
%
%
%
%
  \Ax=     0pt  \Ay= 1.152pt
  \Bx=  .901pt  \By=  .718pt
  \Cx= 1.123pt  \Cy=- .256pt
  \Dx=  .5  pt  \Dy=-1.038pt
  \Ex=- .5  pt  \Ey=-1.038pt
  \Fx=-1.123pt  \Fy=- .256pt
  \Gx=- .901pt  \Gy=  .718pt
  \multiply\Ax by \edgelengthC
  \multiply\Ay by \edgelengthC
  \multiply\Bx by \edgelengthC
  \multiply\By by \edgelengthC
  \multiply\Cx by \edgelengthC
  \multiply\Cy by \edgelengthC
  \multiply\Dx by \edgelengthC
  \multiply\Dy by \edgelengthC
  \multiply\Ex by \edgelengthC
  \multiply\Ey by \edgelengthC
  \multiply\Fx by \edgelengthC
  \multiply\Fy by \edgelengthC
  \multiply\Gx by \edgelengthC
  \multiply\Gy by \edgelengthC
  \ABperpX=.434pt
  \ABperpY=.901pt
  \doublebond\Ax\Ay\Bx\By\ABperpX\ABperpY
  \doublebond\Ax\Ay\Gx\Gy{-\ABperpX}\ABperpY
  \bond\Bx\By\Cx\Cy
  \bond\Gx\Gy\Fx\Fy
  \heavybond\Dx\Dy\Ex\Ey
  \dashedbond\Cx\Cy\Dx\Dy
  \dashedbond\Ex\Ey\Fx\Fy
  \solidnode\Ax\Ay
  \hollownode\Bx\By
  \hollownode\Cx\Cy
  \solidnode\Dx\Dy
  \solidnode\Ex\Ey
  \hollownode\Fx\Fy
  \hollownode\Gx\Gy
  \nearnode{0pt}{0pt}{0}{0}{0pt}{0pt}{c}{$W_1$}%
\end{picture}%
}
\newbox\Wtwobox
\setbox\Wtwobox=\hbox{%
\begin{picture}(0,0)%
\smalldiagrams
\setlength{\unitlength}{1sp}%
%
%
%
%
%
  \Ax=     0pt  \Ay= 1.152pt
  \Bx=  .901pt  \By=  .718pt
  \Cx= 1.123pt  \Cy=- .256pt
  \Dx=  .5  pt  \Dy=-1.038pt
  \Ex=- .5  pt  \Ey=-1.038pt
  \Fx=-1.123pt  \Fy=- .256pt
  \Gx=- .901pt  \Gy=  .718pt
  \multiply\Ax by \edgelengthC
  \multiply\Ay by \edgelengthC
  \multiply\Bx by \edgelengthC
  \multiply\By by \edgelengthC
  \multiply\Cx by \edgelengthC
  \multiply\Cy by \edgelengthC
  \multiply\Dx by \edgelengthC
  \multiply\Dy by \edgelengthC
  \multiply\Ex by \edgelengthC
  \multiply\Ey by \edgelengthC
  \multiply\Fx by \edgelengthC
  \multiply\Fy by \edgelengthC
  \multiply\Gx by \edgelengthC
  \multiply\Gy by \edgelengthC
  \ABperpX=.434pt
  \ABperpY=.901pt
  \BCperpX=.975pt
  \BCperpY=.223pt
  \CDperpX=.782pt
  \CDperpY=-.623pt
  \doublebond\Ax\Ay\Bx\By\ABperpX\ABperpY
  \doublebond\Ax\Ay\Gx\Gy{-\ABperpX}\ABperpY
  \triplebond\Bx\By\Cx\Cy\BCperpX\BCperpY
  \triplebond\Gx\Gy\Fx\Fy{-\BCperpX}\BCperpY
  \heavybond\Dx\Dy\Ex\Ey
  \doublebond\Cx\Cy\Dx\Dy\CDperpX\CDperpY
  \doublebond\Ex\Ey\Fx\Fy{-\CDperpX}\CDperpY
  \solidnode\Ax\Ay
  \hollownode\Bx\By
  \hollownode\Cx\Cy
  \solidnode\Dx\Dy
  \solidnode\Ex\Ey
  \hollownode\Fx\Fy
  \hollownode\Gx\Gy
  \nearnode{0pt}{0pt}{0}{0}{0pt}{0pt}{c}{$W_2$}%
\end{picture}%
}
\newbox\Wthreebox
\setbox\Wthreebox=\hbox{%
\begin{picture}(0,0)%
\smalldiagrams
\setlength{\unitlength}{1sp}%
%
  \Cx=-.851pt  \Cy= 0    pt
  \Dx=-.263pt  \Dy=- .809pt
  \Ex= .688pt  \Ey=- .5pt
  \Fx= .688pt  \Fy=  .5pt
  \Gx=-.263pt  \Gy=  .809pt
  \multiply\Cx by \edgelengthC
  \multiply\Cy by \edgelengthC
  \multiply\Dx by \edgelengthC
  \multiply\Dy by \edgelengthC
  \multiply\Ex by \edgelengthC
  \multiply\Ey by \edgelengthC
  \multiply\Fx by \edgelengthC
  \multiply\Fy by \edgelengthC
  \multiply\Gx by \edgelengthC
  \multiply\Gy by \edgelengthC
  \Bx=\Cx  \By=0pt
  \advance\Bx by -\edgelengthD
  \Ax=\Bx  \Ay=0pt
  \advance\Ax by -\edgelengthD
  \CDperpX=.809pt
  \CDperpY=.588pt
  \triplebond\Ax\Ay\Bx\By{0pt}{1pt}%
  \bond\Bx\By\Cx\Cy
  \doublebond\Cx\Cy\Dx\Dy\CDperpX\CDperpY
  \dashedbond\Dx\Dy\Ex\Ey
  \heavybond\Ex\Ey\Fx\Fy
  \dashedbond\Fx\Fy\Gx\Gy
  \bond\Gx\Gy\Cx\Cy
  \hollownode\Ax\Ay
  \hollownode\Bx\By
  \hollownode\Cx\Cy
  \solidnode\Dx\Dy
  \solidnode\Ex\Ey
  \solidnode\Fx\Fy
  \hollownode\Gx\Gy
  \nearnode{0pt}{0pt}{0}{0}{0pt}{0pt}{c}{$W_3$}%
\end{picture}%
}
\newbox\Wfourbox
\setbox\Wfourbox=\hbox{%
\begin{picture}(0,0)%
\smalldiagrams
\setlength{\unitlength}{1sp}%
  \Ax=0pt    \Ay=0pt
  \Bx=\edgelengthD  \By=0pt
  \Cx=\edgelengthD  \Cy=0pt
  \Dx=\edgelengthD  \Dy=0pt
  \Ex=\edgelengthD  \Ey=0pt
  \Fx=\edgelengthD  \Fy=0pt
  \multiply\Cx by 2
  \multiply\Dx by 3
  \multiply\Ex by 4
  \multiply\Fx by 5
  \triplebond\Ax\Ay\Bx\By{0pt}{1pt}%
  \bond\Bx\By\Cx\Cy
  \bond\Cx\Cy\Dx\Dy
  \doublebond\Dx\Dy\Ex\Ey{0pt}{1pt}%
  \heavybond\Ex\Ey\Fx\Fy
  \hollownode\Ax\Ay
  \hollownode\Bx\By
  \hollownode\Cx\Cy
  \hollownode\Dx\Dy
  \solidnode\Ex\Ey
  \solidnode\Fx\Fy
  \UxD=-\edgelengthD
  \divide\UxD by 2
  \nearnode\UxD{0pt}{0}{0}{0pt}{0pt}{c}{$W_4$}%
\end{picture}%
}
\begin{figure}
\def\LLx{0}
\def\LLy{0}
\def\width{320}
\def\height{220}
\def\Wtwoheight{85}
\def\Wthreeheight{54}
\setlength{\unitlength}{1bp}
\begin{picture}(\width,\height)(\LLx,\LLy)
\put(15,\height){\makebox(0,0)[tl]{\unhbox\Wzerobox}}
\put(305,\height){\makebox(-60,-57)[tr]{\unhbox\Wonebox}}
\put(60,\Wtwoheight){\makebox(-60,-57)[bl]{\unhbox\Wtwobox}}
\put(\width,\Wthreeheight){\makebox(-40,0)[br]{\unhbox\Wthreebox}}
\put(45,0){\makebox(0,0)[br]{\unhbox\Wfourbox}}
\end{picture}
\caption{Coxeter polyhedra for the reflection subgroups $W_j$ of
$\PGamma_j^\R$.  The blackened nodes and triple bonds correspond to
faces of the polyhedra that represent singular cubic surfaces.  See the text for the explanation of the edges.}
\label{fig-small-Coxeter-diagrams}
\end{figure}
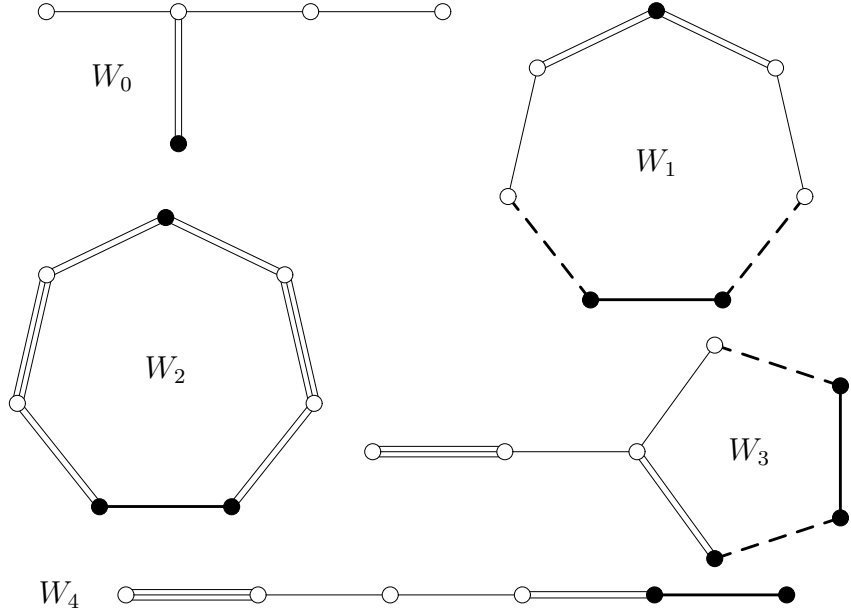

We obtain much more information about the groups $\PGamma_j^\R$ of and
$\PGamma^\R$ than we have stated here.
Section~\ref{sec-H4-stabilizers} gives an arithmetic description
of each $\PGamma_j^\R$ and shows that they are essentially Coxeter
groups. (Precisely: they are Coxeter groups for $j = 0, 3, 4$ and
contain a Coxeter subgroup of index $2$ if $j = 1,2$).  We use
Vinberg's algorithm to derive their Coxeter diagrams and consequently
their fundamental domains.  So we have a very explicit geometric
description of the groups $\PGamma_j^\R$.  The results are summarized
in figure~\ref{fig-small-Coxeter-diagrams}.  In these diagrams the
nodes represent facets of the polyhedron, and two facets meet at an
angle of $\pi/2,\ \pi/3,\ \pi/4$ or $\pi/6$, or are parallel (meet at
infinity) or are ultraparallel, if the number of bonds between the two
corresponding nodes is respectively $0$, $1$, $2$, $3$, or a heavy or
dashed line.  See section~\ref{sec-H4-stabilizers} for more details.
 
The group $\PGamma^\R$ is not a Coxeter group (even up to finite
index) but we find that a subgroup of index two has a fundamental
domain that is a Coxeter polyhedron.  We describe this polyhedron
explicitly in section~\ref{sec-gluing}, thus we have a concrete
geometric description of $\PGamma^\R$, and we also find a
representation of this group by matrices with coefficients in
$\Z[\sqrt{3}]$.
 
\begin{table}
\begin{tabular}[t]{ccccc}
$j$&Topology&Real Lines&Real Tritang. & Monodromy on Lines  \\ 
{}&of Surface&{}& Planes &{}\\
\noalign{\vskip1.5pt}
\hline
\noalign{\vskip2pt}
0 & $\rp^2 \ \# \  3 T^2$& 27 &45 & $A_5$  \\
1 & $\rp^2\  \# \  2 T^2$& 15  & 15 & $S_3\times S_3$ \\
2 & $\rp^2  \# \ T^2$ &7 &  5 &$(\Z/2)^3\rtimes\Z/2$      \\
3 & $\rp^2$ &3 & 7 &  $S_4$\\
4 & $\rp^2 \sqcup S^2$ & 3 & 13 & $S_4$  \\
\noalign{\vskip1pt}
\hline
\end{tabular}
\vskip.5cm
\caption{The  classical results on the components of the moduli space
  of real cubic surfaces.   The components are indexed by $j$ according to
  our conventions.   The third item in the last column corrects an
  error of Segre.} 
\label{tab-volumes-table}
\end{table}

\begin{table}
\begin{tabular}[t]{ccccc}
$j$&Euler&Volume &Fraction &$\orbpi_1(\moduli_{0,j}^\R)$\\ 
\noalign{\vskip1.5pt}
\hline
\noalign{\vskip2pt}
0&$1/1920$& .00685&\phantom{0}2.03\% &$S_5$\\                                     
1& 1/288&.04569&13.51\%&$(S_3\times S_3)\rtimes\Z/2$ \\
2&5/576&.11423&33.78\%&$(D_\infty\times D_\infty)\rtimes\Z/2$\\
3& 1/96 &  .13708  & 40.54\%&
\newcommand{\braceheight}{12pt}
\newcommand{\braceloweramount}{6pt}
\newcommand{\diagloweramount}{5pt}
\smash{\lower\braceloweramount\hbox{$\left.\vrule width0pt
height\braceheight\right\rbrace$}}%
\smash{\lower\diagloweramount\hbox{%
\;\;\begin{picture}(0,0)
\smalldiagrams
\setlength{\unitlength}{1sp}
%
%
%
\Ax=0pt    \Ay=0pt
\Bx=\edgelengthD  \By=0pt
\Cx=\edgelengthD  \Cy=0pt
\Dx=\edgelengthD  \Dy=0pt
\multiply\Cx by 2
\multiply\Dx by 3
\bond\Ax\Ay\Bx\By
\bond\Bx\By\Cx\Cy
\bond\Cx\Cy\Dx\Dy
\hollownode\Ax\Ay
\hollownode\Bx\By
\hollownode\Cx\Cy
\hollownode\Dx\Dy
\divide\Bx by 2
\myput\Bx\By{\makebox(0,0)[b]{\raise 4pt\hbox{$\infty$}}}
\end{picture}
}}%
\kern151pt
\\
4& 1/384 & .03427  & 10.14\%\\
\hline
\noalign{\vskip1pt}
&  &   \\
\end{tabular}
\caption{The orbifold Euler characteristic, volume, fraction of total volume, and orbifold fundamental groups of the moduli
  spaces $\moduli_{0,j}^\R$. 
  See theorem~\ref{thm-orbifold-fundamental-groups} for the notation.}
\label{tab-monodromy-groups}
\end{table}

Much of the classical theory of real cubic surfaces, as well as new
results, are encoded in these Coxeter diagrams.  The new results are
our computation of the groups $\orbpi_1(\moduli_{0,j}^\R)$ (see
table~\ref{tab-monodromy-groups}) and our proof that each
$\moduli_{0,j}^\R$ has contractible universal cover.  These results
appear in section~\ref{sec-topologysmoothmoduli}, where we describe
the topology of the spaces $\moduli_{0,j}^\R$.  As an application to
the classical theory, we re-compute the monodromy representation of
$\pi_1(P\forms_{0,j}^\R )$ on the configuration of lines on a cubic
surface, which was first computed by Segre in his treatise
\cite{segre}.  We confirm four of his computations and correct an
error in the remaining one (the case $j=2$).  See the last column of
table~\ref{tab-volumes-table} and section~\ref{sec-Segre} for details.
We also compute the hyperbolic volume of each component in
section~\ref{sec-volume}.  The results are summarized in
table~\ref{tab-monodromy-groups}.


\medskip Our methods are based on our previous work on the complex
hyperbolic structure of the moduli space of complex cubic surfaces
\cite{ACT}.  We proved that this moduli space $\moduli_s$ is
isomorphic to the quotient $\PGamma\backslash\ch^4$ of complex
hyperbolic 4-space $\ch^4$ by the lattice $\PGamma = PU(4,1,\E)$ in
$PU(4,1)$, where $\E$ is the ring of integers in $\Q(\sqrt{-3})$.  We
also showed that there is an infinite hyperplane arrangement $\H$ in
$\ch^4$ which is $\PGamma$-invariant and corresponds to the
discriminant $P\D$.  Thus there is also an identification of the
moduli space $\moduli_0$ of smooth cubic surfaces with the quotient
$\PGamma\backslash(\ch^4 - \H)$.  The natural map
$\moduli_0^\R\to\moduli_0$ (which is finite to one but not injective)
allows us to give a real hyperbolic structure on $\moduli_0^\R$ and
thus prove Theorem~\ref{thm-main-theorem-smooth}.  The essence of the
proof appears in section~\ref{sec-smoothmoduli}, with refinements in
sections~\ref{sec-five-families}--\ref{sec-discr-and-integral-H4s}.
Theorem~\ref{thm-main-theorem-stable} is considerably more subtle, and
does not follow simply from the corresponding map $\moduli_s^\R
\to\moduli_s$.  Its proof occupies
sections~\ref{sec-stable-moduli}--\ref{sec-nonarithmeticity}.

The subject of real cubic surfaces has a long and fascinating history.
It was Schl\"afli who first discovered that there are five distinct
types of real smooth cubic surfaces, distinguished by the numbers of
real lines and real tritangent planes; see table~\ref{tab-volumes-table}.
He summarized his results in his 1858 paper \cite{schlafli-five}.  In
a later paper \cite{schlafli-italian}, followed by corrections after 
correspondence with Klein, Schl\"afli also determined the topology
(more precisely, the \lq\lq connectivity") of a surface in each of his
five types.  In particular he showed that the topology is constant
in each type.  It is clear from these two later papers that he had a
mental picture of the adjacency relationship of the components and how
the topology of a smooth surface changes by \lq\lq surgery" in going
from one class to a neighboring one by crossing the discriminant.

Even though it is not closely related to the present discussion, we
should also mention Schl\"afli's monumental paper
\cite{schlafli-distribution} where he classifies all possible real 
cubic surfaces with all their possible singularities.  One particular
fact is that a cubic surface can have at most four nodes, and that
there is a real cubic surface with four real nodes.

In 1873 Klein \cite{klein} gave a very clear picture of the space of
smooth real cubic surfaces, of how the discriminant separates it into
components, and of the topology of a surface in each of Schl\"afli's
classes.  He obtained all cubic surfaces by deforming the
four-real-nodal cubic surface, and in this way he could see the
topology of each nodal or smooth real cubic surface.  He also proved
that the space of smooth real cubic surfaces has five connected
components and that Schl\"afli's classes coincide with the
components. The proof required the knowledge of the space of nodal
cubic surfaces, and also some information on cuspidal ones.  Klein was
not satisfied with the arguments in his original 1873 paper
\cite{klein}, and made corrections and substantial amplifications in
the version published in his collected works in 1922.

A proof, by Klein's method, that $P\forms_0^\R$ has five connected components  is given in \S 24 of  Segre's book \cite{segre}.  Segre also determines in \S 64--68  the topology of a surface in each component by studying how the  surface is divided into cells by its real lines.  This book also contains a wealth of information about the real cubic surfaces, including the monodromy representations mentioned above, and much  detailed information on the various configurations of lines.

Since the classification of real cubic surfaces is a special case of
the classification of real cubic hypersurfaces, we also review the
history of the latter.  The earliest work is Newton's classification
of real cubic curves \cite{newton-optics, newton-works}.  But in
dimensions higher than two the classification results are very recent.
We have Krasnov's classification of real cubic threefolds
\cite{krasnov}, which is based on Klein's method of determining the
discriminant and then deforming away to see the components.  Finashin
and Kharlamov have two papers on the classification of real cubic
fourfolds.  The first one \cite{finashin-kharlamov} is based partly on complex period maps (to study
the discriminant) and partly on Klein's method.  
The second \cite{finashin-kharlamov-2} is based on the surjectivity
of the period map for complex cubic fourfolds recently proved by Laza
and Looijenga \cite{laza, looijenga}.  We are not aware of any
classification of real cubic hypersurfaces beyond dimension four.
  
Our approach to $\moduli_0^\R$, namely studying a complex period map
and its interaction with anti-holomorphic involutions is not new.  It has been used, for instance,  by
Kharlamov in the study of quartic surfaces
\cite{kharlamov-surf-types}, by Nikulin for all the families of K3
surfaces \cite{nikulin-1979} and by Gross and Harris for curves and
abelian varieties \cite{gross-harris}.  In particular, Nikulin
parametrizes the different connected components of real K3 surfaces by
quotients of products of real hyperbolic spaces by discrete groups
generated by reflections.  Vinberg's algorithm has been used by
Kharlamov \cite{kharlamov-quartic-surfaces} and more recently by
Finashin and Kharlamov \cite{finashin-kharlamov} to study the topology of some real moduli spaces.

Many authors, in addition to the ones already cited, have studied
moduli of real algebraic varieties in terms of Coxeter diagrams or in
terms of the action of complex conjugation on homology.  We would
like to mention the work of Degtyarev, Itenberg and Kharlamov on
Enriques surfaces \cite{Degtyarev-etal} and of Moriceau on nodal
quartic surfaces \cite{moriceau}.  There is considerable literature on
moduli of $n$-tuples in $\R P^1$; see for example
\cite{thurston-shapes-of-polyhedra}, \cite{yoshida-6tuples},
\cite{yoshida-5tuples} and the papers they cite.  Chu's paper
\cite{chu} gives a real hyperbolic cone manifold structure to the
moduli of stable real octuples in $P^1$.

We have two
expository articles \cite{allcock-lecnotes, toledo-lecnotes}
that develop the ideas of this paper less formally, and in the context
of related but simpler moduli problems, for which the moduli space has
dimension${}\leq3$.  The lower dimension means that all the
fundamental domains can be visualized directly.  The
results of this paper were announced in \cite{announcement}.

We would like to thank J\'anos Koll\'ar for helpful discussions at the
early stages of this work.  And we would like to thank the referee for
numerous constructive comments that have greatly improved  our
exposition.

\section{Moduli of complex cubic surfaces}
\label{sec-complex-moduli}

We record here the key constructions and results of our description
\cite{ACT} of the moduli space of smooth complex cubic surfaces as a
quotient of an open dense subset of complex hyperbolic 4-space
$\ch^4$.  Everything we need from that paper appears here.  Only the
results through theorem~\ref{thm-main-theorem-smooth-complex-case} are
required for section~\ref{sec-smoothmoduli} (moduli of smooth real
surfaces).  The last part of this section is needed for
sections~\ref{sec-five-families} and~\ref{sec-Segre} (relations with
classical work).  For background on singular complex cubic surfaces,
see section~\ref{sec-stable-complex-moduli}.

\subsection{Notation}
\label{subsec-definitons-forms}
The key object is the moduli space $\moduli_0$ of smooth cubic
surfaces in $\cp^3$, which we now define.  As in \cite[(2.1)]{ACT},
and in the introduction, let $\forms$ be the space of all nonzero
cubic forms in 4 complex variables, $\D$ the discriminant locus, and
$\forms_0=\forms-\D$ the set of forms defining smooth cubic surfaces.
We take $g\in\GL(4,\C)$ to act in the usual way on the left on $\C^4$,
and on the left on $\forms$ by
\begin{equation}
\label{eq-GL4-action-on-forms}
(g.F)(X)=F(g^{-1}X)
\end{equation}
for $X\in\C^4$.  This is as in
\cite[(2.17)]{ACT}.  The action on $\forms$ is not faithful: the
subgroup acting trivially is $D=\{I,\w I,\wbar I\}$ where $\w=e^{2\pi
  i/3}$ is a primitive cube root of unity fixed throughout the paper.
We write $G$ for the group $\GL(4,\C)/D$ acting effectively, and define
$\moduli_0$ as $G\backslash\forms_0$.  As discussed in
\cite[(2.18) and (3.1)]{ACT}, $G$ acts properly on $\forms_0$, so the
moduli space $\moduli_0$ is a complex-analytic orbifold in a natural
way.  

\subsection{Framed cubic surfaces}
\label{subsec-definition-framed}
The relation between $\moduli_0$ and $\ch^4$ depends on a
(multi-valued) period map whose construction involves cubic threefolds.
Briefly, we first construct the space $\framed_0$ of ``framed smooth
cubic forms'', a certain covering space of $\forms_0$, and then we
define a (single-valued) period map $g:\framed_0\to\ch^4$.  This map is
equivariant with respect to  the deck group of the
cover.  Taking the quotient by this action gives a period
map from $\forms_0$ to a quotient of $\ch^4$, and this map factors
through $\moduli_0$.  It takes some work to define $\framed_0$, so we
begin with that.

If $F\in\forms$ then we write $S$ for the surface it defines in
$\cp^3$ and $T$ for the threefold in $\cp^4$ defined by
\begin{equation}
\label{eq-def-of-T}
Y^3-F(X_0,\dots,X_3)=0.
\end{equation}
Whenever we have a form $F$ in mind, we implicitly define $S$ and $T$
in this way.  $T$ is the 3-fold cyclic covering of $\cp^3$ with
ramification along $S$.  We call it the cyclic cubic threefold
associated to $F$.  We define $\sigma\in\GL(5,\C)$ by
\begin{equation}
\label{eq-def-of-sigma}
\sigma(X_0,\dots,X_3,Y)=(X_0,\dots,X_3,\w Y).
\end{equation}
It generates the deck group of $T$ over $\cp^3$.  All the notation of
this paragraph is from \cite[(2.1)]{ACT}.

Now suppose $F\in\forms_0$.  Then it is easy to see that $T$ is
smooth.  In \cite[(2.2)]{ACT} we show that
$H^3(T,\Z)\isomorphism\Z^{10}$ and that $\sigma$ fixes no element of
this cohomology group (except $0$).  
Because there are no $\sigma$-invariant elements, we may regard
$H^3(T,\Z)$ as a module over the Eisenstein integers $\E:=\Z[\w]$,
with $\w$ acting as $(\sigma^*)^{-1}$.  This gives a free $\E$-module
of rank~$5$, which we call $\Lambda(T)$.  It is a key ingredient in
the rest of the construction.  (In \cite[(2.2)]{ACT} we took $\w$ to
act as $\sigma^*$, but unfortunately this made the period map
antiholomorphic rather than holomorphic, as discussed in the note
added in proof. )

Combining the action of $\sigma$ with the natural symplectic form $\Omega$ on
$H^3(T,\Z)$  gives an $\E$-valued Hermitian form on
$\Lambda(T)$, defined by
\begin{equation}
\label{eq-def-of-h}
\textstyle
h(x,y)=\frac{1}{2}\bigl[\Omega(\theta x,y)+\theta\Omega(x,y)\bigr]
\end{equation}
Here and throughout the paper, $\theta$ represents the
Eisenstein integer $\w-\wbar=\sqrt{-3}$; in particular, the first
$\theta$ in \eqref{eq-def-of-h} is the action of $\theta$ on $\Lambda(T)$, namely
$(\sigma^*)^{-1}-\sigma^*$.  This definition of $h$ is from
\cite[(2.3)]{ACT}, except that the sign is changed
because of the  change of $\E$-module structures.  The fact that $h$ is $\E$-valued, $\E$-linear in its
first coordinate and $\E$-antilinear in its second is part of
\cite[lemma~4.1]{ACT}.  Finally, in \cite[(2.7)]{ACT} we show that
$\Lambda(T)$ is isometric to the lattice $\Lambda:=\E^{4,1}$, meaning
the free module $\E^5$ equipped with the Hermitian form
\begin{equation}
\label{eq-standard-form-on-E4,1}
h(x,y)=-x_0\bar{y}_0+x_1\bar{y}_1+\cdots+x_4\bar{y}_4.
\end{equation}

Even though $\Lambda(T)$ is isometric to $\Lambda$, there is no
preferred isometry, so we must treat them all equally.  So we define a
framing of $F\in\forms_0$ as a projective equivalence class $[i]$ of
$\E$-linear isometries $\Lambda(T)\to\Lambda$; thus $[i]=[i']$ just if
$i$ and $i'$ differ by multiplication by a unit of $\E$.  A framed
smooth cubic form is a pair $(F,[i])$ with $F\in\forms_0$ and $[i]$ a
framing of it, and $\framed_0$ denotes the family of all framed smooth
cubic forms.  Usually we blur the distinction between $i$ and $[i]$;
the main reason for introducing the equivalence relation is so that
the action of $G$ on $\framed_0$ is well-defined (see below).  In
\cite[(3.9)]{ACT} we defined a natural complex manifold structure on
$\framed_0$, for which the obvious projection $\framed_0\to\forms_0$
is a holomorphic  covering map, and we proved the following.

\begin{theorem}[\protect{\cite[(3.9)]{ACT}}]
\label{thm-smooth-framed-space-connected}
$\framed_0$ is connected.
\qed
\end{theorem}

\subsection{Group actions on the space of framed surfaces}
\label{subsec-def-Gamma}
The deck group for $\framed_0\to\forms_0$ is obviously
$\PAut(\Lambda)$.  We will write $\Gamma$ for $\Aut\Lambda$ and  $\PGamma$ for $\PAut(\Lambda)$.   The meaning of $\PGamma$ is the same as in  \cite{ACT}, but the meaning of $\Gamma$ is slightly different.  There, $\Gamma$ was the
linear monodromy group, defined precisely in \cite[(2.11)]{ACT}, and it was proved
in \cite[theorem~2.14]{ACT} that $\Aut\Lambda=\Gamma\times\{\pm I\}$,
so that $\PAut\Lambda=\PGamma$.  Since the the precise linear monodromy group will not be needed here, it will be convenient to use the abbreviation $\Gamma$ for $\Aut\Lambda$. 

As in \cite[(3.9)]{ACT} we write the action of $\PGamma$ on
$\framed_0$ explicitly by
\begin{equation}
\label{eq-def-of-PGamma-action-on-framed}
\gamma.\bigl(F,[i]\bigr)=\bigl(F,[\gamma\circ i]\bigr),
\end{equation}
and define an action of $G$ on $\framed_0$ as follows.   Any 
$h\in\GL(4,\C)$ acts on $\C^5$ by 
\begin{equation}
\label{eq-GL4-action-on-C5}
h(X_0,\dots,X_3,Y)=(h(X_0,\dots,X_3),Y).
\end{equation}
If $(F,[i])\in\framed_0$ then $h$ carries the points of the threefold
$T_F$ defined by \eqref{eq-def-of-T} to those of the one $T_{hF}$ defined by the
same formula with $hF$ replacing $F$.  So $h$ induces an isometry
$h^*:\Lambda(T_{hF})\to\Lambda(T_F)$.  Therefore $i\circ h^*$ is an
isometry $\Lambda(T_{hF})\to\Lambda$, so $\bigl(hF,[i\circ
h^*]\bigr)\in\framed_0$.  We may therefore define an action of $\GL(4,\C)$ on
$\framed_0$ by
\begin{equation}
\label{eq-def-of-G-action-on-framed}
h.\bigl(F,[i]\bigr)
:=
\bigl(h.F,[i\circ h^*]\bigr)
=
\bigl(F\circ h^{-1},[i\circ h^*]\bigr).
\end{equation}
This action factors through $G$ because $\w I\in\GL(4,\C)$ fixes every
$F$ and acts on $\cp^4$ in the same way as a power of $\sigma$.  That
is, $\w I$ sends every $\Lambda(T)$ to itself by a scalar, so it fixes
every framing.

As in \cite[(2.18)]{ACT} we define the moduli space of framed cubic
surfaces $\moduli_0^f$ as $G\backslash\framed_0$.  It is an analytic
space because $G$ acts properly on $\framed_0$ (since it does on
$\forms_0$).  But more is true:

\begin{theorem}[\protect{\cite[lemma~3.14]{ACT}}]
\label{thm-G-action-on-framed-is-free}
The action of $G$ on $\framed_0$ is free, so $\moduli_0^f$ is a
complex manifold
\qed
\end{theorem}

\subsection{The period map}
\label{subsec-def-period-map}

Having described $\framed_0$, we now describe the period map
$\framed_0\to\ch^4$, following \cite[(3.11)]{ACT}.  
Using the Griffiths residue
calculus, one can work out the Hodge numbers of $T$, which turn out to
be $h^{3,0}=h^{0,3}=0$, $h^{2,1}=h^{1,2}=5$.  This calculation
\cite[lemma~2.6]{ACT} also gives the refinement of the Hodge
decomposition by the eigenspace decomposition
$$
H^3(T,\C)=H^3_{\sigma=\w}(T,\C)\oplus H^3_{\sigma=\wbar}(T,\C)
$$
under $\sigma$, namely $h^{2,1}_\wbar=h^{1,2}_\w=1$,
$h^{1,2}_\wbar=h^{2,1}_\w=4$.  

Given a  framing $[i]$, we may obtain a point of $\ch^4$ as follows.
Define $i_*:H^3_{\wbar}(T,\C)\to\C^{4,1}$ as the composition
\begin{equation}
\label{eq-def-of-i-star}
H^3_{\wbar}(T,\C)
\isomorphism
H^3(T,\R)
=\Lambda(T)\tensor_\E\C
\mathop{\longrightarrow}_{i\tensor1}
\Lambda\tensor_\E\C
=
\C^{4,1}.
\end{equation}
Here the leftmost isomorphism is the eigenspace projection
$H^3(T,\R)\to H^3_{\wbar}(T,\C)$, which is an isomorphism of complex
vector spaces.  This statement has meaning because we have defined
$\w$ to act on $H^3(T,\Z)$ as $(\sigma^*)^{-1}$.  The map is
$\C$-linear because the actions of $\wbar\in\C$ and $\sigma^*$ agree
on the domain (by the definition of the $\w$-action as
$(\sigma^*)^{-1}$) and on the target (since it is the
$\wbar$-eigenspace of $\sigma^*$).  Since the eigenspace projection is
obviously $\sigma$-equivariant, it is also $\wbar$-equivariant, i.e.,
$\C$-linear.

Our model for complex hyperbolic space
$\ch^4$ is the set of negative lines in
$\C^{4,1}:=\Lambda\tensor_\E\C$.  
It follows from the Riemann bilinear relations that
$i_*\bigl(H^{2,1}_\wbar(T,\C)\bigr)\in P(\C^{4,1})$ is a
negative-definite line, i.e., a point of $\ch^4$.  See
\cite[lemmas~2.5--2.6]{ACT}, and note that the map called $Z$ there is
the eigenspace projection.  The period map $g:\framed_0\to\ch^4$ is
then defined by
\begin{equation}
\label{eq-def-of-period-map-on-framed-smooth}
g(F,[i])
=
i_*\bigl(H^{2,1}_\wbar(T,\C)\bigr)
\in
\ch^4.
\end{equation}
It is holomorphic since the Hodge filtration
varies holomorphically \cite[(2.16)]{ACT}.  


The period map
$\framed_0\to\ch^4$ factors through $\moduli_0^f$ because $\ch^4$ is a
complex ball and bounded holomorphic functions on $G$ are constant.
So we may also regard $g$ as a map
\begin{equation}
\label{eq-def-of-perdiod-map-on-framed-moduli}
g:\moduli_0^f=G\backslash\framed_0\to\ch^4.
\end{equation}
This is $\PGamma$-equivariant, so it descends to another map
\begin{equation}
\label{eq-def-of-period-map-on-unframed-moduli}
g
:
\moduli_0
=
G\backslash\forms_0
=
(G\times\PGamma)\backslash\framed_0
\to
\PGamma\backslash\moduli_0^f
\to
\PGamma\backslash\ch^4,
\end{equation}
also called ``the period map''.

\subsection{The main theorem of \cite{ACT} in the smooth case}
\label{subsec-main-theorem-smooth-complex-case}
It turns out that the period map $\framed_0\to\ch^4$ is not quite
surjective.  To describe the image, let $\H\sset\ch^4$ be the union of
the orthogonal complements of the norm~$1$ vectors in $\Lambda$.  It
turns out that points of $\H$ represent singular cubic surfaces; see
section~\ref{sec-stable-complex-moduli}.  We will need the following
combinatorial result about $\H$, as well as the smooth case of the
main theorem of \cite{ACT}.

\begin{lemma}[\protect{\cite[(7.29)]{ACT}}]
\label{lem-hyperplanes-are-orthogonal}
Any two components of $\H$ that meet are orthogonal along their intersection. 
\qed
\end{lemma}

\begin{theorem}[\protect{\cite[Theorem~2.20]{ACT}}]
\label{thm-main-theorem-smooth-complex-case}
The period map $g$ sends $\moduli_0^f=G\backslash\framed_0$
isomorphically to $\ch^4-\H$.  In particular, $g$ has everywhere rank $4$ on $\moduli_0^f$.  Moreover,  the induced map
$\moduli_0\to\PGamma\backslash(\ch^4-\H)$ is an isomorphism of complex
analytic orbifolds.
\qed
\end{theorem}

\subsection{Standard model for $H^2(S)$; vector spaces over $\F_3$}
\label{subsec-finite-vectorspace}
This material may be skipped until needed in sections~\ref{sec-five-families}
and~\ref{sec-Segre}. 
We write $L$ for the lattice $\Z^{1,6}$, whose bilinear form is
\begin{equation}
x\cdot y = x_0y_0 - x_1y_1 - \cdots -x_6y_6,
\end{equation}
and write $\eta$ for $(3, -1,\cdots ,-1)\in L$.  It is standard that
$L(S):=H^2(S;\Z)$ is isometric to $L$ by an isometry identifying the
hyperplane class $\eta(S)$ with $\eta$.  (See \cite[Thm~23.8]{manin}; our
expression for $\eta$ corrects a sign error in \cite[(3.2)]{ACT}.)
Also, the isometry group of $(L,\eta)$ is the Weyl group $W(E_6)$,
generated by the reflections in the norm $-2$ vectors of
$L_0:=\eta^\perp$ (see \cite[Theorem~23.9]{manin}).  
We define $L_0(S)$ as the corresponding sublattice
of $L(S)$, namely the primitive cohomology $\eta(S)^\perp\sset L(S)$.  Our
notations $L, \eta, L(S), \eta(S)$ are from \cite[(3.2)]{ACT} and
$L_0, L_0(S)$ are from \cite[(4.8)]{ACT}.

In \cite[(4.10)]{ACT} we found a special relationship between
$\Lambda(T)$ and $L(S)$ that is needed in sections
sections~\ref{sec-five-families} and~\ref{sec-Segre}.  There is no
natural map between them, in either direction.  But there is a natural
isomorphism between certain $\F_3$-vector spaces associated to them.
In our explanation we will identify $H^*$ and $H_*$ by Poincar\'e
duality

Suppose we have a primitive 2-cycle $c$ on $S$.  Since $S\subset T$
and $T$ has no primitive cohomology, there is a 3-chain $d$ in $T$
bounding it.  The chain $\sigma_*(d) -\sigma_*^{-1}(d)$ is a 3-cycle
on $T$, whose reduction modulo $\theta$ depends only on $c$ and whose  homology class  modulo $\theta$ depends only on the homology class of $c$.  So we
have a natural map $L_0(S)\to\Lambda(T)/\theta\Lambda(T)$.  
It turns out that the kernel is $3L_0'(S)$, where the prime denotes
the dual lattice.
The result
is a natural isomorphism from $V(S):=L_0(S)/3L_0'(S)$ to
$V(T):=\Lambda(T)/\theta\Lambda(T)$, both five-dimensional vector
spaces over $\F_3$.   We will write $V$ for $\Lambda/\theta\Lambda$. 

Also, reducing inner products in $L_0'(S)$ modulo $3$ gives a
symmetric bilinear form $q$ on $V(S)$, and similarly, reducing inner
products in $\Lambda(T)$ modulo $\theta$ gives one on $V(T)$.  We have
no special symbol for the latter because it is essentially the same as
$q$: 

\begin{lemma}[\protect{\cite[(4.10)]{ACT}}]
\label{lem-isometry-finite-spaces}
The map $V(S)\to V(T)$ just defined is an isometry.
\qed
\end{lemma}

A consequence is that the monodromy action of $\pi_1(\forms_0,F)$ on
$V(S)$ is the same as on $V(T)$.  Since $PO(V)\isomorphism W(E_6)$,
the classical monodromy map
$\pi_1(\forms_0,F)\to\Aut\bigl(L(S),\eta(S)\bigr)\isomorphism W(E_6)$ can be
recovered from our monodromy representation 
$\pi_1(\forms_0,F)\to\PGamma\to\PO\bigl(V(T)\bigr)$.

\section{Moduli of smooth real cubic surfaces}
\label{sec-smoothmoduli}

The purpose of this section is to prove those results on moduli of
smooth real cubic surfaces which follow more or less automatically
from the general results on moduli of complex cubic surfaces that we
proved in \cite{ACT} and summarized in section~\ref{sec-complex-moduli}.  

\subsection{Anti-involutions and real structures}
We write $\kappa$ for the standard
complex conjugation map on $\C^4$, and also for the induced map
on $\forms$ given by 
\begin{equation}
\label{eq-complex-conjugation-on-forms}
(\kappa.F)(x)=\overline{F(\kappa^{-1}(x))}=\overline{F(\kappa x)}.
\end{equation}
In coordinates this amounts to replacing the coefficients of $F$ by
their complex conjugates.  This action of $\kappa$ on functions
carries holomorphic functions to holomorphic functions (rather than
anti-holomorphic ones).  Similarly, if $\alpha$ is an anti-holomorphic
map of a complex variety $V_1$ to another $V_2$, then
$\overline{\alpha^*}:H^*(V_2;\C)\to H^*(V_1;\C)$ is defined as the
usual pullback under $V_1\to V_2$ followed by complex conjugation in
$H^*(V_1;\C)$.  If $V_1$ and $V_2$ are compact K\"ahler manifolds,
then $\overline{\alpha^*}$ is an antilinear map that preserves the
Hodge decomposition.  Many different complex conjugation maps appear
in this paper, so we call a self-map of a complex manifold
(resp. complex vector space or $\E$-module) an {\it anti-involution} if it
is anti-holomorphic (resp. anti-linear) and has order~$2$.   We also use the term {\it real structure} for an anti-involution.   By the {\it real locus} of a real structure we mean the fixed point set of the corresponding anti-involution.

\subsection{Notation}
We write $\forms^\R$, $\D^\R$ and $\forms_0^\R$ for the subsets of the
corresponding spaces of \S~\ref{subsec-definitons-forms} whose members
have real coefficients.  Note that it is possible for the zero locus
in $\rp^3$ of $F\in\D^\R$ to be a smooth manifold, for example it
might have two complex-conjugate singularities.  We write $G^\R$ for
the group $\GL(4,\R)$, which is isomorphic to the group of real points
of $G$, and we write $\moduli_0^\R$ for the space
$G^\R\backslash\forms_0^\R$.  This is a real-analytic orbifold in the
sense that it is locally the quotient of a real analytic manifold by a
real analytic action of a finite group.
There is a natural map $\moduli_0^\R\to
\moduli_0$ which is finite-to-one and generically
injective, but not injective (since a cubic surface may have several
inequivalent real structures).  So $\moduli_0^\R$ is not quite the same
as the real locus of $\moduli_0$.  


\subsection{Framed smooth real cubic surfaces; anti-involutions of $\Lambda$}
\label{subsec-framed-smooth-real-cubic-surfaces-anti-involutions}
We write $\framed_0^{\,\R}$ for the preimage of $\forms_0^\R$ in the space 
$\framed_0$ of \S\ref{subsec-definition-framed}.  This is not the real locus of any real structure on
$\framed_0$, but rather the union of the real loci of many different
real structures.  We consider these many different real loci
simultaneously because no one of them is distinguished.  They are all
lifts of $\kappa:\forms_0\to\forms_0$.  To develop this idea, let $\A$
denote the set of anti-involutions of $\Lambda$ and let $P\A$ denote
the set of their projective equivalence classes.  If
$(F,[i])\in\framed_0^{\,\R}$ then $\kappa$ acts on $\Lambda(T)$ as an
anti-involution, so $\x = i\circ\kappa^*\circ i^{-1}$ lies in $\A$.
Because of the ambiguity in the choice of representative $i$ for
$[i]$, $\x$ is not determined by $[i]$; however, its class $[\x]$ in
$P\A$ is well-defined.  Clearly $[\x]$ does not change if $(F,[i])$
varies in a connected component of $\framed_0^{\,\R}$.  
Thus we get a map
$\pi_0(\framed_0^{\,\R})\to P\A$.  

The ``many different real loci'' we referred to are the following
subspaces $\framed_0^{\,\x}$ of $\framed_0$, one for each $[\x]\in P\A$: 
\begin{equation}
\label{eq-def-framed-xi}
\framed_0^{\,\x} = \bigl\{(F,[i])\in\framed_0^{\,\R}:[ i\circ\kappa^*\circ i^{-1}] = [\x]\bigr\}\;.
\end{equation}
Here and in many other places we omit the brackets of $[\x]$ to
simplify the notation.  The various $\framed_0^{\,\x}$ cover
$\framed_0^{\,\R}$, because any $(F,[i])\in\framed_0^{\,\R}$ lies in
$\framed_0^{\,\x}$ with $\x=i\circ\kappa^*\circ i^{-1}$.  We will see
that each $\framed_0^{\,\x}$ is nonempty.
One can check that the lift of $\kappa$ to $\framed_0$ that fixes
$(F,[i])$ has fixed-point set equal to $\framed_0^\x$.  In fact we
give a formula for the action of $\x$ on $\framed_0$ in equation \eqref{eq-PGamma-prime-action-on-framed-smooth-cubics} below.

Similarly, if $\x\in P\A$ then we define $H_{\x}^4$ as its fixed-point
set in $\ch^4\sset P(\Lambda\tensor_\E\C)$.  The notation reflects the
fact that $H^4_{\x}$ is a copy of real hyperbolic $4$-space.  Just as
for $\framed_0$, there is no natural choice of lift $\ch^4\to\ch^4$ of
the action of $\kappa$ on
$\moduli_0\isomorphism\PGamma\backslash\ch^4$.  So we consider all the
$\x\in P\A$ simultaneously, and the various $H^4_\x$'s are the real
loci of the various real structures $\x$ on $\ch^4$.

\subsection{The real period map $g^\R$}
We need the following lemma in order to define the real period map.

\begin{lemma}
\label{lem-gR-takes-values-in-K0}
In the notation of
\S\ref{subsec-framed-smooth-real-cubic-surfaces-anti-involutions},
$g\bigl(\framed_0^{\,\x}\bigr)\subset H_{\x}^4$.
\end{lemma}

\begin{proof}
The key is that  $\overline{\kappa^*}$ 
 is an antilinear map of $H^3(T;\C)$ which preserves the Hodge
decomposition  and each eigenspace of $\sigma$.  Therefore it 
preserves the inclusion
$i_*\bigl(H_\wbar^{2,1}(T)\bigr)\to H^{2,1}_\wbar(T)$.
The lemma is a  formal consequence of this and the relation
$\x=i\circ\kappa^*\circ i^{-1}$.
\end{proof}

We define the real period map $g^\R:\framed_0^{\,\R}\to\ch^4\times P\A$ by
\begin{equation}
\label{eq-def-of-real-period-map-on-smooth-real-framed}
g^\R(F,[i])=\bigl(g(F,[i])\,,\,[i\circ\kappa^*\circ i^{-1}]\bigr)\;.
\end{equation}
The previous lemma asserts that $g(F,[i])\in H^4_{\x}$, so 
$g^\R(F,[i])$ is a point of $\ch^4$ together with an anti-involution
fixing it.  Therefore
$g^\R$
can be regarded as a map $\framed_0^{\,\R}\to\coprod_{\x\in
P\A}H^4_{\x}$.  The next lemma shows that $g^\R$ descends to a  map
$G^\R\backslash\framed_0^{\,\R}\to\coprod_{\x\in P\A}H^4_{\x}$.

\begin{lemma}
\label{lem-real-period-map-is-G-R-invariant}
The real period map $g^\R:\framed_0^{\,\R}\to\coprod_{\x\in
  P\A}H^4_{\x}$ is constant on $G^\R$-orbits. 
\end{lemma}

\begin{proof}
We must show for 
$(F,[i])\in\framed_0^{\,\R}$ and $h\in G^\R$ that
$g^\R\bigl(h.(F,[i])\bigr)=g^\R(F,[i])$. 
We have
\begin{align*}
g^\R\bigl(h.(F,[i])\bigr)
&{}=
g^\R\bigl(h.F,[i\circ h^*]\bigr)\\
&{}=
\Bigl(g\bigl(h.F,[i\circ h^*]\bigr),
\bigl[i\circ h^*\circ\kappa^*\circ(h^*)^{-1}\circ i^{-1}\bigr]\Bigr)\\
&{}=
\Bigl(g\bigl(h.(F,[i])\bigr),[i\circ\kappa^*\circ i^{-1}]\Bigr)\\
&{}=
\bigl(g(F,[i]),[i\circ\kappa^*\circ i^{-1}])\\
&{}=
g^\R(F,[i]).
\end{align*}
Here the first line uses the definition \eqref{eq-def-of-G-action-on-framed} of $G$'s action on
$\framed_0$, the second the definition \eqref{eq-def-of-real-period-map-on-smooth-real-framed} of $g^\R$, the third
the fact that $h$ and $\kappa$ commute, and the fourth the fact that
the complex period map $g$ is $G$-invariant.
\end{proof}

\subsection{The main theorem for smooth real surfaces}
We know that $g^\R$ cannot map $\framed_0^{\,\R}$ onto all of
$\coprod_{\x\in P\A}H^4_{\x}$, because $g(\framed_0)$ misses the
hyperplane arrangement $\H$.  Therefore we define $K_0=\coprod_{\x\in
  P\A}\bigl(H^4_{\x}-\H\bigr)$.  Now we can state the main theorem of
this section.

\begin{theorem}
\label{thm-isomorphism-for-smooth}
The real period map $g^\R$ descends to a $\PGamma$-equivariant
real-analytic diffeomorphism $G^\R\backslash\framed_0^{\,\R}\to K_0=\coprod_{\x\in
  P\A}\bigl(H^4_{\x}-\H\bigr)$.  Thus
$\framed_0^{\,\R}$ is a principal $G^\R$-bundle over $K_0$.  Taking the
quotient by $\PGamma$ yields a real-analytic orbifold isomorphism
$$
\moduli_0^\R=(\PGamma\times G^\R)\backslash\framed_0^{\,\R}\to \PGamma\backslash
K_0\;. 
$$
Equivalently, we have an orbifold isomorphism 
$$
\moduli_0^\R\cong
\coprod_{\x} \PGamma_\x^\R\backslash(H_\x - \H),
$$ where $\x$ now
ranges over a set of representatives for the set $C\A$ of 
$P\Gamma$-conjugacy
classes of elements of $P\A$, and $\PGamma_\x^\R$ is the
$\PGamma$-stabilizer of $H^4_\x$.
\end{theorem}

To prove the theorem we extend some of the
constructions of section~\ref{sec-complex-moduli} to include antiholomorphic
transformations.  These notions will not be needed later in the
paper.  First, let $\GL(4,\C)'$ be the group of all linear and
antilinear automorphisms of $\C^4$.  We regard it as also acting on
$\C^5$, with an element $h$ acting by \eqref{eq-GL4-action-on-C5} if $h$ is linear and by
$$
h(X_0,\dots,X_3,Y)=(h(X_0,\dots,X_3),\bar Y)
$$
if $h$ is antilinear.  If $h$ is linear then it acts on $\forms$ as in
\eqref{eq-GL4-action-on-forms}, and if $h$ is antilinear then we define
$$
(h.F)(X_0,\dots,X_3)=\overline{F(h^{-1}(X_0,\dots,X_3))}.
$$
This is consistent with our definition
\eqref{eq-complex-conjugation-on-forms} of the action of $\kappa$.

We let $\framed_0'$ be the space of all pairs $(F,[i])$ where
$F\in\forms_0$, $i:\Lambda(T)\to \Lambda$ is either a linear or
antilinear isometry, and $[i]$ is its projective equivalence class.
$\framed_0'$ is a disjoint union of two copies of $\framed_0$.  Since
$\framed_0$ is connected
(theorem~\ref{thm-smooth-framed-space-connected}), $\framed_0'$ has 2
components.  Formula 
\eqref{eq-def-of-G-action-on-framed} now defines an action of
$\GL(4,\C)'$ on $\framed_0'$.  We also 
let $\Gamma'$ be the group of all
linear and antilinear isometries of $\Lambda$, and  
observe that \eqref{eq-def-of-PGamma-action-on-framed} defines an
action of it on $\framed_0'$. 
The antilinear elements in each
group exchange the two components of $\framed_0'$.  The subgroup
$D=\{I,\w I,\w^2I\}$ of $\GL(4,\C)'$ acts trivially, inducing an
action of the quotient group, which we call $G'$.  The scalars in
$\Gamma'$ also act trivially, inducing an action of the quotient
group, which we call $\PGamma'$.  Each of $G'$ and $\PGamma'$ acts
freely on $\framed_0'$, because $G$ and $\PGamma$ act freely on
$\framed_0$ (theorem~\ref{thm-G-action-on-framed-is-free}).  

The following two lemmas are generalities about group actions that we
will need in the proof of theorem~\ref{thm-isomorphism-for-smooth}.

\begin{lemma}
\label{lem-general-principle-used-for-surjectivity}
Let $Y$ be a set and suppose $L$ and $M$ are groups with commuting
free actions on it.  Suppose $y\in Y$ has images $m\in L\backslash Y$
and $l\in M\backslash Y$.  For any $\x\in L$ preserving $l$, there
exists a unique $\hat\x\in M$ such that $(\x\hat\x).y=y$.
Furthermore, the map $\x\mapsto\hat\x$ defines an isomorphism from the
stabilizer $L_l$ of $l$ to the stabilizer $M_m$ of $m$.
\qed
\end{lemma}

\begin{lemma}
\label{lem-general-principle-used-for-injectivity}
If a group $G$ acts freely on a set $X$,  $\phi$
is a transformation of $X$ normalizing $G$, and $Z$ is the centralizer
of $\phi$ in $G$, then the natural map $Z\backslash X^\phi\to
G\backslash X$ is
injective.
\end{lemma}

\begin{proof}
If $h\in G$ carries $x\in X^\phi$ to $y\in X^\phi$
then so does $\phi^{-1}h\phi$, so $\phi^{-1}h\phi=h$ by freeness, so
$h\in Z$.
\end{proof}

\begin{proof}[Proof of theorem~\ref{thm-isomorphism-for-smooth}.]
  First observe that $g^\R: G^\R \backslash\framed_0^{\,\R}\to K_0$ is a
  local diffeomorphism.  This follows immediately from the fact that
  the rank of $g^\R$ is the same as that of $g$, which is $4$
  everywhere in $G\backslash\framed_0$ by
  Theorem~\ref{thm-main-theorem-smooth-complex-case}.

To prove surjectivity, suppose $\x\in P\A$ and $x\in H^4_\x-\H$.  We
must exhibit a point of $\framed_0^{\,\R}$ mapping to $(x,[\x])$ under $g^\R$.
First, by the surjectivity of the complex period map, there exists
$(F,[i])\in\framed_0$ with $g(F,[i])=x$.  Now we apply lemma~\ref{lem-general-principle-used-for-surjectivity}
with $Y=\framed_0'$, $y=(F,[i])$, $L=\PGamma'$ and $M=G'$.  Our choice
of $y$ gives 
$$
l
\,=\,
x
\,\in\,
\ch^4\!-\!\H
\,=\,
G'\backslash\framed_0'
\,=\,
M\backslash Y.
$$
By hypothesis, $\x$ is an anti-involution in $\PGamma'$ fixing $x$.
By lemma~\ref{lem-general-principle-used-for-surjectivity} there exists $\hat\x\in G'$, of order~2, with
$\x\hat\x$ fixing $(F,[i])$.  Since $\x$ swaps the components of
$\framed_0'$, $\hat\x$ does too, so $\hat\x$ is antiholomorphic.  We
have constructed a complex cubic surface $\{F=0\}$ preserved by an anti-involution
$\hat\x\in G'$.  Now we will verify that it (or rather a translate of it in
$\forms_0^\R$) maps to $x$ under $g^\R$.

Since $\GL(4,\C)'\to G'$ has kernel $\Z/3$, there is an anti-involution
$\alpha\in\GL(4,\C)'$ lying over $\hat\x$.  (In fact all 3 elements
lying over $\hat\x$ are anti-involutions.)  
The fact that $\x\hat\x$ fixes $(F,[i])$
implies
$\alpha.F=F$ and $[\x\circ
i\circ\alpha^*]=[i]$, i.e., $[i\circ\alpha^*\circ i^{-1}]=\x^{-1}=\x$.
Because all real structures on a complex vector space are equivalent,
$\alpha$ is conjugate to $\kappa$, that is, there exists $h\in\GL(4,\C)$
with $\alpha=h^{-1}\kappa h$.  We claim $h.(F,[i])$ lies in $\framed_0^{\,\R}$
and maps to $(x,\x)$ under $g^\R$.  That it lies in $\framed_0^{\,\R}$ is
just the claim $h.F\in\forms_0^\R$; here is the verification:
$$
\kappa.hF
=
hh^{-1}\kappa h.F
=
h\alpha.F
=
h.\alpha F
=
hF.
$$
And finally:
\begin{align*}
g^\R\bigl(h.(F,[i])\bigr)
&{}=
g^\R\bigl(h.F,[i\circ h^*]\bigr)\\
&{}=
\Bigl(g(h.F),\bigl[i\circ h^*\circ\kappa^*\circ(h^*)^{-1}\circ i^{-1}
\bigr]\Bigr)\\
&{}=
\Bigl(g(F),\bigl[i\circ(h^{-1}\circ\kappa\circ h)^*\circ
i^{-1}\bigr]\Bigr)\\
&{}=
\bigl(x,[i\circ\alpha^*\circ i^{-1}]\bigr)\\
&{}=
(x,\x).
\end{align*}
This finishes the proof of surjectivity.

To prove injectivity it suffices to show that
$g^\R:G^\R\backslash\framed_0^{\,\x}\to G\backslash\framed_0=\ch^4-\H$
is injective for each $\x\in P\A$.  This also follows from a general
principle, best expressed by regarding $\framed_0^{\,\x}$ as the
fixed-point set of $\x$ in $\framed_0$.  We have formulated an action
of $\PGamma'$ on $\framed_0'$, but we can regard it as acting on
$\framed_0$ by identifying $\framed_0$ with
$\langle\kappa\rangle\backslash\framed_0'$.  The subgroup $\PGamma$ acts by
\eqref{eq-def-of-PGamma-action-on-framed} as before, but an anti-linear
$\gamma\in \PGamma'$ now acts by
\begin{equation}
\label{eq-PGamma-prime-action-on-framed-smooth-cubics}
\gamma.(F,[i])=(\kappa.F,[\gamma\circ i\circ\kappa^*])\;.
\end{equation}
It follows from these definitions that $\framed_0^{\,\x}$ is the
fixed-point set of $\x$.  We apply
lemma~\ref{lem-general-principle-used-for-injectivity} with
$X=\framed_0$, $G=G$ and $\phi=\x$; then $X^\phi=\framed_0^{\,\x}$ and
$Z=G^\R$.  The conclusion is that
$G^\R\backslash\framed_0^{\,\x}\to G\backslash\framed_0=\ch^4-\H$ is injective.  This
concludes the proof of the first statement of the theorem.  The
remaining statements follow.

\end{proof}

\section{The five families of real cubics}
\label{sec-five-families}

Theorem~\ref{thm-isomorphism-for-smooth} described $\moduli_0^\R$ in
terms of the $H^4_{\x}$, where $[\x]$ varies over a complete set of
representatives of $C\A$, the set of $P\Gamma$-conjugacy classes in
the set $P\A$ of projective equivalence classes of anti-involutions of
$\Lambda$.  In this section we find such a set of representatives.
That is, we classify the $\x$ up to conjugacy by
$\Gamma$; there are exactly 10 classes, and we give a recognition
principle which allows one to easily compute which class contains a
given anti-involution. In fact there are only five classes up to sign,
so  $C\A$ has 5 elements, and there are 5 orbits of
$H^4_{\x}$'s under $P\Gamma$.  Unlike in the rest of the paper, in
this section we will be careful to distinguish between an
anti-involution $\x$ of $\Lambda$ and its projective equivalence class
$[\x]$.

\subsection{Classification of anti-involutions of $\Lambda$}
Using the coordinate system
\eqref{eq-standard-form-on-E4,1}, we define the following five
anti-involutions of $\Lambda$:
\renewcommand\-{\phantom{-}}
\begin{equation}
\label{eq-anti-invols-of-Lambda}
\begin{split}
\x_0:(x_0,x_1,x_2,x_3,x_4)&\mapsto
        (\bar{x}_0,\-\bar{x}_1,\-\bar{x}_2,\-\bar{x}_3,\-\bar{x}_4)\\
\x_1:(x_0,x_1,x_2,x_3,x_4)&\mapsto
        (\bar{x}_0,\-\bar{x}_1,\-\bar{x}_2,\-\bar{x}_3,-\bar{x}_4)\\
\x_2:(x_0,x_1,x_2,x_3,x_4)&\mapsto
        (\bar{x}_0,\-\bar{x}_1,\-\bar{x}_2,-\bar{x}_3,-\bar{x}_4)\\
\x_3:(x_0,x_1,x_2,x_3,x_4)&\mapsto
        (\bar{x}_0,\-\bar{x}_1,-\bar{x}_2,-\bar{x}_3,-\bar{x}_4)\\
\x_4:(x_0,x_1,x_2,x_3,x_4)&\mapsto
        (\bar{x}_0,-\bar{x}_1,-\bar{x}_2,-\bar{x}_3,-\bar{x}_4)\;.
\end{split}
\end{equation}
\let\-\originaldash
The subscript indicates how many of the
coordinates are replaced by the negatives of their complex
conjugates rather than just their conjugates.  

In order to distinguish their conjugacy classes we will use the
$5$-dimensional vector space $V = \Lambda/\theta\Lambda$ over the
field $\F_3=\E/\theta\E$,
and its quadratic form $q$, the reduction of the Hermitian form
\eqref{eq-standard-form-on-E4,1}.   These were defined in
\S\ref{subsec-finite-vectorspace}.  The dimensions of
$\x$'s eigenspaces and the determinants of $q$'s restrictions to them
are conjugacy invariants of $\x$ (the determinants lie in
$\F_3^*/(\F_3^*)^2=\{\pm1\}$).  We use the abbreviation {\it negated
  space} for the $(-1)$-eigenspace of $\x$.

\begin{theorem}
\label{thm-classification-of-anti-involutions}
An anti-involution of $\Lambda$ is $\Gamma$-conjugate to
exactly one of the $\pm\x_j$.  Two anti-involutions of $\Lambda$ are conjugate if and only if
the restrictions of $q$ to the two fixed spaces in $V$ (or to the two
negated spaces) have the same dimension and determinant.
\end{theorem}

\begin{caution}
  The obvious analogue of the theorem fails for some other $\E^{n,1}$,
  for example $n=3$.
\end{caution}



\begin{proof}
  It is classical that $P\forms_0^\R$ has 5 connected components
  \cite[\S 24] {segre}.  Because $-1$ lies in the identity
  component of $G^\R$, it follows that $\forms_0^\R$ itself has 5
  components, and thence that $\moduli_0^\R$ has at most 5 components.
  The surjectivity part of theorem~\ref{thm-isomorphism-for-smooth}
  implies that for every $\x\in\A$ there exists $F\in\forms_0^\R$ such
  that $\bigl(\Lambda(T),[\kappa^*]\bigr)\isomorphism(\Lambda,[\x])$.
  Therefore the number of components of $\moduli_0^\R$ is at least the
  cardinality of $C\A$, so $|C\A|\leq5$.  Also, the elements of $[\x]$
  are $\x\cdot(-\w)^i$, $i=0,\dots,5$, and these fall into at most two
  conjugacy classes (proof: conjugate by scalars).  Therefore there
  are at most 10 classes of anti-involutions of $\Lambda$.

Now we exhibit 10 distinct classes.
  It is easy to check that $\x_j$ has negated (resp. fixed) space of
  dimension $j$ (resp. $5-j$) and the restriction of $q$ to it has
  determinant $+1$ (resp. $-1$).  For $-\x_j$, the negated and fixed
  spaces are reversed.  Therefore $\pm\x_0,\dots,\pm\x_4$ all lie in
  distinct conjugacy classes.  
Since we have exhibited 10
  classes, they must be a complete set of representatives, justifying
  the first part of the theorem.  In distinguishing them, we also
  proved the second part.
\end{proof}

Let 
 $H^4_j$ be the fixed-point set of $\x_j$ in $\ch^4$, and let 
$\PGamma^\R_j$ be the stabilizer of $H^4 _j$ in $\PGamma$.  
We have  
the following  improvement on theorem~\ref{thm-isomorphism-for-smooth}.  

\begin{corollary}
\label{cor-componentscorollary}
The set $C\A$ has cardinality $5$ and is represented by
$\x_0\dots,\x_4$ of \eqref{eq-anti-invols-of-Lambda}.  We have an
isomorphism $\moduli_0^\R = \coprod_{j=0}^4
\PGamma_j^\R\backslash(H_j^4 - \H)$ of real analytic orbifolds.  
 For each $j$, $\PGamma_j^\R\backslash(H^4_j-\H)$ is connected.

 \qed
\end{corollary}

\subsection{The classical labeling of the five components}
\label{subsec-classical-parametrization-of-5-components}
The classical labeling of the 5 types of real cubic surface was
in terms of the topology of the real locus of $S$, or the action of
complex conjugation $\kappa$ on the 27 lines of $S$, or the action of
$\kappa$ on $H^2(S)$.  We will develop enough of this to establish the
correspondence between the 5 types of surface and our
$\PGamma_j^\R\backslash(H^4_j-\H)$.

Recall from \S\ref{subsec-finite-vectorspace} the lattice
$L(S)=H^2(S;\Z)$, the hyperplane class $\eta(S)$, the primitive
cohomology $L_0(S)$, and their ``standard models'' $L$, $\eta$, $L_0$.
As stated there, the isometries of $L(S)$ preserving $\eta(S)$ form a
copy of the Weyl group $W=W(E_6)=\Aut(L,\eta)$, which is generated by
the reflections in the roots (norm~$-2$ vectors) of $L_0$.

Since $\kappa$ is antiholomorphic, it negates $\eta(S)$ and hence acts
on $L(S)$ by the product of $-I$ and some element $g$ of
$\Aut(L(S),\eta(S))$ of order~$1$ or~$2$. Therefore, to classify the
possible actions of $\kappa$ on $L(S)$ we will enumerate the
involutions of $W$ up to conjugacy.  According to \cite[p.~27]{ATLAS} or \cite[Table~1]{manin}, 
there are exactly four conjugacy classes of involutions.  Each class
may be constructed as the product of the reflections in $1\leq j\leq
4$ mutually orthogonal roots. To make this explicit we choose four
distinct commuting reflections $R_1,\dots,R_4$ in $W$.

We write $\forms_{0,0}^\R,\dots,\forms_{0,4}^\R$ for the set of those
$F\in\forms_0^\R$ for which $(L(S),\breakok\eta(S),\breakok\kappa^*)$ is equivalent to
$(L,\eta,-g)$ for $g=I$, $R_1$, $R_1R_2$, $R_1R_2R_3$, $R_1R_2R_3R_4$.
The $j$ in $\forms^\R_{0,j}$ is the number of $R$'s involved.  By the
previous paragraph, the $\forms_{0,j}^\R$ are disjoint and cover
$\forms_0^\R$.  We will write $\moduli^\R_{0,j}$ for
$G^\R\backslash\forms^\R_{0,j}$.

Now we relate the $\kappa$-action on $L(S)$ to the configuration of
lines.  In the terminology of \cite[\S23 ]{segre}, a line is called real
if it is preserved by $\kappa$, and a non-real line is said to be of
the first (resp. second) kind if it meets (resp. does not meet) its
complex conjugate.  The terminology becomes a little easier to
remember if one thinks of a real line as being a line of the $0$th
kind.  The lines define 27 elements of $L(S)$, which are exactly the
27 vectors of norm~$-1$ that have inner product~$1$ with $\eta(S)$,
\cite[\S 23 ]{manin}.  Two lines meet (resp. do not meet) if the
corresponding vectors have inner product $1$ (resp. $0$).  So which
lines of $S$ are real or nonreal of the first or second kind can be
determined by studying the action of $\kappa$ on $L(S)$.  The numbers
of lines of the various types depends only on the isometry class of
$(L(S),\eta(S),\kappa^*)$, with the results given in the first five
columns of table~\ref{tab-actions-of-conjugation}.  This allows us to
identify our $\forms_{0,j}^\R$ with the classically defined families.
For example, Segre \cite[\S23]{segre} names the families $F_1,\dots,F_5$;
his $F_{j+1}$ corresponds to our $\forms_{0,j}^\R$.

\begin{table}
\centeroverfull{%
\begin{tabular}{cccccccc}%
&%
{\bf class of}&%
&%
\multicolumn{2}{c}{{\bf non-real of}}&%
{\bf fixed}&%
{\bf fixed}&%
{\bf class of}\\
&%
{\bf action}&%
{\bf real}&%
{\bf 1st}&%
{\bf 2nd}&%
{\bf space}&%
{\bf space}&%
{\bf action}\\
{\bf family}&%
{\bf on $L(S)$}&%
{\bf lines}&%
{\bf kind}&%
{\bf kind}&%
{\bf in $V(S)$}&%
{\bf in $V(T)$}&%
{\bf on $\Lambda(T)$}\\
\noalign{\smallskip}
$\forms_{0,0}^\R$&$-I$&$27$&$0$&$0$&$[\,]$&$[{+}{+}{+}{+}{-}]$&$\x_0$\\
$\forms_{0,1}^\R$&$-R_1$&$15$&$0$&$12$&$[{+}]$&$[{+}{+}{+}{-}]$&$\x_1$\\
$\forms_{0,2}^\R$&$-R_1R_2$&$7$&$4$&$16$&$[{+}{+}]$&$[{+}{+}{-}]$&$\x_2$\\
$\forms_{0,3}^\R$&$-R_1R_2R_3$&$3$&$12$&$12$&$[{+}{+}{+}]$&$[{+}{-}]$&$\x_3$\\
$\forms_{0,4}^\R$&$-R_1R_2R_3R_4$&$3$&$24$&$0$&$[{+}{+}{+}{+}]$&$[{-}]$&$\x_4$%
\end{tabular}%
}
\medskip
\caption{Action of complex conjugation on various objects
associated to $F\in\forms_{0,j}^\R$.  The 6th and 7th columns indicate
diagonalized $\F_3$-quadratic forms with $\pm1$'s on the diagonal.}  
\label{tab-actions-of-conjugation}
\end{table}

\subsection{Relation between our anti-involutions and the classical
  labeling}
From corollary~\ref{cor-componentscorollary} and \S\ref{subsec-classical-parametrization-of-5-components} We now have two labelings for the
components of $\moduli_0^\R$, namely
$$
\coprod_{j=0}^4\moduli_{0,j}^\R
=
\moduli_0^\R
\isomorphism
\coprod_{j=0}^4(G^\R\times\PGamma^\R_j)\backslash\framed_0^{\x_j}
\isomorphism
\coprod_{j=0}^4\PGamma_j^\R\backslash(H_j^4-\H)
.
$$
Our next goal is corollary~\ref{cor-our-vs-classical-components},
which shows that the labelings correspond in the obvious way.  We
defined the spaces $\forms_{0,j}^\R$ and $\moduli_{0,j}^\R$ in terms
of the action of complex conjugation $\kappa$ on $L(S)$, and we
defined the spaces $\framed_0^{\,\x_j}$ in terms of $\kappa$'s action
on $\Lambda(T)$.  To relate them, we consider the action of $\kappa$
on the 5-dimensional quadratic $\F_3$-vector spaces $V(S)$, $V(T)$
defined in terms of $L(S)$ and $\Lambda(T)$ in
\S\ref{subsec-finite-vectorspace}.

\begin{lemma}
\label{lem-actions-on-V(T)-and-V(S)}
Let $F\in\forms_0^\R$  and denote the actions of
$\kappa$ on $V(S)$ and $V(T)$ by $\hat\kappa$.  Then
$(V(S),\hat\kappa)$ and $(V(T),-\hat\kappa)$ are isomorphic as
quadratic spaces equipped with isometries.
\end{lemma}

\begin{proof}
  By Lemma~\ref{lem-isometry-finite-spaces} there is a natural
  isometry $V(S)\to V(T)$, which we will denote by $A$. 
If $a\in V(S)$ then $A(a)$ is defined by lifting $a$ to some $c\in
L_0(S)$ and then applying the construction in
\S\ref{subsec-finite-vectorspace}.  The result is the reduction modulo $\theta$ of the homology class of  $\sigma_*(d)
-\sigma_*^{-1}(d)$ for some $3$-chain $d$ in $T$.  (The asterisk in the subscript comes from our
identification of homology and cohomology in \S\ref{subsec-finite-vectorspace}.)
  From this  and the fact that $\kappa\sigma = \sigma^{-1}\kappa$
  it follows  that $A\hat\kappa = - \hat\kappa A$.
  Therefore $A$ is an isometry between the pairs
  $(V(S),\hat\kappa)$ and $(V(T),-\hat\kappa)$.
\end{proof}

\begin{lemma}
\label{lem-anti-involutions-of-5-families}
Suppose $F\in\forms_{0,j}^\R$.  Then the isometry classes of the
fixed spaces for $\kappa$ in $V(S)$ and $V(T)$ are given by the 6th
and 7th columns of table~\ref{tab-actions-of-conjugation}, and
$(\Lambda(T),\kappa^*)$ is isometric to $(\Lambda,\x_j)$ as
indicated in the last column.
\end{lemma}

\begin{proof}
  Since the conjugacy class of the action of $\kappa$ on $L_0(S)$ is
  known, it is easy to compute the fixed space in $V(S)$.  It is just
  the span of the images of the roots corresponding to
  $R_1,\dots,R_j$.  This space has dimension $j$,  and its determinant
  is $+1$
  because the roots have norm $-2\equiv1\pmod3$.  This justifies the
  6th column.  Lemma~\ref{lem-actions-on-V(T)-and-V(S)} shows that the
  fixed space in $V(T)$ is isometric to the negated space in $V(S)$,
  justifying the 7th column.  The last claim follows from
  theorem~\ref{thm-classification-of-anti-involutions}.
\end{proof}

\begin{corollary}
\label{cor-our-vs-classical-components}
We have 
$\moduli^\R_{0,j}\isomorphism\PGamma^\R_j\backslash(H^4_j-\H)$ 
for  $j=0,\dots,4$.\qed
\end{corollary}

\section{The stabilizers of the $H^4$'s}
\label{sec-H4-stabilizers}

In this section we continue to make theorem~\ref{thm-isomorphism-for-smooth} more explicit; we
know that $\moduli_0^\R=\coprod_{j=0}^4
\PGamma^\R_j\backslash(H^4_j-\H)$, and now we will describe the
$\PGamma^\R_j$.  We give two descriptions, one arithmetic and one in
the language of Coxeter groups.  The arithmetic description is easy:

\begin{theorem}
\label{thm-arithmetic-description-of-groups}
$\PGamma_j^\R\isomorphism\PO(\Psi_j)$, where $\Psi_j$
is the quadratic form on $\Z^5$ given by 
\begin{align*}
\Psi_0(y_0,\dots,y_4)&=-y_0^2+\3y_1^2+\3y_2^2+\3y_3^2+\3y_4^2\\
\Psi_1(y_0,\dots,y_4)&=-y_0^2+\3y_1^2+\3y_2^2+\3y_3^2+3y_4^2\\
\Psi_2(y_0,\dots,y_4)&=-y_0^2+\3y_1^2+\3y_2^2+3y_3^2+3y_4^2\\
\Psi_3(y_0,\dots,y_4)&=-y_0^2+\3y_1^2+3y_2^2+3y_3^2+3y_4^2\\
\Psi_4(y_0,\dots,y_4)&=-y_0^2+3y_1^2+3y_2^2+3y_3^2+3y_4^2\;.
\end{align*}
\end{theorem}

\noindent
The mnemonic is that $j$ of the coefficients of $\Psi_j$ are~$3$
rather than~$1$.  To prove the theorem, write
$\Lambda_j:=\Lambda^{\x_j}$ for the $\Z$-lattice of $\x_j$-invariant
vectors in $\Lambda$, so
$\Lambda_j=\Z^{5-j}\oplus\theta\Z^j\sset\E^5$.  The theorem now
follows from this lemma:

\begin{lemma}
\label{lem-real-lattice-isometries-extend}
For each $j$, every isometry of the $\Z$-lattice
$\Lambda_j$ is induced by an isometry of $\Lambda$.
\end{lemma}

\begin{proof}
One can check that the $\Z$-lattice $L:=\Lambda_j\cap\theta
\Lambda$ can be described in terms of $\Lambda_j$ alone as
$L=3(\Lambda_j)'$, where the prime denotes the dual
lattice.  Therefore every isometry of $\Lambda_j$ preserves
the $\E$-span of $\Lambda_j$ and $\frac{1}{\theta}L$, which
in each case is exactly $\Lambda$.
\end{proof}

Now we describe the $\PGamma^\R_j$ more geometrically; this is
interesting in its own right, and also necessary for when we allow our
cubic surfaces to have singularities
(section~\ref{sec-stable-moduli}).  Our description
relies on the good fortune that the subgroup $W_j$ generated by
reflections has index 1 or~2 in each case.  The $W_j$ are Coxeter
groups,  described in figures~\ref{fig-small-Coxeter-diagrams} and~\ref{fig-large-Coxeter-diagrams}
using an extension of the usual conventions for Coxeter diagrams,
which we now explain.  For background on Coxeter groups in this
context, see \cite{vinberg-discrete-linear-reflection-groups}.

Namely, the mirrors (fixed-point sets) of the reflections in $W_j$ chop $H_j^4$
into components, which $W_j$ permutes freely and transitively.
The closure of any one of these components is called a Weyl
chamber; we fix one and call it $C_j$.  Then $W_j$ is generated
by the reflections across the facets of $C_j$, and $C_j$ is a
fundamental domain in the strong sense that any point of $H_j^4$
is $W_j$-equivalent to a unique point of $C_j$.  We describe
$W_j$ by drawing its Coxeter diagram: its vertices (``nodes'') correspond to
the facets of $C_j$, which are joined by edges (``bonds'') that
are decorated according to  how facets meet each
other, using the following scheme:
\begin{equation}
\label{eq-bond-labeling-conventions}
\begin{tabular}{r@{\ \relax}c@{\relax}c@{\ }l@{\relax}}
no bond &
\begin{picture}(60,8)(-4,-3)
\smalldiagrams
\setlength{\unitlength}{1sp}
  \Ax=0pt    \Ay=0pt
  \Bx=\edgelengthD  \By=0pt
  \hollownode\Ax\Ay
  \hollownode\Bx\By
\end{picture}
&$\iff$&they meet orthogonally;\\
a single bond &
\begin{picture}(60,8)(-4,-3)
\smalldiagrams
\setlength{\unitlength}{1sp}
  \Ax=0pt    \Ay=0pt
  \Bx=\edgelengthD  \By=0pt
  \bond\Ax\Ay\Bx\By
  \hollownode\Ax\Ay
  \hollownode\Bx\By
\end{picture}
&$\iff$&their interior angle is $\pi/3$;\\
a double bond &
\begin{picture}(60,8)(-4,-3)
\smalldiagrams
\setlength{\unitlength}{1sp}
  \Ax=0pt    \Ay=0pt
  \Bx=\edgelengthD  \By=0pt
  \doublebond\Ax\Ay\Bx\By{0pt}{1pt}
  \hollownode\Ax\Ay
  \hollownode\Bx\By
\end{picture}
&$\iff$&their interior angle is $\pi/4$;\\
a triple bond &
\begin{picture}(60,8)(-4,-3)
\smalldiagrams
\setlength{\unitlength}{1sp}
  \Ax=0pt    \Ay=0pt
  \Bx=\edgelengthD  \By=0pt
  \triplebond\Ax\Ay\Bx\By{0pt}{1pt}
  \hollownode\Ax\Ay
  \hollownode\Bx\By
\end{picture}
&$\iff$&their interior angle is $\pi/6$;\\
a strong bond &
\begin{picture}(60,8)(-4,-3)
\smalldiagrams
\setlength{\unitlength}{1sp}
  \Ax=0pt    \Ay=0pt
  \Bx=\edgelengthD  \By=0pt
  \heavybond\Ax\Ay\Bx\By
  \hollownode\Ax\Ay
  \hollownode\Bx\By
\end{picture}
&$\iff$&they are parallel;\\
a weak bond &
\begin{picture}(60,8)(-4,-3)
\smalldiagrams
\setlength{\unitlength}{1sp}
  \Ax=0pt    \Ay=0pt
  \Bx=\edgelengthD  \By=0pt
  \dashedbond\Ax\Ay\Bx\By
  \hollownode\Ax\Ay
  \hollownode\Bx\By
\end{picture}
&$\iff$&they are ultraparallel.
\end{tabular}
\end{equation}
Parallel walls are those that do not meet in hyperbolic space
but do meet at the sphere at infinity.  Ultraparallel walls are
those that do not meet even at infinity.  

\setbox\Wzerobox=\hbox{%
\begin{picture}(0,0)%
\largediagrams
\setlength{\unitlength}{1sp}%
  \Ax=0pt    \Ay=0pt
  \Bx=\edgelengthD  \By=0pt
  \Cx=\edgelengthD  \Cy=0pt
  \Dx=\edgelengthD  \Dy=0pt
  \multiply\Cx by 2
  \multiply\Dx by 3
  \Ex=\edgelengthD  \Ey=-\edgelengthD
  \bond\Ax\Ay\Bx\By
  \bond\Bx\By\Cx\Cy
  \bond\Cx\Cy\Dx\Dy
  \doublebond\Bx\By\Ex\Ey{1pt}{0pt}%
  \twonode\Ax\Ay
  \twonode\Bx\By
  \twonode\Cx\Cy
  \twonode\Dx\Dy
  \hollownode\Ex\Ey
  \nearnode\Ax\Ay{0}{100}{0pt}{4pt}{b}{$r_5$}%
  \nearnode\Bx\By{-70}{-70}{-1pt}{0pt}{tr}{$r_3$}%
  \nearnode\Cx\Cy{0}{100}{0pt}{4pt}{b}{$r_2$}%
  \nearnode\Dx\Dy{0}{-100}{0pt}{-8pt}{t}{$r_1$}%
  \nearnode\Ex\Ey{100}{0}{2pt}{0pt}{l}{$r_4$}%
  \nearnode\Ax\Ay{0}{-140}{0pt}{0pt}{t}{$1,-1,-1,-1,0$}%
  \nearnode\Bx\By{0}{120}{0pt}{0pt}{b}{$0,0,0,1,-1$}%
  \nearnode\Cx\Cy{0}{-140}{0pt}{0pt}{t}{$0,0,1,-1,0$}%
  \nearnode\Dx\Dy{0}{120}{0pt}{0pt}{b}{$0,1,-1,0,0$}%
  \nearnode\Ex\Ey{-120}{0}{0pt}{0pt}{r}{$0,0,0,0,1$}%
  \UxD=\edgelengthD \UyD=-\edgelengthD
  \divide\UxD by 2
  \divide\UyD by 2
  \nearnode\UxD\UyD{0}{0}{0pt}{0pt}{c}{$W_0$}%
\end{picture}%
}
\setbox\Wonebox=\hbox{%
\begin{picture}(0,0)%
\largediagrams
\setlength{\unitlength}{1sp}%
%
%
%
%
%
  \Ax=     0pt  \Ay= 1.152pt
  \Bx=  .901pt  \By=  .718pt
  \Cx= 1.123pt  \Cy=- .256pt
  \Dx=  .5  pt  \Dy=-1.038pt
  \Ex=- .5  pt  \Ey=-1.038pt
  \Fx=-1.123pt  \Fy=- .256pt
  \Gx=- .901pt  \Gy=  .718pt
  \multiply\Ax by \edgelengthC
  \multiply\Ay by \edgelengthC
  \multiply\Bx by \edgelengthC
  \multiply\By by \edgelengthC
  \multiply\Cx by \edgelengthC
  \multiply\Cy by \edgelengthC
  \multiply\Dx by \edgelengthC
  \multiply\Dy by \edgelengthC
  \multiply\Ex by \edgelengthC
  \multiply\Ey by \edgelengthC
  \multiply\Fx by \edgelengthC
  \multiply\Fy by \edgelengthC
  \multiply\Gx by \edgelengthC
  \multiply\Gy by \edgelengthC
  \ABperpX=.434pt
  \ABperpY=.901pt
  \doublebond\Ax\Ay\Bx\By\ABperpX\ABperpY
  \doublebond\Ax\Ay\Gx\Gy{-\ABperpX}\ABperpY
  \bond\Bx\By\Cx\Cy
  \bond\Gx\Gy\Fx\Fy
  \heavybond\Dx\Dy\Ex\Ey
  \dashedbond\Cx\Cy\Dx\Dy
  \dashedbond\Ex\Ey\Fx\Fy
  \hollownode\Ax\Ay
  \twonode\Bx\By
  \twonode\Cx\Cy
  \threenode\Dx\Dy
  \threenode\Ex\Ey
  \twonode\Fx\Fy
  \twonode\Gx\Gy
  \nearnode\Ax\Ay{0}{-100}{0pt}{-8pt}{t}{$r_3$}%
  \nearnode\Bx\By{-70}{-70}{0pt}{0pt}{tr}{$r_6$}%
  \nearnode\Cx\Cy{-100}{0}{-2pt}{0pt}{r}{$r_5$}%
  \nearnode\Dx\Dy{-70}{70}{-2pt}{2pt}{br}{$r_4$}%
  \nearnode\Ex\Ey{70}{70}{0pt}{2pt}{bl}{$r_7$}%
  \nearnode\Fx\Fy{100}{0}{3pt}{0pt}{l}{$r_1$}%
  \nearnode\Gx\Gy{70}{-70}{3pt}{0pt}{tl}{$r_2$}%
  \nearnode\Ax\Ay{0}{100}{0pt}{0pt}{b}{$0,0,0,1,0$}%
  \nearnode\Bx\By{100}{0}{2pt}{0pt}{l}{$1,-1,-1,-1,0$}%
  \nearnode\Cx\Cy{100}{0}{2pt}{0pt}{l}{$1,0,0,0,-\theta$}%
  \nearnode\Dx\Dy{70}{-70}{0pt}{0pt}{tl}{$0,0,0,0,\theta$}%
  \nearnode\Ex\Ey{0}{-120}{0pt}{0pt}{t}{$3,-3,0,0,-\theta$}%
  \nearnode\Fx\Fy{-100}{0}{-2pt}{0pt}{r}{$0,1,-1,0,0$}%
  \nearnode\Gx\Gy{-100}{0}{-2pt}{0pt}{r}{$0,0,1,-1,0$}%
  \nearnode{0pt}{0pt}{0}{0}{0pt}{0pt}{c}{$W_1$}%
\end{picture}%
}
\setbox\Wtwobox=\hbox{%
\begin{picture}(0,0)%
\largediagrams
\setlength{\unitlength}{1sp}%
%
%
%
%
%
  \Ax=     0pt  \Ay= 1.152pt
  \Bx=  .901pt  \By=  .718pt
  \Cx= 1.123pt  \Cy=- .256pt
  \Dx=  .5  pt  \Dy=-1.038pt
  \Ex=- .5  pt  \Ey=-1.038pt
  \Fx=-1.123pt  \Fy=- .256pt
  \Gx=- .901pt  \Gy=  .718pt
  \multiply\Ax by \edgelengthC
  \multiply\Ay by \edgelengthC
  \multiply\Bx by \edgelengthC
  \multiply\By by \edgelengthC
  \multiply\Cx by \edgelengthC
  \multiply\Cy by \edgelengthC
  \multiply\Dx by \edgelengthC
  \multiply\Dy by \edgelengthC
  \multiply\Ex by \edgelengthC
  \multiply\Ey by \edgelengthC
  \multiply\Fx by \edgelengthC
  \multiply\Fy by \edgelengthC
  \multiply\Gx by \edgelengthC
  \multiply\Gy by \edgelengthC
  \ABperpX=.434pt
  \ABperpY=.901pt
  \BCperpX=.975pt
  \BCperpY=.223pt
  \CDperpX=.782pt
  \CDperpY=-.623pt
  \doublebond\Ax\Ay\Bx\By\ABperpX\ABperpY
  \doublebond\Ax\Ay\Gx\Gy{-\ABperpX}\ABperpY
  \triplebond\Bx\By\Cx\Cy\BCperpX\BCperpY
  \triplebond\Gx\Gy\Fx\Fy{-\BCperpX}\BCperpY
  \heavybond\Dx\Dy\Ex\Ey
  \doublebond\Cx\Cy\Dx\Dy\CDperpX\CDperpY
  \doublebond\Ex\Ey\Fx\Fy{-\CDperpX}\CDperpY
  \threenode\Ax\Ay
  \sixnode\Bx\By
  \twonode\Cx\Cy
  \hollownode\Dx\Dy
  \hollownode\Ex\Ey
  \twonode\Fx\Fy
  \sixnode\Gx\Gy
  \nearnode\Ax\Ay{0}{-100}{0pt}{-8pt}{t}{$r_4$}%
  \nearnode\Bx\By{-70}{-70}{0pt}{0pt}{tr}{$r_3$}%
  \nearnode\Cx\Cy{-100}{0}{-2pt}{0pt}{r}{$r_5$}%
  \nearnode\Dx\Dy{-70}{70}{-2pt}{2pt}{br}{$r_6$}%
  \nearnode\Ex\Ey{70}{70}{0pt}{2pt}{bl}{$r_2$}%
  \nearnode\Fx\Fy{100}{0}{3pt}{0pt}{l}{$r_1$}%
  \nearnode\Gx\Gy{70}{-70}{3pt}{0pt}{tl}{$r_7$}%
  \nearnode\Ax\Ay{0}{100}{0pt}{0pt}{br}  {$0,0,0,\theta,0$}%
  \nearnode\Bx\By{100}{0}{2pt}{0pt}{l}  {$0,0,0,-\theta,\theta$}%
  \nearnode\Cx\Cy{100}{0}{2pt}{0pt}{l}  {$1,0,0,0,-\theta$}%
  \nearnode\Dx\Dy{0}{-140}{6pt}{0pt}{tr}{$1,-1,-1,0,0$}%
  \nearnode\Ex\Ey{-80}{-40}{-4pt}{0pt}{r}{$0,0,1,0,0$}%
  \nearnode\Fx\Fy{-100}{0}{-2pt}{0pt}{r} {$0,1,-1,0,0$}%
  \nearnode\Gx\Gy{-100}{0}{-2pt}{0pt}{r} {$3,-3,0,-\theta,-\theta$}%
  \nearnode{0pt}{0pt}{0}{0}{0pt}{0pt}{c}{$W_2$}%
\end{picture}%
}
\setbox\Wthreebox=\hbox{%
\begin{picture}(0,0)%
\largediagrams
\setlength{\unitlength}{1sp}%
%
  \Cx=-.851pt  \Cy= 0    pt
  \Dx=-.263pt  \Dy=- .809pt
  \Ex= .688pt  \Ey=- .5pt
  \Fx= .688pt  \Fy=  .5pt
  \Gx=-.263pt  \Gy=  .809pt
  \multiply\Cx by \edgelengthC
  \multiply\Cy by \edgelengthC
  \multiply\Dx by \edgelengthC
  \multiply\Dy by \edgelengthC
  \multiply\Ex by \edgelengthC
  \multiply\Ey by \edgelengthC
  \multiply\Fx by \edgelengthC
  \multiply\Fy by \edgelengthC
  \multiply\Gx by \edgelengthC
  \multiply\Gy by \edgelengthC
  \Bx=\Cx  \By=0pt
  \advance\Bx by -\edgelengthD
  \Ax=\Bx  \Ay=0pt
  \advance\Ax by -\edgelengthD
  \CDperpX=.809pt
  \CDperpY=.588pt
  \triplebond\Ax\Ay\Bx\By{0pt}{1pt}%
  \bond\Bx\By\Cx\Cy
  \doublebond\Cx\Cy\Dx\Dy\CDperpX\CDperpY
  \dashedbond\Dx\Dy\Ex\Ey
  \heavybond\Ex\Ey\Fx\Fy
  \dashedbond\Fx\Fy\Gx\Gy
  \bond\Gx\Gy\Cx\Cy
  \twonode\Ax\Ay
  \sixnode\Bx\By
  \sixnode\Cx\Cy
  \threenode\Dx\Dy
  \hollownode\Ex\Ey
  \hollownode\Fx\Fy
  \sixnode\Gx\Gy
  \nearnode\Ax\Ay{0}{-100}{0pt}{-8pt}{t}{$r_5$}%
  \nearnode\Bx\By{0}{100}{-5pt}{4pt}{b}{$r_2$}%
  \nearnode\Cx\Cy{100}{0}{3pt}{0pt}{l}{$r_3$}%
  \nearnode\Dx\Dy{0}{100}{4pt}{3pt}{b}{$r_4$}%
  \nearnode\Ex\Ey{-70}{70}{-2pt}{0pt}{br}{$r_7$}%
  \nearnode\Fx\Fy{-70}{-70}{0pt}{0pt}{tr}{$r_1$}%
  \nearnode\Gx\Gy{0}{-100}{4pt}{-8pt}{t}{$r_6$}%
  \nearnode\Ax\Ay{0}{140}{0pt}{0pt}{b}{$1,0,0,0,-\theta$}%
  \nearnode\Bx\By{0}{-120}{0pt}{0pt}{t}{$0,0,0,-\theta,\theta$}%
  \nearnode\Cx\Cy{0}{140}{4pt}{-3pt}{br}{$0,0,-\theta,\theta,0$}%
  \nearnode\Dx\Dy{0}{-120}{0pt}{-10pt}{t}{$0,0,\theta,0,0$}%
  \nearnode\Ex\Ey{0}{-120}{-5pt}{-10pt}{tl}{$3,-1,-\theta,-\theta,-\theta$}%
  \nearnode\Fx\Fy{100}{0}{3pt}{0pt}{l}{$0,1,0,0,0$}%
  \nearnode\Gx\Gy{0}{140}{0pt}{-2pt}{bl}{$3,-3,0,-\theta,-\theta$}%
  \nearnode{0pt}{0pt}{0}{0}{0pt}{0pt}{c}{$W_3$}%
\end{picture}%
}
\setbox\Wfourbox=\hbox{%
\begin{picture}(0,0)%
\largediagrams
\setlength{\unitlength}{1sp}%
  \Ax=0pt    \Ay=0pt
  \Bx=\edgelengthD  \By=0pt
  \Cx=\edgelengthD  \Cy=0pt
  \Dx=\edgelengthD  \Dy=0pt
  \Ex=\edgelengthD  \Ey=0pt
  \Fx=\edgelengthD  \Fy=0pt
  \multiply\Cx by 2
  \multiply\Dx by 3
  \multiply\Ex by 4
  \multiply\Fx by 5
  \triplebond\Ax\Ay\Bx\By{0pt}{1pt}%
  \bond\Bx\By\Cx\Cy
  \bond\Cx\Cy\Dx\Dy
  \doublebond\Dx\Dy\Ex\Ey{0pt}{1pt}%
  \heavybond\Ex\Ey\Fx\Fy
  \twonode\Ax\Ay
  \sixnode\Bx\By
  \sixnode\Cx\Cy
  \sixnode\Dx\Dy
  \threenode\Ex\Ey
  \threenode\Fx\Fy
  \nearnode\Ax\Ay{0}{100}{0pt}{4pt}{b}{$r_5$}%
  \nearnode\Bx\By{0}{-100}{0pt}{-8pt}{t}{$r_1$}%
  \nearnode\Cx\Cy{0}{100}{0pt}{4pt}{b}{$r_2$}%
  \nearnode\Dx\Dy{0}{-100}{0pt}{-8pt}{t}{$r_3$}%
  \nearnode\Ex\Ey{0}{100}{0pt}{4pt}{b}{$r_4$}%
  \nearnode\Fx\Fy{0}{-100}{0pt}{-8pt}{t}{$r_6$}%
  \nearnode\Ax\Ay{0}{-140}{0pt}{0pt}{t}{$1,0,0,0,-\theta$}%
  \nearnode\Bx\By{0}{120}{0pt}{0pt}{b}{$0,0,0,-\theta,\theta$}%
  \nearnode\Cx\Cy{0}{-140}{0pt}{0pt}{t}{$0,0,-\theta,\theta,0$}%
  \nearnode\Dx\Dy{0}{120}{0pt}{0pt}{b}{$0,-\theta,\theta,0,0$}%
  \nearnode\Ex\Ey{0}{-140}{0pt}{0pt}{t}{$0,\theta,0,0,0$}%
  \nearnode\Fx\Fy{0}{120}{0pt}{0pt}{b}{$3,-\theta,-\theta,-\theta,-\theta$}%
  \UxD=\edgelengthD
  \divide\UxD by 2
  \nearnode\UxD\UxD{0}{0}{0pt}{0pt}{c}{$W_4$}%
\end{picture}%
}
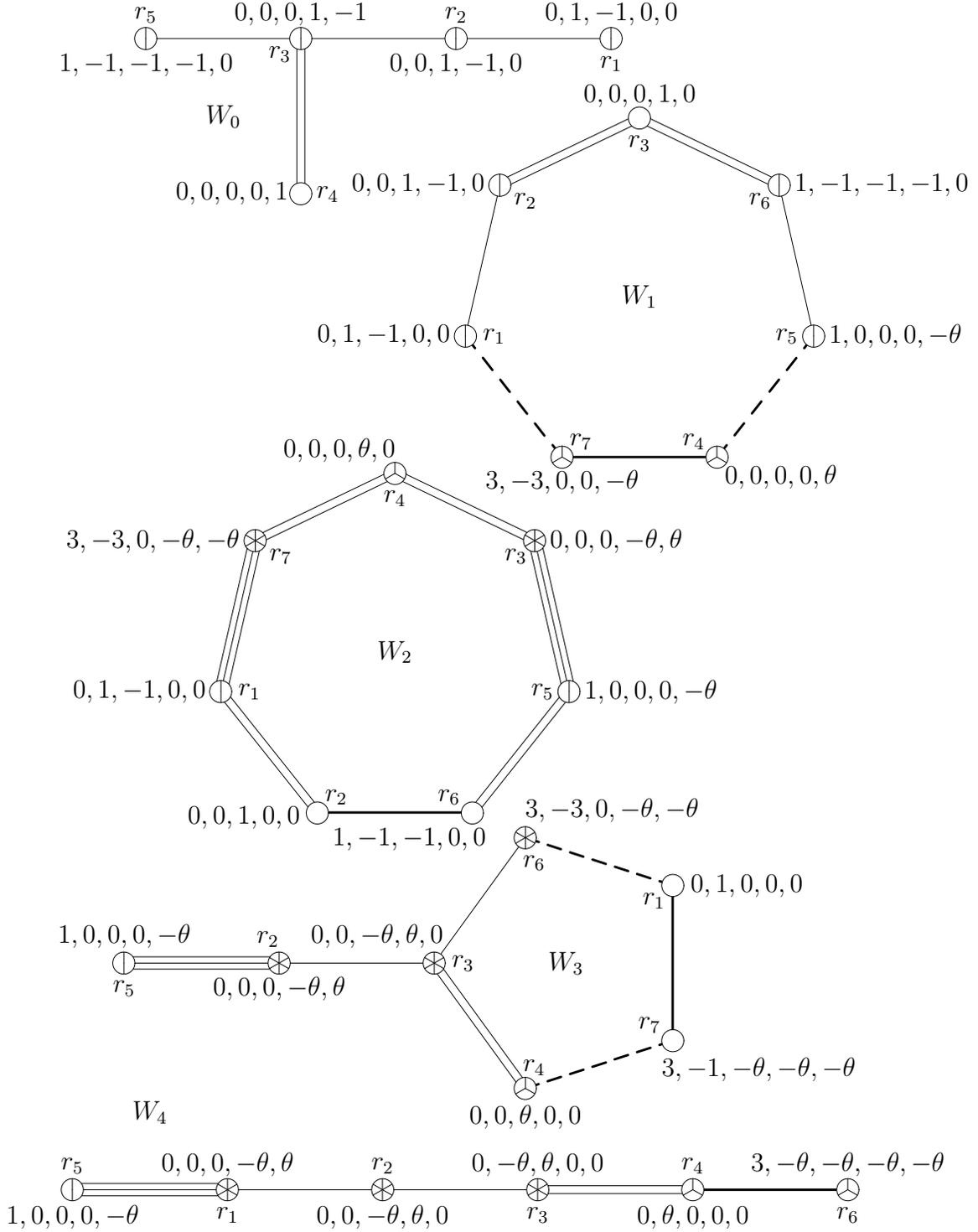
\begin{figure}[p]
\setlength{\unitlength}{1bp}
\def\LLx{0}
\def\LLy{0}
\def\width{359}
\def\height{552}
\begin{picture}(\width,\height)(\LLx,\LLy)
\put(38,536){\unhbox\Wzerobox}
\put(260,420){\unhbox\Wonebox}
\put(150,260){\unhbox\Wtwobox}
\put(227,120){\unhbox\Wthreebox}
\put(5,18){\unhbox\Wfourbox}
\end{picture}
\caption{Simple roots for the $W_j$.}
\label{fig-large-Coxeter-diagrams}
\end{figure}

Note that
the diagram for $W_j$ admits a symmetry for $j=1$ or~$2$; this
represents an isometry of $C_j$.  We now state the main theorem of this section.
\begin{theorem}
\label{thm-H4-stabilizers-in-terms-of-diagrams}
$\PGamma_j^\R$ is the semidirect product of its reflection subgroup
$W_j$, given in figure~\ref{fig-small-Coxeter-diagrams} and in more
detail in figure~\ref{fig-large-Coxeter-diagrams}, by the group of diagram automorphisms, which
is $\Z/2$ if $j=1$ or~$2$ and trivial otherwise.
\end{theorem} 

The rest of the section is devoted to the proof.  For the most part
the argument is uniform in $j$, so we will write $H$ for
$H_j^4=H_{\x_j}^4$, $W$ for $W_j$, $\x$ for $\x_j$ and $C$ for $C_j$.
We will write $\Lambda^\x$ for $\Lambda_j=\Lambda^{\x_j}$.
We call $r\in\Lambda^\x$ a root of $\Lambda^\x$ if it is primitive,
has positive norm, and the reflection in it, 
$$
x\mapsto x-2\frac{x\cdot r}{r^2}r,
$$ preserves $\Lambda^\x$.  It is easy to say what the roots
are: 

\begin{lemma}
\label{lem-what-the-roots-are}
Suppose $r\in\Lambda^\x$ is primitive in $\Lambda^\x$ and has positive
norm.  Then $r$ is a root of $\Lambda^\x$ if and only if either
$r^2\in\{1,2\}$ or else $r^2\in\{3,6\}$ and $r\in3(\Lambda^\x)'$,
where the prime denotes the dual lattice.
\end{lemma}

\begin{remark}
  Norm~3 and~6 roots are really just norm~1 and~2 roots of $\Lambda$
  in disguise.  They are primitive in $\Lambda^\x$ but divisible by
  $\theta$ in $\Lambda$, and occur when $\x$ negates rather than
  preserves a norm~1 or~2 vector of $\Lambda$.
\end{remark}

\begin{proof}
 Because $\Lambda^\x$ is 3-elementary
(the quotient by its dual lattice is an elementary abelian
3-group), any primitive $r\in\Lambda^\x$ has $3\Z\sset
r\cdot\Lambda^\x$.  If $r$ is also a root then
$r\cdot\Lambda^\x\sset\frac{1}{2}r^2\Z$, so $r^2|6$.  It is obvious
that every norm 1 or~2 vector is a root, and it is easy to see that a
norm 3 or~6 vector is a root if and only if it lies in
$3(\Lambda^\x)'$. 
\end{proof}

Given some roots $r_1,\dots,r_n$ of
$\Lambda^\x$ whose inner products are non-positive, their polyhedron is
defined to be a particular one of the regions bounded by the
hyperplanes $r_i^\perp$, namely the image in $H$ of
$$
\set{v\in \Lambda^\x\tensor\R}{\text{$v^2<0$ and $v\cdot
r_i\leq0$ for $i=1,\dots,n$}}\;.
$$ A set of simple roots for $W$ is a set of roots of $\Lambda^\x$
whose pairwise inner products are non-positive and whose polyhedron is
a Weyl chamber $C$.  Vinberg's algorithm \cite{vinberg-algorithm} seeks  a set of
simple roots for $W$.  We briefly outline how we use this algorithm.

First one chooses a vector $k$ (the ``controlling vector'')
representing a point $p$ of $H$.  We choose $k=(1,0,0,0,0)$, which
conveniently lies in all the $\Lambda^{\x_j}$.  Second, one considers
the finite subgroup of $W$ generated by the reflections in $W$
that fix $p$.  These are the
reflections in the roots of $\Lambda^\x$ that are orthogonal to $k$.

In each case it is easy to enumerate these roots, recognize the finite
Weyl group generated by their reflections, and extract a set of simple
roots for this finite group.  For example, for $j=2$ the roots are
$(0,\pm1,0,0,0)$, $(0,0,\pm1,0,0)$, $(0,0,0,\pm\theta,0)$,
$(0,0,0,0,\pm\theta)$, $(\breakok0,\breakok\pm1,\breakok\pm1,\breakok0,\breakok0)$ and
$(\breakok0,\breakok0,\breakok0,\breakok\pm\theta,\breakok\pm\theta)$, the finite Weyl group has type
$B_2\times B_2$, and a set of simple roots is $(0,1,-1,0,0)$,
$(0,0,1,0,0)$, $(0,0,0,\bar\theta,\theta)$ and $(0,0,0,\theta,0)$.  In
each of the 5 cases we called the simple roots $r_1,\dots,r_4$, and
they can be found in figure~\ref{fig-large-Coxeter-diagrams}.  

%
\newbox\normonebox
\setbox\normonebox=\hbox{\bigtextdot{
\begin{picture}(0,0)
\largediagrams
\setlength{\unitlength}{1sp}
\Ax=0pt
\Ay=0pt
\hollownode\Ax\Ay
\end{picture}
}}
%
\newbox\normtwobox
\setbox\normtwobox=\hbox{\bigtextdot{
\begin{picture}(0,0)
\largediagrams
\setlength{\unitlength}{1sp}
\Ax=0pt
\Ay=0pt
\twonode\Ax\Ay
\end{picture}
}}
%
\newbox\normthreebox
\setbox\normthreebox=\hbox{\bigtextdot{
\begin{picture}(0,0)
\largediagrams
\setlength{\unitlength}{1sp}
\Ax=0pt
\Ay=0pt
\threenode\Ax\Ay
\end{picture}
}}
%
\newbox\normsixbox
\setbox\normsixbox=\hbox{\bigtextdot{
\begin{picture}(0,0)
\largediagrams
\setlength{\unitlength}{1sp}
\Ax=0pt
\Ay=0pt
\sixnode\Ax\Ay
\end{picture}
}}
In that figure, a node
indicated by 
{\unhbox\normonebox} 
(resp. 
{\unhbox\normtwobox},
{\unhbox\normthreebox},
{\unhbox\normsixbox}) 
represents a root of
norm 1 (resp. 2, 3, 6).  The mnemonic is that the norm of the
root is the number of white regions in the symbol. Nodes are
joined according to \eqref{eq-bond-labeling-conventions}.

Next, one orders the mirrors of $W$ that miss $p$ according to their
``priority'', where the priority is any decreasing function of the
distance to $p$.  
The iterative step in Vinberg's algorithm is to consider all roots of
a given priority $p$, and suppose that previous batches have
enumerated all simple roots of higher priority.  Batch $0$ has already
been defined.  We discard those roots of priority $p$ that have
positive inner product with some simple root of a previous batch.
Those that remain are simple roots and form the current batch.  If the
polyhedron $P$ defined by our newly-enlarged set of simple roots has
finite volume then the algorithm terminates.  Otherwise, we proceed to
the next batch.  The finite-volume condition can be checked using a
criterion of Vinberg \cite[p.~22]{vinberg} on the simple roots.  There
is no guarantee that the algorithm will terminate, but if it does then
the roots obtained (the union of all the batches) form a set of simple
roots for $W$.  The algorithm terminates in all cases, with simple
roots given in figure~\ref{fig-large-Coxeter-diagrams}.

Now we can finish the proof of
theorem~\ref{thm-H4-stabilizers-in-terms-of-diagrams}, which describes
$\PGamma_j^\R$ as the semidirect product of $W_j$ by its group of
diagram automorphisms.  $W_j$ is obviously a normal subgroup of
$\PGamma_j^\R$.  It follows that $\PGamma_j^\R$ is the semidirect
product of $W_j$ by the subgroup of $\PGamma_j^\R$ that carries $C_j$
to itself.  In cases $j=0$, $3$ and $4$, $C_j$ has no symmetry, so
$\PGamma_j^\R=W_j$ as claimed.  In the remaining cases all we have to
do is check is that the nontrivial diagram automorphism $\gamma$ lies
in $\PGamma_j^\R$.  In each case, the simple roots span $\Lambda_j$,
and $\gamma$ preserves their norms and inner products.  So $\gamma\in\PGamma_j^\R$ by
lemma~\ref{lem-real-lattice-isometries-extend}.

\section{The discriminant in the real moduli space}
\label{sec-discr-and-integral-H4s}

Theorem~\ref{thm-isomorphism-for-smooth} identifies the moduli space
$G^\R\backslash\framed_0^{\,\R}$ of smooth framed real cubics with the
incomplete hyperbolic manifold $K_0$, which is the disjoint union of
the $H^4_\x-\H$.  Here $\x$ varies over the set $P\A$ of  projective classes of 
anti-involutions of $\Lambda$, as in \S\ref{subsec-framed-smooth-real-cubic-surfaces-anti-involutions},  and $\H$ is the locus in $\ch^4$
representing the singular cubic surfaces, defined in \S\ref{subsec-main-theorem-smooth-complex-case}.  For a
concrete understanding of $K_0$ we need to understand how $\H$ meets
the various $H^4_\x$'s.  Since $\H$ is the union of
the orthogonal complements $r^\perp$ of the norm~1 vectors $r$ of
$\Lambda$, we will study how such an $r^\perp$ can meet one of the $H^4_\x$'s.   We will call a component $r^\perp$ of $\H$ a {\it discriminant mirror}.

If $\x$ is an anti-involution of $\Lambda$, then one way $H^4_\x$ can
meet $r^\perp$ is if $\x(r)=\pm r$; then $H^4_\x\cap r^\perp$ is a
copy of $H^3$.  But a more complicated intersection can occur; to
describe it we need the idea of a $G_2$ root system in $\Lambda^\x$.
As in section~\ref{sec-H4-stabilizers},  a root of $\Lambda^\x$ means a norm~1 or~2
vector of $\Lambda^\x$, or a norm~3 or~6 vector of $\Lambda^\x$ that
is divisible in $\Lambda$ by $\theta$.  By a $G_2$ root system in
$\Lambda^\x$ we mean a set of six roots of norm~2 and six roots of
norm~6, all lying in a two-dimensional sublattice of $\Lambda^\x$.
Such a set of vectors automatically forms a copy of what is commonly
known as the $G_2$ root system.  The reason these root systems are
important is that each $G_2$ root system $R$ in $\Lambda^\x$
determines an isometric copy of $\E^2$ in $\Lambda$, and hence two discriminant mirrors.
The $\E^2$ is just $\Lambda\cap\left(\langle
  R\rangle\tensor_\Z\C\right)$.  To
see this, introduce coordinates on the complex span of $R$, in which
$R$ consists of the vectors obtained by permuting the coordinates of $(1,-1,0)$ and $\pm(2,-1,-1)$ in the space 
$$
\C^2=\{(x,y,z)\in\C^3:x+y+z=0\}\;,
$$
with the usual metric.
Since $\Lambda$ contains $\frac{1}{\theta}$ times the norm~6 roots, it
also contains
\begin{equation}
\label{eq-roots-from-g2}
\begin{split}
r_1&=\frac{1}{\theta}(2,-1,-1)+\w(1,-1,0)=-\frac{1}{\theta}(\w,\wbar,1)
\rlap{\quad and}\\
r_2&=-\frac{1}{\theta}(2,-1,-1)+\wbar(1,-1,0)=\frac{1}{\theta}(\wbar,\w,1)\;.
\end{split}
\end{equation}
These have norm~1 and are orthogonal, so they span a copy of
$\E^2$.  Observe also that $\x$ exchanges the $r_i$, and that
each of the discriminant mirrors $r_i^\perp$ meets $H^4_\x$ in the same $H^2$,
namely $H^4_\x\cap R^\perp$. 

The following lemma asserts that these are the only ways that $H^4_\x$ can meet $\H$.   In terms of cubic surfaces, the first possibility parametrizes surfaces with a  real node, while the second parametrizes surfaces with a complex conjugate pair of nodes.

\begin{lemma}
\label{lem-intersections-of-mirrors-andH4s}
Suppose $\x$ is an  anti-involution of $\Lambda$ and $M$ is a
discriminant mirror with $M\cap H^4_\x\neq\emptyset$.  Then either
\begin{enumerate}
\item
\label{case-H3}
$M\cap H^4_\x$ is a copy of $H^3$, namely $H^4_\x\cap r^\perp$ for
a root $r$ of $\Lambda^\x$ of norm~$1$ or~$3$, or
\item
\label{case-H2}
$M\cap H^4_\x$ is a copy of $H^2$, namely $H^4_\x \cap R^\perp$ for a  $G_2$ root system $R$ in $\Lambda^\x$.
\end{enumerate}
Conversely, if $r$ is a root of norm~$1$ or~$3$ in $\Lambda^\x$ (resp. $R$ is a
$G_2$ root system in $\Lambda^\x$), then $H^4_\x\cap r^\perp$
(resp. $H^4_\x\cap R^\perp$)  is the intersection of $H^4_\x$
with some discriminant mirror.  
\end{lemma}

\begin{proof}
  As a discriminant mirror, $M=r^\perp$ for some norm~1 vector $r$ of
  $\Lambda$.  Since $M\cap H^4_\x\neq\emptyset$, $M$ contains points
  fixed by $\x$, so that $\x(M)$ meets $M$, which is to say that
  $r^\perp$ meets $\x(r)^\perp$.  By lemma~\ref{lem-hyperplanes-are-orthogonal}, either $r^\perp=\x(r)^\perp$ or
  $r\bot\x(r)$.  In the first case, $\x$ preserves the $\E$-span of
  $r$.  The anti-involutions of a rank one free $\E$-module are easy
  to understand: every one leaves invariant either a generator or
  $\theta$ times a generator.  Therefore $\Lambda^\x$ contains a unit
  multiple of $r$ or $\theta r$.  Then conclusion \ref{case-H3}
  applies.  In the second case, $r_1:=r$ and $r_2:=\x(r)$ span a copy of
  $\E^2$ and $\Lambda^\x$ contains the norm~2 roots $\alpha
  r_1+\bar\alpha r_2$ and norm~6 roots $\alpha\theta
  r_1+\bar\alpha\bar\theta r_2$, where $\alpha$ varies over the units
  of $\E$.  These form a $G_2$ root system $R$ in $\Lambda^\x$, and it
  is easy to see that
$$
M\cap H^4_\x\ =\ M\cap\x(M)\cap\ H^4_\x\ =\ R^\perp\cap H^4_\x
$$
is a copy of $H^2$.  Therefore conclusion \ref{case-H2} applies.  

The
converse is easy:  if $r$ is a root of $\Lambda^\x$ of norm~1
or~3 then we take the discriminant mirror to be $r^\perp$, and if $R$ is
a $G_2$ root system in $\Lambda^\x$ then we take $M$ to be
either $r_1^\perp$ or $r_2^\perp$ for $r_1$ and $r_2$ as in \eqref{eq-roots-from-g2}.
\end{proof}

\begin{corollary}
\label{cor-H4s-meeting-hyperplanes}
For $j=0,\dots,4$, $H^4_j\cap\H$ is the union of the orthogonal
complements of the discriminant roots of $\Lambda_j$ and the $G_2$
root systems in $\Lambda_j$.  \qed
\end{corollary}

For our applications we need to re-state this result in terms of the
fundamental chamber $C_j$ for $W_j$:

\begin{lemma}
\label{lem-intersection-of-discriminant-and-weyl-chamber}
If $x\in C_j$ then $x\in\H$ if and only if either
\begin{enumerate}
\item
\label{item-norm-1-or-3-perp-implies-discriminant}
$x$ lies in  $r^\perp$ for $r$ a simple root of $W_j$ of
norm~$1$ or~$3$, or
\item
\label{item-G2-implies-discriminant}
$x$ lies in $r^\perp\cap s^\perp$, where $r$ and $s$ are simple
roots of $W_j$ of norms~$2$ and~$6$, whose mirrors meet at angle $\pi/6$.
\end{enumerate}
\end{lemma}

\begin{proof}
If \ref{item-norm-1-or-3-perp-implies-discriminant} holds then
$x$ obviously lies in $\H$.  If \ref{item-G2-implies-discriminant}
holds then  the reflections in $r$ and $s$
generate a dihedral group of order~$12$, and the images of $r$
and $s$ under this group form a $G_2$ root system $R$ in
$\Lambda_j$.  Then $x\in\H$ by
lemma~\ref{lem-intersections-of-mirrors-andH4s}.  

To prove the
converse, suppose $x\in C_j\cap\H$.  By lemma~\ref{lem-intersections-of-mirrors-andH4s}, either
$x\in r^\perp$ for a root $r$ of $\Lambda_j$ of norm~$1$
or~$3$, or else $x\in R^\perp$ for a $G_2$ root system $R$ in
$\Lambda_j$.  We treat only the second case because the
first is similar but simpler.  We choose a set $\{r,s\}$ of
simple roots for $R$, which necessarily have norms~$2$ and~$6$
and whose mirrors necessarily meet at angle $\pi/6$.  Then
$r^\perp$ and $s^\perp$ are two of the walls for some Weyl
chamber $C'$ of $W_j$.  This uses the fact that no two distinct mirrors of
$W_j$ can meet, yet make an angle less than $\pi/6$.  (If there
were such a pair of mirrors then there would be such a pair
among the simple roots of $W_j$.)  We apply
the element of $W_j$ carrying $C'$ to $C_j$; since $C_j$ is a
fundamental domain for $W_j$ in the strong sense, this
transformation fixes $x$.  Then the images of $r$ and $s$ are
simple roots of $W_j$ and the facets of $C_j$ they define both
contain $x$.
\end{proof}

We remark that all the triple bonds in
figure~\ref{fig-large-Coxeter-diagrams} come from $G_2$ root systems,
so the condition on the norms of $r$ and $s$ in part
\ref{item-G2-implies-discriminant} of the lemma may be dropped.  This
leads to our final description of the moduli space of smooth real
cubic surfaces:

\begin{theorem}
\label{thm-smooth-moduli-orbifolds}
The moduli space $\moduli_0^\R$ falls into five components
$\moduli_{0,j}^\R$, $j=0,\dots,4$.  As a real analytic orbifold, $\moduli_{0,j}^\R$ is
isomorphic to an open sub-orbifold of $\PGamma_j^\R\backslash
H_j^4$, namely the open subset obtained by deleting the images
in $\PGamma_j^\R\backslash H_4^j$ of the faces of $C_j$
corresponding to the blackened nodes 
\hbox{\smalltextdot{%
\begin{picture}(0,0)%
\smalldiagrams
\setlength{\unitlength}{1sp}%
\solidnode{0}{0}%
\end{picture}%
}}
and triple bonds 
\hbox{\smallcoxintext{%
\begin{picture}(0,0)%
\smalldiagrams
\setlength{\unitlength}{1sp}%
  \Ax=0pt    \Ay=0pt
  \Bx=\edgelengthD  \By=0pt
  \triplebond\Ax\Ay\Bx\By{0pt}{1pt}%
  \hollownode\Ax\Ay
  \hollownode\Bx\By
\end{picture}%
}}
of figure~\ref{fig-small-Coxeter-diagrams}.
\qed
\end{theorem}

The two kinds of walls of the $C_j$ play such different roles that we
will use the following language.  In light of the theorem, a wall
corresponding to a blackened node in
figure~\ref{fig-small-Coxeter-diagrams} will be called a {\it
  discriminant wall}.  The other walls will be called {\it Eckardt
  walls}, because the corresponding real cubic surfaces are exactly
those that have real Eckardt points.  (Eckardt points are not
important in this paper; they just provide a convenient name for these
walls.  They are points through which pass three lines of the surface.
The reader interested in more background, in particular the relation between Eckardt points and existence of  automorphisms of order two, may consult \cite[\S\S 98, 100  and
101]{segre}.)

\section{Topology of the moduli space of smooth surfaces}
\label{sec-topologysmoothmoduli}

This section and the next two are applications of the theory developed
so far.  The theoretical development continues in section~\ref{sec-stable-moduli}.

The  description of $\moduli_0^\R$ in theorem~\ref{thm-smooth-moduli-orbifolds} is so explicit that many facts
about real cubic surfaces and their moduli can be read off the
diagrams.  In this section we give
presentations of the orbifold fundamental groups
$\orbpi_1(\moduli_{0,j}^\R)$ of the components of $\moduli_0^\R$
and  prove that the $\moduli_{0,j}^\R$ have contractible
(orbifold) universal covers.

\begin{theorem}
\label{thm-orbifold-fundamental-groups}
The orbifold fundamental groups of the components of $\moduli_0^\R$
are:
\begin{align*}
\orbpi_1(\moduli_{0,0}^\R)&\isomorphism S_5\\
\orbpi_1(\moduli_{0,1}^\R)&\isomorphism (S_3\times S_3)\semidirect\Z/2\\
\orbpi_1(\moduli_{0,2}^\R)&\isomorphism (D_\infty\times D_\infty)\semidirect\Z/2\\
\orbpi_1(\moduli_{0,3}^\R)\isomorphism
\orbpi_1(\moduli_{0,4}^\R)&\isomorphism
\ \begin{picture}(0,0)%
\smalldiagrams
\setlength{\unitlength}{1sp}%
%
%
%
\Ax=0pt    \Ay=0pt
\Bx=\edgelengthD  \By=0pt
\Cx=\edgelengthD  \Cy=0pt
\Dx=\edgelengthD  \Dy=0pt
\multiply\Cx by 2
\multiply\Dx by 3
\bond\Ax\Ay\Bx\By
\bond\Bx\By\Cx\Cy
\bond\Cx\Cy\Dx\Dy
\hollownode\Ax\Ay
\hollownode\Bx\By
\hollownode\Cx\Cy
\hollownode\Dx\Dy
\divide\Bx by 2
\myput\Bx\By{\makebox(0,0)[b]{\raise 5pt\hbox{$\infty$}}}%
\end{picture}
\kern151pt
\end{align*}
where the $\Z/2$ in each semidirect product exchanges
the displayed factors of the normal subgroup.
\end{theorem}

Here $S_n$ is the symmetric group, $D_\infty$ is the infinite
dihedral group, and the last group is a Coxeter group with the
given diagram.  We have labeled the leftmost bond ``$\infty$'',
indicating the absence of a relation between two generators,
rather than a strong or weak bond, because we are describing
the fundamental group as an abstract group, not as a concrete
reflection group.  We remark that $\orbpi_1(\moduli^\R_{0,2})$ is
isomorphic to the Coxeter group of the Euclidean $(2,4,4)$ triangle. 

\begin{proof}[Proof of theorem~\ref{thm-orbifold-fundamental-groups}]
The general theory of Coxeter groups (see for example
\cite{de-la-Harpe}) allows us to write down a presentation for
$W_j$.    The standard
generators for $W_j$ are the reflections across the facets of
$C_j$.  Two of these reflections $\rho$ and $\rho'$ satisfy
$(\rho\rho')^n=1$ for $n=2$ (resp. $3$, $4$, or $6$) if the
corresponding nodes are joined by no bond
(resp. a single bond, double bond, or triple bond).  These
relations and the relations that the generators are involutions
suffice to define $W_j$.

We get a presentation of $\orbpi_1\bigl(W_j\backslash(H_j^4-\H)\bigr)$ from
the presentation of $W_j$ by omitting some of the generators and
relations.  Since the generators of $W_j$ correspond to the
walls of $C_j$, and removing $\H$ from $C_j$ removes the discriminant
walls, we leave out those generators.  Since removing these
walls also removes all the codimension two faces which are their
intersections with other walls, we also leave out all the
relations involving the omitted generators.  Finally, we leave
out the relations coming from triple bonds, because removing
$\H$ from $C_j$ removes the codimension two faces corresponding
to these bonds.  For $j=0$, $4$ or $5$, $\PGamma_j^\R=W_j$ and
we can read off $\orbpi_1(\moduli_{0,j}^\R)$ from the diagram, with
the results given in the statement of the theorem.  For $j=1$ or
$2$ the same computation shows that
$\orbpi_1\bigl(W_j\backslash(H_j^4-\H)\bigr)$ is $S_3\times S_3$ or
$D_\infty\times D_\infty$.  To describe
$\orbpi_1(\moduli_{0,j}^\R)$ one must take the semidirect product by
the diagram automorphism.  This action can also be read from
figure~\ref{fig-small-Coxeter-diagrams}.
\end{proof}

In the proof of the following theorem, the subgroup of $W_j$ generated
by the reflections across the Eckardt walls of $C_j$ (the walls
represented by hollow nodes in
figure~\ref{fig-small-Coxeter-diagrams}) plays a major role.  We call
it $\WEck_j$.  It has index~1 or~2 in a group $T_j$ that plays a major
role in the next section.

\begin{theorem}
\label{thm-aspherical}
The $\moduli_{0,j}^\R$ are aspherical orbifolds, in the sense
that their orbifold universal covers are contractible manifolds.
\end{theorem}

\begin{proof}
We write $D_j$ for the component of $H_j^4-\H$ containing
$C_j-\H$, and think of
$\moduli_{0,j}^\R$ as 
$$
\text{(the stabilizer of $D_j$ in $\PGamma_j^\R$)}\bigm\backslash D_j\;.
$$
Since $D_j$ is an orbifold cover of $\moduli_{0,j}^\R$, it suffices to
show that $D_j$ is aspherical.  One way to understand $D_j$ is as the
union of the translates of $C_j-\H$ under $\WEck_j$.
Alternately, $\WEck_j$ is the stabilizer
of $D_j$ in $W_j$.  Now we look at the $D_j$ individually.  $\WEck_0$ is
the finite group $S_5$, and the four Eckardt walls are the walls
containing a vertex $P$ of $C_0$.  (Vertices in $H^n$ of an
$n$-dimensional Coxeter polyhedron correspond bijectively to $n$-node
subdiagrams of the Coxeter diagram which generate finite Coxeter
groups.)  Therefore $D_0$ is the interior of a finite-volume
polyhedron centered at $P$, so $D_0$ is not just aspherical but even
contractible.  The same argument works for $j=1$, with $S_3\times S_3$
in place of $S_5$.

The case $j=2$ is more complicated, even though $\WEck_2$
is still finite (the 
Coxeter group $G_2\times G_2$) and the Eckardt walls are still the walls
meeting at a vertex $P$ of $C_2$.  The complication is that the
fixed-point set of each $G_2$ factor lies in $\H$.  The
result is that $D_2$ is the interior of a finite-volume
polyhedron centered at $P$, minus its intersection with two
mutually orthogonal $H^2$'s that meet transversely at $P$.
Therefore $D_2$ is homeomorphic to a product of two punctured
open disks, so it is aspherical.

Now we will treat $j=3$; the case $j=4$ is just the same.  What
is new is that $\WEck_3$ is infinite.  However, one of the discriminant
walls (the lower of the rightmost two in
figure~\ref{fig-small-Coxeter-diagrams}) is orthogonal to all of
the Eckardt walls.  Therefore  $\WEck_3$ preserves the
hyperplane $H$ containing this discriminant wall.  Furthermore, $\WEck_3$ is
the Coxeter group
\begin{equation}
\label{eq-633-Coxeter-diagram}
\begin{picture}(0,0)
\setlength{\unitlength}{1sp}
\smalldiagrams
\Ax=0pt    \Ay=0pt
\Bx=\edgelengthD  \By=0pt
\Cx=\edgelengthD  \Cy=0pt
\Dx=\edgelengthD  \Dy=0pt
\multiply\Cx by 2
\multiply\Dx by 3
\triplebond\Ax\Ay\Bx\By{0pt}{1pt}
\bond\Bx\By\Cx\Cy
\bond\Cx\Cy\Dx\Dy
\hollownode\Ax\Ay
\hollownode\Bx\By
\hollownode\Cx\Cy
\hollownode\Dx\Dy
\end{picture}
\kern151pt
\end{equation}
which is a nonuniform lattice in $\PO(3,1)$, acting on $H$ in the
natural way.  In particular, $H$ is a component of the boundary of
$D_3$, and every $\WEck_3$-translate of $C_3$ has one of its facets
lying in $H$.  Finally, $H$ is orthogonal to the codimension two face
of $C_3$ associated to the triple bond in
figure~\ref{fig-small-Coxeter-diagrams}, and therefore orthogonal to
all of its $\WEck_3$-translates.  We summarize: $D_3$ is the interior of
an infinite-volume convex polyhedron in $H^4_3$, minus the union of a
family of $H^2$'s, each orthogonal to the distinguished facet $H$.
Therefore $D_3$ is homeomorphic to the product of an open interval
with $H-Z$, where $Z$ is the intersection of $H$ with the union of
these $H^2$'s.  

$H-Z$ can be understood in terms of $\WEck_3$'s action on
it.  A fundamental domain for $\WEck_3$ is a simplex with shape described in
\eqref{eq-633-Coxeter-diagram}, and the edge corresponding to the
triple bond lies in $Z$.  Indeed, $Z$ is the union of the
$\WEck_3$-translates of this edge.  Direct visualization in hyperbolic 3-space shows
that $Z$ is the union of countably many disjoint geodesics.  Therefore
$H^3-Z$ has the homotopy type of countably many circles, all identified
at a point.  This follows from stratified Morse theory; see
Theorem 10.8 of \cite{GoreskyMacPherson}.
\end{proof}

\section{Relation with the work of B. Segre}
\label{sec-Segre}

Classical knowledge about the topology of each connected component of
the space of real smooth cubic forms was restricted to Segre's
computation \cite[\S\S34--54]{segre} of the monodromy of the
fundamental group of each component on the configuration of lines
(real and complex) of a surface of that type.  Our methods give a
transparent calculation of this monodromy group $M_j$ over each
component $\forms^\R_{0,j}$, because the fundamental groups of these
components are almost the same as the groups
$\orbpi_1(\moduli_{0,j}^\R)$ computed in the last section.  In
particular we show that four of Segre's computations are correct and
correct an error in the remaining one.

\begin{lemma}
\label{lem-exact-homotopy-sequence-segre}
For each $j=0,\dots,4$, there is an exact sequence
\begin{equation}
\label{eq-exact-homotopy-sequence-segre}
1\to\Z/2\to\pi_1(\forms_{0,j}^\R)\to\orbpi_1(\moduli_{0,j}^\R)\to\Z/2\to1.
\end{equation}
Here, the image of the middle map is the orientation-preserving
subgroup of $\orbpi_1(\moduli_{0,j}^\R)$ and the kernel is
$\pi_1(G^\R)$.
\end{lemma}

For use in the proof and elsewhere in this section, we write $D_j$ for
a component of $H^4_j-\H$ and $T_j$ for its stabilizer in
$\PGamma_j^\R$.  This group is generated by the subgroup $\WEck_j$ of
$W_j$ introduced in section~\ref{sec-topologysmoothmoduli}, together
with the diagram automorphism if one is present.

\begin{proof}[Proof sketch.]
  There are two ingredients.  One is the exact homotopy sequence of
  the fibration $G^\R\to Y_j\to D_j$, where
  $Y_j\sset\framed_0^{\,\R}$ is the $g^\R$-preimage of
  $D_j$.  The other ingredient is the interaction of this sequence
  with the $T_j$-action on $Y_j$ and $D_j$.
  We omit the details.  We remark that it would be more classical to
  consider $\pi_1(P\forms_{0,j}^\R)$ instead.  This would change the
  $\Z/2=\pi_1(G^\R)$ on the left into $(\Z/2)^2=\pi_1(PG^\R)$, but not
  affect our other considerations.
\end{proof}

The monodromy of $\pi_1(\forms_0,F)$ on lines is the classical map to
the Weyl group $W(E_6)\isomorphism\Aut\bigl(L(S),\eta(S)\bigr)$, as in
\S\ref{subsec-finite-vectorspace}.  As explained there, this is the
same as the reduction modulo $\theta$ of the monodromy representation
$\pi_1(\forms_0)\to\PGamma$.    Therefore the monodromy of
$\pi_1(\forms_{0,j}^\R)$ on lines can be computed by taking the image
of $\pi_1(\forms_{0,j}^\R)$ in $PO(V)$.  

By lemma~\ref{lem-exact-homotopy-sequence-segre},
$\pi_1(\forms_{0,j}^\R)$ acts by the orientation-preserving subgroup of
$T_j$,  so it will suffice
to compute the map $T_j\to PO(V)$ and then pass to the image of the
subgroup.  Computing this map is very easy: one lifts each generator
of $T_j$ to an element of $\Gamma$, reduces modulo $\theta$ to get an
element of $O(V)$, and then passes to $PO(V)$.  The ambiguity in the
lift is unimportant because of the passage to $PO(V)$.  One can work
out the details in each case (see below for $j=2$), with the
following result:
 
\begin{theorem}
\label{thm-monodromy-images}
Let $M_j$ denote the image of the monodromy representation
$\pi_1(\forms_{0,j})\to W(E_6)$.
Then
\par
$\begin{array}{l}
\phantom{M_3\isomorphism{}}M_0 \isomorphism A_5 \\
\phantom{M_3\isomorphism{}}M_1 \isomorphism S_3 \times S_3 \\
\phantom{M_3\isomorphism{}}M_2 \isomorphism  (\Z/2)^3\semidirect \Z/2 \\
\phantom{M_3\isomorphism{}}{M_3\isomorphism{}}M_4 \isomorphism S_4
\end{array}$

\noindent In $M_2$,  $\Z/2$ has fixed-point set $(\Z/2)^2$ in
$(\Z/2)^3$, and this characterizes the group structure.
\qed
\end{theorem}

\begin{caution}
  It turns out that $\orbpi_1(\moduli_{0,1}^\R)\isomorphism (S_3\times
  S_3)\semidirect\Z/2\isomorphism(\Z/3)^2\semidirect D_8$ has two subgroups isomorphic to $S_3\times
  S_3$.  The one which is the image of $\pi_1(\forms_{0,1}^\R)$, here
  manifesting as $M_1$, is not the obvious one but the other one.
\end{caution}

\begin{remark}
In the two cases where $\pi_1(\forms_{0,j}^\R)$ is
finite, namely $j=0$ or $1$, its representation in $W$ is almost
faithful.  The kernel is precisely the central $\Z/2=\pi_1(G^\R)$.
\end{remark}

Our results confirm Segre's computation of $M_0,\dots,M_4$, except for
$M_2$, which he gives as $(\Z/2)^2$ at the end of \S46 (page
72).  Our $M_0,\dots,\breakok M_4$ are his $\Gamma_1,\dots,\Gamma_5$,
introduced in  \S 34 and
computed in \S 35 to \S 54.   In each case, he also gave a very
detailed  description of the action
on various configurations of lines and tritangent planes of a surface
in the appropriate component.  We will show how to obtain this more
detailed information from our perspective, in the case $j=2$.  

By definition, $T_2=\WEck_2\semidirect\Z/2$ is
generated by the reflections in $r_1$, $r_3$, $r_5$ and $r_7$ from the
middle diagram in figure~\ref{fig-large-Coxeter-diagrams}, together
with the diagram automorphism.  By the choice of roots, we already
have lifts of the four reflections to $\Gamma$.  For $r_1$ and $r_5$,
reduction modulo $\theta$ gives the reflections of $V$ in the images
of these two roots.  The same applies to $r_3$ and $r_7$, except that
one must divide them by $\theta$ before reducing modulo $\theta$.  The
point is that reflection in $r_3$ is the same as reflection in
$r_3/\theta$, a primitive element of $\Lambda$.  Therefore it acts on
$V$ as the reflection in the image of $r_3/\theta$.  We lift the
diagram automorphism in the obvious way, to the isometry of
$\Lambda$ that exchanges $r_1\leftrightarrow r_5$, $r_3\leftrightarrow
r_7$ and fixes $r_4$.

So $M_2$ is the subgroup of $PO(V)$ generated by the diagram
automorphism and the products of any evenly many reflections in the
vectors $(0,1,-1,0,0)$, $(0,0,0,1,-1)$, $(1,0,0,0,0)$ and
$(0,0,0,1,1)$, which are the reductions modulo $\theta$ of $r_1$,
$r_3/\theta$, $r_5$ and $r_7/\theta$, respectively.  These are
mutually orthogonal, so $M_2\isomorphism(\Z/2)^3\semidirect(\Z/2)$.  The
action of the diagram automorphism on $(\Z/2)^3$ is easy to work out,
with the result stated in theorem~\ref{thm-monodromy-images}.

Now we work out the action on lines.  A key ingredient is the
dictionary on p. 26 of \cite{ATLAS} between the lines and tritangent
planes of a cubic surface and certain objects in $V$.  Namely, the
tritangent planes of $S$ correspond to the ``plus-points'' of
$PV$;   with our choice of $q $, these are
the lines $\langle v\rangle$ in $V$ with $q(v) = -1$ (see p.~xii of
\cite{ATLAS}).  And a line of $S$ corresponds to a ``base'', which
with our choice of $q$ means a collection of five mutually orthogonal
lines in $V$, each spanned by a vector $v$ with $q(v) = -1$.  The fact
that a base contains five plus-points corresponds to the fact that
each line on $S$ is contained in $5$ tritangent planes of $S$.  One
can check that each plus-point is contained in exactly $3$ bases,
corresponding to the fact that each tritangent plane contains $3$
lines.

The anti-involution $\x_2$ that defines the component $\forms_{0,2}^\R$  
acts on
$V$ by
$$
\x_2(x_0,\dots,x_4) = (x_0,x_1,x_2,-x_3,-x_4).
$$
This lets one work out which lines and tritangent planes are real.
Together with our explicit generators for $M_2$, one can obtain
extremely detailed results, for example:

\begin{theorem}
  A real cubic surface of type $j=2$ has exactly five real tritangent
  planes.  These have exactly one line $\ell$ in common, necessarily
  preserved by the monodromy group $M_2$.  This group preserves
  exactly one of these five planes, and also each of the lines in it,
  which are real.  Of the remaining four, two (say $t_1,t_2$) contain
  non-real lines and two (say $t_3,t_4$) contain only real lines.
  $M_2$ acts on these planes by $\Z/2$, and contains an involution
  acting by $t_1\leftrightarrow t_2$, $t_3\leftrightarrow t_4$.  The
  subgroup of $M_2$ that preserves each of $t_1,\dots,t_4$ acts on the
  lines they contain as follows.  It consists of every permutation of
  the form: for any evenly many of $t_1,\dots,t_4$, in each of them
  exchange the two lines of $S$ other than $\ell$.
  The action of $M_2\isomorphism(\Z/2)^3\semidirect(\Z/2)$ on these eight
  lines is faithful.  \qed
\end{theorem}

\begin{remark}
  A careful reading of \S 46 of \cite{segre} shows that Segre
  discusses actions of subgroups of $M_2$  on the set of four lines in
  $t_1$ and $t_2$, and on the set of four lines in $t_3$ and $t_4$,
  but does not seem to discuss the whole group.  It is not clear how
  he reaches his conclusion that $M_2 \cong (\Z/2)^2$.
\end{remark}

\section{Volumes}
\label{sec-volume}

In this section we compute the volume of each $\PGamma_j^\R\backslash H^4_j$
by computing its orbifold Euler characteristic and using the general relation
\[
  \hbox{vol}(M) = \frac{\hbox{vol}(S^n)}{\chi(S^n)} | \chi(M) |
=\frac{2^n\pi^{n/2}(n/2)!}{n!}|\chi(M)|
\]
for a hyperbolic orbifold $M$ with $n=\dimension M$ even.  For the
Euler characteristic, consider the subgroup $W_j$ generated by
reflections, and its
fundamental polyhedron $C_j$, described by the its Coxeter diagram in
figure~\ref{fig-small-Coxeter-diagrams}.   $W_j$ has index $\delta$ in $\PGamma_j^\R$ with
$\delta = 1$ or $2$.  The latter case occurs when the diagram has an
automorphism of order two.  Consider therefore the orbifold
$W_j\backslash H^4_j$.  Since
\[
\delta\cdot \chi\bigl(\PGamma_j^\R\backslash H^4_j)=
\chi\bigl(W_j\backslash H^4_j\bigr), 
\]
it suffices to compute the right-hand side. To this end, consider a
face $F$ of $C_j$ and its stabilizer $W_j(F)$ in
$W_j$.  If $\Phi$ stands for the set of proper faces of $C_j$, then
\[
   \chi\bigl(W_j\backslash H^4_j\bigr) 
     =1 + \sum_{F \in \Phi} 
       \frac{ (-1)^{\mathop{\hbox{\scriptsize dim}} F} \chi(F)}{ | W_j(F) | } 
     = 1 + \sum_{F \in \Phi} 
       \frac{ (-1)^{\mathop{\hbox{\scriptsize dim}} F}}{ | W_j(F) | } .
\]
Let $\Delta$ be a Coxeter diagram, let $\Sigma(\Delta)$ be the set of
nonempty subdiagrams describing finite Coxeter groups, and for $E$ in
$\Sigma$, let $|E|$ be the number of its nodes and $W(E)$ be the
associated Coxeter group.  The face of $C_j$ corresponding to $E$ has
codimension $|E|$ in an even-dimensional space, so the previous
equation can be written as
\[
 \chi\bigl(W_j\backslash H^4_j\bigr) 
     = 1 + \sum_{E \in \Sigma} 
       \frac{ (-1)^{|E|}}{ | W(E) | }.
\] 
 
For the Coxeter diagrams that occur in this paper, the enumeration of
subdiagrams is lengthy but easy.  We did the computations by hand and
then checked them with a computer.  Consider, for instance, the case
of $W_0$.  Every proper subdiagram describes a finite Coxeter group
except the one got by omitting the rightmost node; for example, the
other four-node subdiagrams (which describe vertices of $C_0$) have
types $B_4$, $A_1^2\times A_2$, $A_1\times B_3$ and $A_4$.  The
resulting contribution to the Euler characteristic is
\[
(-1)^4\Bigl(  \frac{ 1 }{ 2^4\cdot 4! } + \frac{1 }{ 2^2\cdot 3!} + 
  \frac{ 1 }{ 2^4\cdot3!} + \frac{1 }{ 5!}\Bigr)  = \frac{121 }{ 1920}.
\] 
Carrying out the full enumeration and computing the orders of the corresponding Weyl groups, one finds that
\[
   \chi(\PGamma^\R_0\backslash H^4) = 1 - \frac{ 5 }{ 2} + \frac{ 17 }{ 8 } - \frac{ 11 }{ 16 } + \frac{121 }{ 1920} = \frac{ 1 }{ 1920 }.
\]
This gives the first entry in table~\ref{tab-volumes-table}.  The other
calculations are similar.

\section{Moduli of Stable Complex Cubic Surfaces}
\label{sec-stable-complex-moduli}

All of our discussions have been restricted to smooth cubic surfaces.
However, one can still discuss moduli of singular surfaces, when the
singularities are mild.  In this section we recall from \cite{ACT} the
material necessary for our treatment in the next section of the moduli
space of stable real cubic surfaces.  Here stable means stable in the
sense of geometric invariant theory (GIT).  For cubic surfaces it is
classical that this is simply the condition that the singularities be
no worse than nodes (ordinary double points).  See \cite[(3.1)]{ACT}.
It is also classical that a cubic surface can have at most 4 nodes;
see \cite{bruce-wall} for a modern treatment.

We write $\forms_s$ for the space of all complex cubic forms defining
stable surfaces.  It contains $\forms_0$, and a standard result from
GIT is that $G$ acts properly on $\forms_s$.  Therefore
$\moduli_s:=G\backslash\forms_s$ is a complex-analytic manifold as well
as a quasi-projective variety.  The main result of \cite{ACT}
is an isomorphism
$\moduli_s\isomorphism\PGamma\backslash\ch^4$, extending the
isomorphism $\moduli_0\isomorphism\PGamma\backslash(\ch^4-\H)$ of
theorem~\ref{thm-isomorphism-for-smooth}.  For $\moduli_s$, this is an isomorphism of algebraic
varieties, but not of orbifolds (see below).  

Here are the main ideas
behind this isomorphism; theorem~\ref{thm-main-theorem-stable-complex-case} is the precise statement.  As
explained in \cite[(3.10) and (3.3)]{ACT}, the covering space
$\framed_0\to\forms_0$ extends to a ramified covering space
$\framed_s\to\forms_s$.  Here $\framed_s$ is the Fox completion (or
normalization) of $\framed_0\to\forms_0$ over $\forms_s$, and we call
its elements framed stable cubic forms.  The naturality of this
construction implies that the $G$- and $\PGamma$-actions on
$\framed_0$ extend to $\framed_s$.  The key facts about $\framed_s$
are the following:

\begin{lemma}[\protect{\cite[(3.14)]{ACT}}]
\label{lem-G-acts-freely-on-stable-forms}
$G$ acts freely on $\framed_s$, so $G\backslash\framed_s$ is a complex
manifold.  \qed
\end{lemma}

\begin{theorem}[\protect{\cite[(3.17--19)]{ACT}}]
\label{thm-main-theorem-stable-complex-case}
The period map $g:\framed_0\to\ch^4$ extends to $\framed_s$, factors
through $G\backslash\framed_s$, and induces a $\PGamma$-equivariant
diffeomorphism $G\backslash\framed_s\isomorphism\ch^4$.  It sends the
$k$-nodal cubic surfaces to the locus in $\ch^4$ where exactly $k$ of
the hyperplanes of $\H$ meet.  Furthermore, the induced map
$$
\moduli_s
=
G\backslash\forms_s
=
(G\times\PGamma)\backslash\framed_s
\to
\PGamma\backslash\ch^4
$$
is an isomorphism of analytic spaces, but not of complex orbifolds. \qed
\end{theorem}

We will need the following local description of $\framed_s\to\forms_s$
in the next section, and give a refinement of it in lemma~\ref{lem-local-model-for-Fs-to-Cs-real-case}.

\begin{lemma}[\protect{\cite[(3.10)]{ACT}}]
\label{lem-local-model-for-complex-framed-stable-to-stable}
Suppose $f\in\framed_s$ lies over $F\in\forms_s$, the cubic surface
$S$ defined by $F$ has $k$ nodes, and $x:=g(f)\in\ch^4$.  

Then there
exist coordinates $t_1,\dots,t_4$ on $\ch^4$ which identify it with
the unit ball in $\C^4$ and $x$ with the origin, whose pullbacks
to $\framed_s$ can be extended to local coordinates $t_1,\dots,t_{20}$
around $f$, such that
\begin{enumerate}
\item
\label{item-ti=0-gives-hyperplanes-thru-x}
The components of $\H$ passing through $x$ are defined by
$t_1=0,\dots,t_k=0$.
\item
\label{item-action-of-stabilizer-on-ts}
The stabilizer $\PGamma_{\!f}$ of $f$ is $(\Z/6)^k$, acting on $\framed_s$
 by multiplying $t_1,\dots,t_k$ by sixth roots of unity and
leaving $t_{k+1},\dots,t_{20}$ invariant.
\item
\label{item-ui-in-terms-of-ti}
The functions
$u_1=t_1^6,\dots,u_k=t_k^6,u_{k+1}=t_{k+1},\dots,u_{20}=t_{20}$ are
local coordinates around $F\in\forms_s$.
\item
\label{item-ui=0-gives-discriminant-thru-F}
The discriminant $\D\sset\forms_s$ near $F$ is the union of the
hypersurfaces $u_1=0,\dots,u_k=0$.
\end{enumerate}
In these coordinates, the period map $g:\framed_s\to\ch^4$ is given
near $f$ by forgetting $t_5,\dots,t_{20}$.
\qed
\end{lemma}

The failure of the variety isomorphism
$\moduli_s\isomorphism\PGamma\backslash\ch^4$ to be an orbifold
isomorphism arises because of the presence of the $(\Z/6)^k$ ramification
of $\framed_s\to\forms_s$, described in parts
\ref{item-action-of-stabilizer-on-ts} and \ref{item-ui-in-terms-of-ti} of this lemma.
This is explained in more detail  in
\cite[(3.18)]{ACT}.  We showed in \cite[(3.19--20)]{ACT} how to modify
the orbifold structure of $\PGamma\backslash\ch^4$ so that its
identification with $\moduli_s$ becomes an orbifold isomorphism.  The
ramification of $\framed_s\to\forms_s$ will be the main issue in our
treatment of stable real surfaces.  Although we will not strictly need
the results of \cite[(3.18--20)]{ACT}, the ideas they embody will play
a major role in our analysis.

\section{Moduli of Stable Real Cubic Surfaces}
\label{sec-stable-moduli}

The goal of this section is to understand the moduli space
$\moduli_s^\R$ of stable real cubic surfaces as a quotient of real
hyperbolic space $H^4$.
In the previous section we defined $\forms_s$ as the space of forms
defining GIT-stable cubic surfaces, and recalled that $G$ acts
properly on it.  Therefore $G^\R$ acts properly on
$\forms_s^\R:=\forms_s\cap\forms^\R$.  We denote the quotient by
$\moduli_s^\R$, which is a real-analytic orbifold in a natural way.
In the smooth case we were able to pass from the complex orbifold
isomorphism $\moduli_0\isomorphism\PGamma\backslash(\ch^4-\H)$ to the
real orbifold isomorphisms
$\moduli_{0,j}^\R\isomorphism\PGamma_j^\R\backslash(H^4_j-\H)$ fairly
easily.  A very substantial complication in the stable case is that
the isomorphism $\moduli_s\isomorphism\PGamma\backslash\ch^4$ is not
an orbifold isomorphism (see the end of section~\ref{sec-stable-complex-moduli}).  Nevertheless
we will find a real-hyperbolic orbifold structure on
$\moduli_s^\R$ by identifying it with $\PGamma^\R\backslash H^4$ for a
suitable lattice $\PGamma^\R$ in $\PO(4,1)$.  It will be obvious that
this structure agrees with the moduli-space orbifold structure on
$\moduli_0^\R$.  

It is possible  to skip the theory of this section and
construct  $\PGamma^\R$ by gluing together the 5 orbifolds
$\PGamma^\R_j\backslash H_j^4$ along their discriminant walls.  There is
an essentially unique way to do this that makes sense (the one in
section~\ref{sec-gluing}), and one obtains the orbifold
$\PGamma^\R\backslash H^4$.  This is what we did at first, but this did not give a proof that the resulting space is homeomorphic to  $\moduli_s^\R$.
The essential content of this section is to give an intrinsic
definition of the hyperbolic structure on $\moduli_s^\R$.  Then section~\ref{sec-gluing} plays the
role of computing an orbifold structure already  known to exist, rather than
constructing it.
We begin the detailed analysis.  

\subsection{The space of framed stable real
  surfaces; their moduli space $K$}
We define $\framed_s^{\,\R}$ as
the preimage of $\forms_s^\R$ in $\framed_s$.  We will see that it is
not a manifold, because of the ramification of $\framed_s\to\forms_s$,
but it is a union of embedded submanifolds.  We define $K$ to be
$G^\R\backslash\framed_s^{\,\R}$, which is not a manifold either.  At this
point it is merely a topological space; below, we will equip it with a
metric structure.  Essentially by definition, $\moduli_s^\R$ coincides
with $\PGamma\backslash K$.  If $K$ were a manifold then this would
define an orbifold structure on $\moduli_s^\R$.  But it is not, so we
must take a different approach.  First we will give a local
description of $\framed_s^{\,\R}\sset\framed_s$, and then show that
$g:\framed_s\to\ch^4$ induces a local embedding $K\to\ch^4$.  This
makes $K$ into a metric space, using the path metric obtained by
pulling back the metric on $\ch^4$.  Finally, we will study the action
of $\PGamma$ on $K$ to deduce that $\PGamma\backslash K$, as a metric
space, is locally modeled on quotients of $H^4$ by finite groups.
Such a metric space has a unique hyperbolic orbifold structure.  The
completeness of this structure on $\moduli_s^\R$ then follows from the
completeness of $\PGamma\backslash\ch^4$, and orbifold uniformization
then implies the existence of a discrete group $\PGamma^\R$ acting on
$H^4$ with $\moduli_s^\R\isomorphism \PGamma^\R\backslash H^4$.  See
section~\ref{sec-gluing} for a concrete description of $\PGamma^\R$
and section~\ref{sec-nonarithmeticity} for a proof that it is not
arithmetic.

We begin with a local description of $\framed_s^{\,\R}$.  This
requires a refinement of the local description of $\framed_s$ given in
lemma~\ref{lem-local-model-for-complex-framed-stable-to-stable}.

\begin{lemma}
\label{lem-local-model-for-Fs-to-Cs-real-case}
Under the assumptions of
lemma~\ref{lem-local-model-for-complex-framed-stable-to-stable},
suppose $F$ lies in $\forms_s^\R$ and
defines a surface with $2a$ non-real and $b$ real nodes.  Then the
local coordinates of
lemma~\ref{lem-local-model-for-complex-framed-stable-to-stable} on
$\ch^4$, $\framed_s$ and $\forms_s$ may be chosen to also satisfy the
following:  near $F$, complex conjugation $\kappa:\forms_s\to\forms_s$
acts by 
\begin{equation}
\label{eq-local-coordinates-complex-conjugation-forms}
u_i\mapsto
\begin{cases}
\bar u_{i+1}&\hbox{for $i$ odd and $i\leq2a$}\\
\bar u_{i-1}&\hbox{for $i$ even and $i\leq2a$}\\
\bar u_i&\hbox{for $i>2a$.}
\end{cases}
\end{equation}
\end{lemma}

\begin{proof}
The coordinates $t_1,\dots t_k$ of Lemma~\ref{lem-local-model-for-complex-framed-stable-to-stable} are in  one to one correspondence with the nodes of $S$ and, once a correspondence is fixed, each $t_i$ is unique up to multiplication by a complex number of absolute value one.  The same is therefore true of the coordinates $u_1,\dots u_k$.   Since $\kappa$ permutes the sets $u_i = 0$ in the same way that it permutes the nodes of $S$, namely interchanges complex conjugate nodes and preserves real ones, it is clear that the $t_i$ and hence the $u_i$ can be chosen so that $\kappa$ acts on $u_1,\dots u_k$ as in \eqref{eq-local-coordinates-complex-conjugation-forms}. That the
  coordinates may be chosen so that $\kappa$ also acts this way on
  $u_{k+1},\dots,u_{20}$ can be derived from the fact that any
  anti-involution of a complex manifold in a neighborhood of any fixed
  point is modeled on complex conjugation of $\C^n$.

\end{proof}

In these local coordinates, $\forms_s^\R$ is the fixed-point set of
$\kappa$.  To describe $\framed_s^{\,\R}$ near $f$, we simply compute the
preimage of $\forms_s^\R$.  The most important
cases are first, a single real node ($a=0$, $b=1$), and second, a
single pair of conjugate nodes ($a=1$, $b=0$).

In the case of a single real node, $\framed_s^{\,\R}$ near $f$ is modeled on
a neighborhood of the origin in
\begin{equation}
\label{eq-local-descrip-FsR-when-1-real-node}
\{(t_1,\dots,t_{20})\in \C^{20} : t_1^6, t_2,\dots,t_{20}\in\R\}\;.
\end{equation}
That is, a neighborhood of $f$ is modeled on six copies of $\R^{20}$, glued
together along a  common $\R^{19}$.    

In the case of two complex conjugate nodes, $\framed_s^{\,\R}$ near $f$ is
modeled on a neighborhood of the origin in
\begin{equation}
\label{eq-local-descrip-FsR-when-2-conjugate-nodes}
\{(t_1,\dots,t_{20})\in \C^{20} : t_2^6 = \bar t_1^{\,6}\hbox{ and }
t_3,\dots t_{20}\in \R\}\;.
\end{equation} 
That is,  on the union of six copies of
$\R^{20}$, glued together along a common $\R^{18}$.  The $\R^{18}$ is
given by $t_1=t_2=0$ and 
maps diffeomorphically to $\D\cap\forms_s^\R$, and each component of the complement is
a six-fold cover of the part of $\forms_s^\R-\D$ near $F$.

\subsection{Many different real structures}
We have defined $\framed_s^{\,\R}$ as the preimage of $\forms_s^\R$ in
$\framed_s$, but it is helpful to think of it as the union of the real
loci of many different real structures.  Namely, if $f\in\framed_0^{\,\R}$
lies over $F\in\forms_0^\R$ then there is a unique lift of the complex
conjugation $\kappa$ of $\forms_s$ to an anti-involution $\x$ of
$\framed_0$ that fixes $f$.  
We saw this construction in \S\ref{subsec-framed-smooth-real-cubic-surfaces-anti-involutions}, and wrote $\framed_0^{\,\x}$ for
$\x$'s fixed points.  (See also the last part of the proof of
theorem~\ref{thm-isomorphism-for-smooth}.) 

The naturality of the Fox completion
implies that $\x$ extends to $\framed_s$.  
Then the set $\framed_s^{\,\x}$
of $\x$'s fixed points is the real locus of one real structure, namely
$\x$.  It is clear that $\framed_s^{\,\R}$ is the union of the
$\framed_s^{\,\x}$ as $\x$ varies over the anti-involutions of $\framed_s$
lying over $\kappa$.  We have already seen this set of
anti-involutions: it is $P\A$, the set of projective classes of
anti-involutions of $\Lambda$ defined in \S\ref{subsec-framed-smooth-real-cubic-surfaces-anti-involutions}. 

\subsection{The family of isomorphisms
  $G^\R\backslash\framed_s^{\,\x}\to H^4_\x$}
Recall from section~\ref{subsec-def-period-map} that $g:\framed_s\to\ch^4$ is the complex period map.
Lemma~\ref{lem-diffeomorphism-for-each-antiinvolution-stable} below is
the extension of the diffeomorphism
$G^\R\backslash\framed_0^{\,\x}\isomorphism H^4_\x-\H$ of
theorem~\ref{thm-isomorphism-for-smooth} to
$G^\R\backslash\framed_s^{\,\x}\isomorphism H^4_\x$; to prove it we
need the following general principle.

\begin{lemma}
\label{lem-real-moduli-to-complex-moduli-is-proper}
Let $G$ be a Lie group acting properly and with finite stabilizers on
a smooth manifold $X$, let $F$ be a finite group of diffeomorphisms of
$X$ normalizing $G$, let $X^F$ be its fixed-point set, and let $G^F$
be its centralizer in $G$.  Then the natural map $G^F\backslash X^F\to
G\backslash X$ is
proper. 
\end{lemma}

\begin{proof}
  We write $\pi$ and $\pi^F$ for the maps $X\to G\backslash X$ and
  $X^F\to G^F\backslash X^F$, and $f$ for the natural map
  $G^F\backslash X^F\to G\backslash X$.  We prove the theorem under
  the additional hypothesis that $F$ and $G$ meet trivially; this is
  all we need and the proof in the general case is similar.  This
  hypothesis implies that the group $H$ generated by $G$ and $F$ is
  $G\rtimes F$.  Begin by choosing a complete $H$-invariant Riemannian
  metric on $X$.

To prove $f$ proper it suffices to exhibit for any
$G$-orbit $\mathcal{O}\sset X$ a $G$-invariant neighborhood $U\sset X$ with
$f^{-1}(\pi(U))$ precompact.
Since  $G$ has finite index in $H$,  $\mathcal{O}.H\sset
X$ is the union of finitely many $G$-orbits.  Using properness and
Riemannian geometry one finds $\e>0$ such that (1) distinct
$G$-orbits in $\mathcal{O}.H$ lie at distance~$>\e$, and (2) any point of $X$
at distance~$<\e$ from $\mathcal{O}$ has a unique nearest point in $\mathcal{O}$.  We
take $U$ to be the open $\e/2$-neighborhood of $\mathcal{O}$.

To show that $f^{-1}(\pi(U))$ is precompact we will exhibit a compact set
$K\sset X^F$ with $\pi^F(K)$ containing $f^{-1}(\pi(U))$.  We claim
that there are finitely many $G^F$-orbits in $\mathcal{O}\cap X^F$, so we can
choose orbit representatives $\tilde x_1,\dots,\tilde x_n$.  If
$\mathcal{O}\cap X^F$ is empty then this is trivial.  If $\mathcal{O}\cap X^F$
is nonempty, say containing $\tilde x$, then the $G^F$-orbits in
$\mathcal{O}\cap X^F$ are in bijection with the conjugacy classes of splittings
of
$$
1\to G_{\tilde x}\to G_{\tilde x}\rtimes F\to F\to 1\;,
$$
where $G_{\tilde x}$ is the $G$-stabilizer of $\tilde x$.  Since
$G_{\tilde x}$ is finite, there are finitely many splittings, hence
finitely many orbits.  We take $K$ to be the union of the closed
$\e/2$-balls around $\tilde x_1,\dots,\tilde x_n$, intersected with
$X^F$.  (In particular, $K$ is empty if $\mathcal{O}\cap X^F$ is.)

$K$ is obviously compact, so all that remains is to prove $f^{-1}(\pi(U))\sset
\pi^F(K)$.  If $f^{-1}(\pi(U))$ is empty then we are done.  Otherwise,
suppose $y\in f^{-1}(\pi(U))\sset G^F\backslash X^F$ and let $\tilde y\in X^F$ lie
over it.  Now, $\tilde y$ is $F$-invariant and $F$ permutes the
$G$-orbits in $\mathcal{O}.H$.  Since $\tilde y$ lies within $\e/2$ of $\mathcal{O}$, it
lies at distance~$>\e/2$ of every other $G$-orbit in $\mathcal{O}.H$, so $F$
preserves $\mathcal{O}$.  Therefore $F$ preserves the unique point $\tilde x$
of $\mathcal{O}$ closest to $\tilde y$, so $\tilde x\in X^F$.  We choose
$g\in G^F$ with $\tilde x.g$ equal to one of the $\tilde x_i$.  Then
$\tilde y.g$ lies within $\e/2$ of $\tilde x.g=\tilde x_i$, hence lies
in $K$, and $\pi^F(\tilde y.g)=y$, proving $f^{-1}(\pi(U))\sset\pi^F(K)$. 
\end{proof}

\begin{lemma}
\label{lem-diffeomorphism-for-each-antiinvolution-stable}
For every $\x\in P\A$, the restriction of the period map
$g:\framed_s\to\ch^4$ to $\framed_s^{\,\x}$ defines an isomorphism
$G^\R\backslash\framed_s^{\,\x}\isomorphism H^4_{\x}$ of real-analytic
manifolds.
\end{lemma}

\begin{proof}
  It is a local diffeomorphism because its rank is everywhere~$4$ by
  theorem~\ref{thm-main-theorem-stable-complex-case}.  Injectivity
  follows from the argument used for
  theorem~\ref{thm-isomorphism-for-smooth}; this uses the freeness of the 
  $G$ action on $\framed_s$, see 
  lemma~\ref{lem-G-acts-freely-on-stable-forms}.  To see surjectivity,
  we apply the previous lemma with $G=G$, $X=\framed_s$, $F=\{1,\x\}$,
  $X^F=\framed_s^{\,\x}$ and $G^F=G^\R$.  Therefore the map
  $G^\R\backslash\framed_s^{\,\x}\to
  G\backslash\framed_s\isomorphism\ch^4$ is proper, so its image is
  closed.  Theorem~\ref{thm-isomorphism-for-smooth} tells us that the
  image contains the open dense subset
  $g(\framed_0^{\,\x})=H^4_{\x}-\H$, so the map is surjective.
\end{proof}

\subsection{The local embedding $K\to\ch^4$}

The purpose of this subsection is to show that the complex period map
$g$ defines a local embedding $K=G^\R\backslash\framed_s^{\,\R}\to\ch^4$,
and use this to define a piecewise-hyperbolic metric on $K$.

\begin{lemma}
\label{lem-local-homeomorphism-for-framed-stable-moduli}
Suppose $f\in\framed_s^{\,\R}$, and $\alpha_1,\dots,\alpha_\ell$ are the
elements of $P\A$  that fix $f$.  Then the map
\begin{equation}
\label{eq-local-homeomorphism-for-framed-stable-moduli}
G^\R\backslash\bigl(\cup_{i=1}^\ell\framed_s^{\,\alpha_i}\bigr)
\to
\cup_{i=1}^\ell H^4_{\alpha_i}
\end{equation}
induced by $g$ is a homeomorphism.
\end{lemma}

The left side of
\eqref{eq-local-homeomorphism-for-framed-stable-moduli} contains a
neighborhood of the image of $f$ in $K$, so the lemma implies that
$g:K\to\ch^4$ is a local embedding.  We will write $K_f$
for the right side of
\eqref{eq-local-homeomorphism-for-framed-stable-moduli}.
It is the part of $K$ relevant to $f$.

Before giving the proof, we observe that it's easy to work out
formulas for $\alpha_1,\dots,\alpha_\ell$ in the local coordinates
$t_1,\dots,t_{20}$ on $\framed_s$.  Since
lemma~\ref{lem-local-model-for-Fs-to-Cs-real-case} gives a formula for
$\kappa$ near $F\in\forms_s^\R$ in the local coordinates
$u_1,\dots,u_{20}$, and $\framed_s\to\forms_s$ is given by
$u_1=t_1^6,\dots,u_k=t_k^6,u_{k+1}=t_{k+1},\dots,u_{20}=t_{20}$, one
can simply write down the lifts of $\kappa$.  For example, in the
one-real-node case there are $6$ lifts, given by
$$
(t_1,\dots,t_{20})\mapsto(\bar t_1\zeta^i,\bar t_2,\dots,\bar t_{20})\;,
$$
where $\zeta=e^{\pi i/3}$, and in the conjugate-pair case there are also $6$ lifts, given by
$$
(t_1,\dots,t_{20})\mapsto(\bar t_2\zeta^i,\bar t_1\zeta^i,\bar t_3,\dots,\bar
t_{20})\;.
$$

\begin{proof}[Proof of Lemma~\ref{lem-local-homeomorphism-for-framed-stable-moduli}]
We first claim that for all $i$ and $j$, 
$$
g:\framed_s^{\,\alpha_i}\cap\framed_s^{\,\alpha_j}\to H^4_{\alpha_i}\cap H^4_{\alpha_j}
$$ is surjective.  To see this, let $C$ be the component of
$\framed_s^{\,\alpha_i}\cap\framed_s^{\,\alpha_j}$ containing $f$.
This is a component of the fixed-point set of the finite group
generated by $\alpha_i$ and $\alpha_j$.  In particular, it is a smooth
manifold whose tangent spaces are all totally real.  Since $C$ is
connected, its dimension everywhere is its dimension at $f$, which by
our local coordinates is $16+\dimension_\R (H^4_{\alpha_i}\cap
H^4_{\alpha_j})$.  Since the tangent spaces are totally real and the
kernel of the derivative of the period map has complex dimension~16,
the (real) rank of $g|_C$ equals $\dimension_\R (H^4_{\alpha_i}\cap
H^4_{\alpha_j})$ everywhere.  Therefore $g(C)$ is open.  It is also
closed, since $g$ induces a diffeomorphism from each
$G^\R\backslash\framed_s^{\,\alpha_i}$ to $H^4_{\alpha_i}$ for each $i$.  This
proves surjectivity, since $H^4_{\alpha_i}\cap H^4_{\alpha_j}$ is
connected.

Now we prove the lemma itself; 
the map \eqref{eq-local-homeomorphism-for-framed-stable-moduli} is surjective and
proper because $G^\R\backslash\framed_s^{\,\alpha_i}\to H^4_{\alpha_i}$ is surjective
and proper for each $i$.  To prove injectivity, suppose
$a_i\in G^\R\backslash\framed_s^{\,\alpha_i}$ for $i=1,2$ have the same image in $\ch^4$.
Then their common image lies in $H^4_{\alpha_1}\cap H^4_{\alpha_2}$, so by
the claim above there exists
$b\in G^\R\backslash(\framed_s^{\,\alpha_i}\cap\framed_s^{\,\alpha_j})$ with the same image.
Since each $G^\R\backslash\framed_s^{\,\alpha_i}\to H^4_{\alpha_i}$ is injective, each
$a_i$ coincides with $b$, so $a_1=a_2$.
\end{proof}

\subsection{The local metric structure on $\PGamma\backslash
  K\isomorphism\moduli_s^\R$}
At this point we know that $g$ locally embeds
$K=G^\R\backslash\framed_s^{\,\R}$ into $\ch^4$, and even have an
identification of small open sets in $K$ with open sets in unions of
copies of $H^4$ in $\ch^4$.  The induced path-metric on $K$ is the
largest metric which preserves the lengths of paths; under it, $K$ is
piecewise isometric to $H^4$.  $K$ is not a manifold---it may be
described locally by suppressing the coordinates $t_5,\dots,t_{20}$
from our local description of $\framed_s^{\,\R}$.  (See \eqref{eq-local-descrip-FsR-when-1-real-node} and
\eqref{eq-local-descrip-FsR-when-2-conjugate-nodes} for two examples.)  Nevertheless,
corollary~\ref{cor-locally-a-hyperbolic-orbifold} below shows us that
the path metric on $\moduli_s^\R=\PGamma\backslash K$ is locally
isometric to quotients of $H^4$ by finite groups.  This is the key
to defining the hyperbolic orbifold structure on $\moduli_s^\R$.

Our goal is to prove corollary~\ref{cor-locally-a-hyperbolic-orbifold}, that every point of
$\PGamma\backslash K$ has a neighborhood modeled on $H^4$ modulo a
finite group.  This requires a careful analysis with several different
subgroups of $\PGamma$ associated to $f\in\framed_s^{\,\R}$.  One of them
is $A_f$, the subgroup of $\PGamma$ fixing the image of $f$ in
$K=G^\R\backslash\framed_s^{\,\R}$.  This contains
$\PGamma_{\!f}\isomorphism(\Z/6)^k$, often strictly.  The third group is
$B_f$, the subgroup of $\PGamma_{\!f}$ generated by the order~6 complex
reflections associated to the {\it real} nodes of $S$, rather than all
the nodes.  

Lemma~\ref{lem-local-homeomorphism-for-framed-stable-moduli} says that
\begin{equation}
\label{eq-Xf-mod-Af-mapping-to-MsR}
A_f\backslash K_f
=
(A_f\times G^\R)\backslash\bigl(\cup_{i=1}^\ell\framed_s^{\,\alpha_i}\bigr)
\to
(\PGamma\times G^\R)\backslash\framed_s^{\,\R}
=
\moduli_s^\R
\end{equation}
is a homeomorphism in a neighborhood of the image of $f$ in $K_f\sset
K$, where $\alpha_1,\dots,\alpha_\ell$ are as in that lemma.  Therefore it
suffices to study $A_f\backslash K_f$.  It turns out that this is best
done by first treating the intermediate quotient $B_f\backslash K_f$.

So our next goal is to understand $B_f\backslash K_f$ in
coordinates.  The all-nodes-real case is much simpler than the general
case, and should allow the reader to understand all the ideas in the
rest of this section.

\begin{lemma}
\label{lem-preparation-for-orbifold-charts}
If $S$ has only real nodes, then $B_f\backslash K_f$ is isometric to
$H^4$.  If $S$ has a single pair of conjugate nodes, and possibly also
some real nodes, then $B_f\backslash K_f$ is isometric to the union of
six copies of $H^4$ identified along a common $H^2$.  If $S$ has two
pairs of conjugate nodes then $B_f\backslash K_f=K_f$ is the union of
$36$ copies of $H^4$, any two of which meet along an $H^2$ or at a
point.  

In each case, $A_f$ acts transitively on the indicated
$H^4$'s.  If $H$ is any one of them, and $(A_f/B_f)_H$ its stabilizer,
then the natural map
\begin{equation}
\label{eq-locally-hyperbolic-modulo-finite-group}
(A_f/B_f)_H\!\bigm\backslash\! H\to (A_f/B_f)\!\bigm\backslash\!(B_f\backslash K_f)
=A_f\backslash K_f
\end{equation}
is an isometry of path metrics.
\end{lemma}

\begin{corollary}
\label{cor-locally-a-hyperbolic-orbifold}
Every point of $\PGamma\backslash K$ has a neighborhood isometric to
the quotient of an open set in $H^4$ by a finite group of isometries.
\end{corollary}

\begin{proof}
The left term of \eqref{eq-locally-hyperbolic-modulo-finite-group} is
a quotient of $H^4$ by a finite group, and the right term contains a
neighborhood of the image of $f$ in $\PGamma\backslash K$.  The
corollary follows from the fact that
\eqref{eq-locally-hyperbolic-modulo-finite-group} is a local isometry.
\end{proof}

\begin{proof}[Proof of lemma~\ref{lem-preparation-for-orbifold-charts}.]
  We take $x=g(f)$ as before and refer to the coordinates
  $t_1,\dots,t_4$ from
  lemma~\ref{lem-local-model-for-complex-framed-stable-to-stable} that
  identify $\ch^4$ with $B^4$.  Recall that the cubic surface $S$ has
  $2a$ non-real and $b$ real nodes, with $k=2a+b$.  The stabilizer
  $\PGamma_{\!f}$ of $f$ in $\PGamma$ acts on $\ch^4$ by multiplying
  $t_1,\dots,t_k$ by 6th roots of unity, and $B_f$ acts by multiplying
  $t_{2a+1},\dots,t_{2a+b}$ by 6th roots of unity.  $K_f$ may be
  described in the manner used to obtain
  \eqref{eq-local-descrip-FsR-when-1-real-node} and
  \eqref{eq-local-descrip-FsR-when-2-conjugate-nodes}, with
  $t_5,\dots,t_{20}$ omitted.  With concrete descriptions of $K_f$ and
  $B_f$ in hand, one can work out $B_f\backslash K_f$.  Here are the
  results for the various cases.

First suppose $S$ has only real nodes.  Then 
$$
K_f=\{(t_1,\dots,t_4)\in B^4:t_1^6,\dots,t_k^6,t_{k+1},\dots,t_4\in\R\}\;.
$$
Each of the $2^k$ subsets 
\begin{align*}
K_{f,\e_1,\dots,\e_k}=\{(t_1,\dots,t_4)\in B^4:i^{\e_1}t_1,&\dots,i^{\e_k}t_k\in[0,\infty)\\
&\hbox{and } t_{k+1},\dots,t_4\in\R\}\;,
\end{align*}
indexed by $\e_1,\dots,\e_k\in\{0,1\}$, is isometric to the closed region in
$H^4$ bounded by $k$ mutually orthogonal hyperplanes.  Their union
$U$ is
a fundamental domain for $B_f$ in the sense that it maps
homeomorphically and piecewise-isometrically onto $B_f\backslash
K_f$.  Under its path metric, $U$ is isometric to $H^4$, say by the
following map, defined separately on each $K_{f,\e_1,\dots,\e_k}$ by
$$
(t_1,\dots,t_k)\mapsto
(-i^{\e_1}t_1,\dots,-i^{\e_k}t_k,t_{k+1},\dots,t_4)\;. 
$$
This identifies $B_f\backslash K_f$ with the standard $H^4$ in
$\ch^4$.  

If $S$ has a single pair of non-real nodes and no real nodes, then
$B_f$ is trivial and $B_f\backslash K_f=K_f$.  
The $\alpha_i$ are  the 6 maps
$$
\alpha_i:(t_1,\dots,t_4)\mapsto(\bar t_2\zeta^i,\bar t_1\zeta^i,\bar
t_3,\bar t_4)
$$
with $i\in\Z/6$, whose fixed-point sets are 
$$
H^4_{\alpha_i}=\{(t_1,\dots,t_4)\in B^4:t_2=\bar t_1\zeta^i\hbox{ and
}t_3,t_4\in\R\}\;. 
$$
It is obvious that any two of these $H^4$'s meet along the $H^2\sset B^4$
described by $t_1=t_2=0$ and $t_3,t_4\in\R$.  

If $S$ has two pairs of non-real nodes (hence no real nodes at all)
then the argument is essentially the same.  The difference is that
there are now 36 anti-involutions
$$
\alpha_{m,n}:(t_1,\dots,t_4)\mapsto(\bar t_2\zeta^m,\bar
t_1\zeta^m,\bar t_4\zeta^n, \bar t_3\zeta^n)\;,
$$
where $m,n\in\Z/6$, with fixed-point sets 
$$
H^4_{\alpha_{m,n}}=\{(t_1,\dots,t_4)\in B^4:t_2=\bar t_1\zeta^m, t_4=\bar
t_3\zeta^n\}\;. 
$$
If $(m',n')\neq(m,n)$ then $H^4_{\alpha_{m,n}}$ meets
$H^4_{\alpha_{m',n'}}$ in an $H^2$ if $m=m'$ or $n=n'$, and
otherwise only at the origin.  

If $S$ has a pair of non-real nodes and also a single real node then the
argument is a mix of the cases above.  $B_f\isomorphism\Z/6$ acts by
multiplying $t_3$ by powers of $\zeta$, and there are  36
anti-involutions, namely
$$
\alpha_{m,n}:(t_1,\dots,t_4)\mapsto(\bar t_2\zeta^m,\bar t_1\zeta^m,
\bar t_3\zeta^n, \bar t_4)\;.
$$
We have
$$
K_f=\{(t_1,\dots,t_4)\in B^4:t_2^6=\bar t_1^{\,6},\ t_3^6\in\R,\ t_4\in\R\}\;.
$$
The union $U$ of the subsets with $t_3$ or $it_3$ in $[0,\infty)$
is a fundamental domain for $B_f$; applying the identity map to
the first subset and $t_3\mapsto -it_3$ to the second identifies $U$
with 
$$
\{(t_1,\dots,t_4)\in B^4:t_2^6=\bar t_1^{\,6},t_3,t_4\in\R\}\;.
$$
That is, $B_f\backslash K_f$ is 
what $K_f$ was in the case of no real nodes, as claimed.  If  $S$ has two
non-real and two real nodes then the argument is only notationally
more complicated.

The remaining claims are trivial unless there are non-real nodes.  In
every case, the transitivity of $A_f$ on the $H^4$'s in $B_f\backslash
K_f$ is easy to see because $\PGamma_{\!f}\sset A_f$ contains
transformations multiplying $t_1,\dots,t_{2a}$ by powers of $\zeta$.
If $H$ is one of the $H^4$'s and $J=(A_f/B_f)_H$ is its stabilizer,
then it remains to prove that $J\backslash H\to A_f\backslash K_f$ is
an isometry.  Surjectivity follows from the transitivity of $A_f$ on
the $H^4$'s.  It is obviously a piecewise isometry, so all we must
prove is injectivity.  That is, if two points of $H$ are equivalent
under $A_f/B_f$, then they are equivalent under $J$.  To prove this it
suffices to show that for all $y\in B_f\backslash K_f$, the stabilizer
of $y$ in $A_f/B_f$ acts transitively on the $H^4$'s in $B_f\backslash
K_f$ containing $y$.  This is easy, using the stabilizer of $y$ in
$\PGamma_{\!f}/B_f\isomorphism (\Z/6)^{2a}$.
\end{proof}

\subsection{The hyperbolic orbifold structure}
We have equipped $\moduli_s^\R$ with a path metric which is locally
isometric to quotients of $H^4$ by finite groups.  It is easy to see
that if $X$ is such a metric space then there is a unique
real-hyperbolic orbifold structure on $X$ whose path metric is the
given one.  (The essential point is that if $U$ and $U'$ are connected
open subsets of $H^4$ and $\Gamma$ and $\Gamma'$ are finite groups of
isometries of $H^4$ preserving $U$ and $U'$ respectively, with
$\Gamma\backslash U$ isometric to $\Gamma'\backslash U'$, then there is an isometry of
$H^4$ carrying $U$ to $U'$ and $\Gamma$ to $\Gamma'$.)  Therefore
$\moduli_s^\R$ is a real hyperbolic orbifold.

For completeness, we give explicit orbifold charts.  Take $f$ as
before, and $H$ one of the $H^4$'s comprising $B_f\backslash K_f$.
Recall that $(A_f/B_f)_H$
is its stabilizer in $A_f/B_f$.  The orbifold chart is the restriction of the composition 
\begin{align*}
H&{}\to(A_f/B_f)_H\backslash H\\
&{}\isomorphism (A_f/B_f)\backslash(B_f\backslash K_f)\\
&{}=A_f\backslash K_f\\
&{}\isomorphism (A_f\times G^\R)\backslash(\cup_{i=1}^\ell\framed_s^{\,\alpha_i})\\
&\to (\PGamma\times G^\R)\backslash\framed_s^{\,\R}=\moduli_s^\R\;.
\end{align*}
 to a suitable open
subset of $H$.  The homeomorphism of the second line is part of
lemma~\ref{lem-preparation-for-orbifold-charts}, and that of the
fourth is
lemma~\ref{lem-local-homeomorphism-for-framed-stable-moduli}.  The map
in the last line is a homeomorphism in a neighborhood $U$ of the image
of $f$ in $(A_f\times G^\R)\backslash\bigl(\cup_{i=1}^\ell\framed_s^{\alpha_i}\bigr)$.
We take the domain of the orbifold chart to be the subset of $H$ which
is the preimage of $U$.

\begin{theorem}
\label{thm-MsR-uniformization}
With the orbifold structure given above, $\moduli_s^\R$ is a complete
real hyperbolic orbifold of finite volume, and there is a 
properly discontinuous group $\PGamma^\R$ of motions of $H^4$ such that
$\moduli_s^\R$ and 
$\PGamma^\R\backslash H^4$ are isomorphic hyperbolic orbifolds.
\end{theorem}

\begin{proof}
To prove $\moduli_s^\R$ complete, consider $K=G^\R\backslash\framed_s^{\,\R}$.
We know that $g$ maps $K$ to $\ch^4$; this is proper because any
compact set in $\ch^4$ meets only finitely many $H^4_{\x}$, $\x\in
P\A$, and $g$ carries each $G^\R\backslash\framed_s^{\,\x}$ homeomorphically to
$H^4_{\x}$ (lemma~\ref{lem-diffeomorphism-for-each-antiinvolution-stable}).  Since $K\to\ch^4$ is proper and $\PGamma\backslash\ch^4$
is complete, so is $\PGamma\backslash K$.

The uniformization theorem for complete hyperbolic orbifolds implies
the existence of $\PGamma^\R$ with the stated properties.  See
Proposition 13.3.2 of \cite{thurston} or Chapter IIIG of
\cite{bridson-haefliger} for discussion and proofs of this theorem.
The volume of $\moduli_s^\R$ is the sum of the volumes of the
$\PGamma_j^\R\backslash H^4_j$.
Since these have finite volume, so does $\moduli_s^\R$.
\end{proof}

\begin{remark}
It turns out  that the orbifold structures on $\moduli_s^\R$ and
$\PGamma^\R\backslash H^4$ differ on
$\moduli_s^\R-\moduli_0^\R$.  But they do define the same topological
orbifold structure, except along the locus of real surfaces having a
conjugate pair of nodes.  There, even the topological orbifold
structures differ.  
\end{remark}

\section{A fundamental domain for $\PGamma^\R$}
\label{sec-gluing}

In the previous section we equipped the moduli space $\moduli_s^\R$ of
stable real cubic surfaces with a complete hyperbolic orbifold
structure, so $\moduli_s^\R\isomorphism \PGamma^\R\backslash H^4$ for
some discrete group $\PGamma^\R$.  In this section we construct a
fundamental domain and the associated generators for $\PGamma^\R$.
Besides its intrinsic interest, this allows us to prove in
section~\ref{sec-nonarithmeticity} that $\PGamma^\R$ is nonarithmetic.  Throughout this
section, when we refer to $\moduli_s^\R$ as an orbifold, we refer to
the hyperbolic structure.  

\subsection{The tiling of $H^4$ by chambers}
We begin by explaining how the orbifold universal cover $H\isomorphism H^4$
of $\moduli_s^\R$ is tiled by copies of the polyhedra $C_j$ of
section~\ref{sec-H4-stabilizers}.  Consider the set of points in the orbifold
$\moduli_0^\R\sset\moduli_s^\R$ whose local group contains no
reflections, and its preimage under the orbifold covering map
$H\to\moduli_s^\R$.  Because the restriction of the hyperbolic
structure of $\moduli_s^\R$ to $\moduli_0^\R$ is the (incomplete)
structure described in section~\ref{sec-smoothmoduli}, each component of the preimage
is a copy of the interior of one of the $C_j$.  We call the closure of
such a component a chamber of type $j$.  It is clear that the union of
the chambers is $H$ and that their interiors are disjoint, so that
they tile $H$.

Recall from section~\ref{sec-topologysmoothmoduli} that we call a wall of a chamber a discriminant wall if it lies
over the discriminant, and an Eckardt wall otherwise.  By
theorem~\ref{thm-smooth-moduli-orbifolds}, it is a discriminant wall
if and only if it corresponds to a blackened node of $C_j$ in
figure~\ref{fig-small-Coxeter-diagrams}.  Because the orbifold
structure on $\moduli_s^\R$ restricts to that on $\moduli_0^\R$, every
point of an Eckardt wall is fixed by some reflection of $\PGamma^\R$.
Therefore $\PGamma^\R$ contains the reflections across the Eckardt
walls of the chambers.  The same argument shows that if a chamber has
type~1 or~2, so that it has a diagram automorphism, then some element
of $\PGamma^\R$ carries it to itself by this automorphism.

We have seen that across any Eckardt wall of a chamber lies another
chamber of the same type, in fact the mirror image of the first.  Now
we describe how the chambers meet across the discriminant walls.  This is
most easily understood by considering the 5 specific chambers
$C_j\sset H_j^4$ given in section~\ref{sec-H4-stabilizers}, regarding all the $H_j^4$'s
as lying in $\ch^4$.  Using the labeling of figure~\ref{fig-large-Coxeter-diagrams}, we refer
to the $k$th simple root of $C_j$ as $r_{jk}$ and to the corresponding
wall of $C_j$ as $C_{jk}$.  The following lemma leads to complete
information about how chambers meet across discriminant walls.

\begin{lemma}
\label{lem-intersections-in-ch4}
As subsets of $\ch^4$, we have $C_{04}=C_{14}$, $C_{13}=C_{24}$,
$C_{22}=C_{34}$ and $C_{31}=C_{44}$.  There is an element of
$\PGamma$ carrying $C_{37}$ isometrically to $C_{46}$.
\end{lemma}

\begin{proof}
The first assertion is just a calculation; it is even easy if
organized along the lines of the following treatment of the first
equality.  It is obvious that $r_{04}^\perp\sset H_0^4$ and
$r_{14}^\perp\sset H_1^4$ coincide, since $r_{14}=\theta\cdot
r_{04}$.  Simple roots describing $C_{04}$ may be obtained by
projecting the simple roots of $C_0$ into $r_{04}^\perp$, which
amounts to setting  the last coordinate equal to zero.  Simple roots describing
$C_{14}$ may be obtained by listing the walls of $C_1$ meeting
$C_{14}$, namely $C_{11}$, $C_{12}$, $C_{13}$ and $C_{16}$, and
projecting the corresponding roots into $r_{14}^\perp$, which again
amounts to setting the last coordinate to zero.  The two 4-tuples of
vectors so obtained coincide, so they define the same polyhedron in
$H_0^4\cap H_1^4\isomorphism H^3$.

Now we prove the second claim. 
Since only two discriminant walls remain unmatched, we expect $C_{37}$ to
coincide with some $\PGamma$-translate of $C_{46}$.  One can argue
that this must happen, but it is
easier to just find a suitable element $\gamma$ of $\PGamma$.   It
should take $\theta r_{37}$ to $r_{46}$; it should also carry
$r_{35}$, $r_{32}$, $r_{33}$ and $r_{36}$ to $r_{45}$, $r_{41}$,
$r_{42}$ and $r_{43}$ in the order stated.  These conditions determine
$\gamma$, which turns out to be
\renewcommand\-{\phantom{-}}
$$
\gamma=
\begin{pmatrix}
 {{10} + {6}\w}&{{4} + {2}\w}&{\-{1} - {4}\w}&{\-{1} - {4}\w}&{\-{1} - {4}\w}\cr
 {{\02} - {2}\w}&{1}&{-2-{2}\w}&{-2-{2}\w}&{-2-{2}\w}\cr
 {{\01} - {4}\w}&{ - {2}\w}&{-2-{2}\w}&{{-3} - {2}\w}&{{-3} - {2}\w}\cr
 {{\01} - {4}\w}&{ - {2}\w}&{{-3} - {2}\w}&{-2-{2}\w}&{{-3} - {2}\w}\cr
 {{\01} - {4}\w}&{ - {2}\w}&{{-3} - {2}\w}&{{-3} - {2}\w}&{-2-{2}\w}\cr
\end{pmatrix}\;,
$$
where we regard vectors as column vectors and $\gamma$ acts on the left.
Since $\gamma$ has entries in $\E$ and satisfies
$$
\gamma^T\cdot\diag[-1,1,1,1,1]\cdot\bar\gamma=\diag[-1,1,1,1,1]\;,
$$
it lies in $\PGamma$.  By construction, it carries $C_{37}$ to
$C_{46}$. 
\end{proof}

The lemma completes our picture of how the chambers meet along walls,
as follows.  Suppose $P$ (resp. $P'$) is a chamber of type~$0$
(resp.~$1$), with walls named $P_1,\dots,P_5$
(resp. $P_1',\dots,P_7'$) according to an isometric identification of
$P$ with $C_0$ (resp. $P'$ with $C_1$).  The lemma implies that
$P_4$ and $P_4'$ are identified under the map $H\to\moduli_s^\R$, so
there must be an element of $\PGamma^\R$ carrying $P_4$ to $P_4'$.
This implies that $P_4$ is a wall not only of $P$ but also of another
chamber, of type~$1$.  Applying this argument to the other cases of
the lemma implies that every discriminant wall of a chamber is also a discriminant
wall of another chamber, of known type.

\subsection{The polyhedron $Q$}
Now we construct what will turn out to be a fundamental domain for a
subgroup $\frac{1}{2}\PGamma^\R$ of index~2 in $\PGamma^\R$.  We
choose a chamber $P_0$ of type~$0$ and write $P_{0k}$ for its walls
corresponding to the $C_{0k}$ under the unique isometry
$P_0\isomorphism C_0$.  
(The detailed naming of walls is not needed for a conceptual understanding.)
Across its discriminant wall $P_{04}$ lies a chamber
$P_1$ of type~$1$; write $P_{1k}$ for its walls corresponding to
$C_{1k}$ under the unique isometry $P_1\isomorphism C_1$ that
identifies $P_{04}\sset P_1$ with $C_{14}$.  In particular,
$P_{04}=P_{14}$.  $P_1$ shares its discriminant wall $P_{17}$ with the image
$P_0'$ of $P_0$ under the diagram automorphism of $P_1$; we label the walls of
$P_0'$ by $P_{0k}'$ just as we did for $P_0$.  We write $P_2$ for the
chamber of type~$2$ on the other side of $P_{13}$.  There are two
isometries $P_2\isomorphism C_2$, both of which identify $P_{13}\sset
P_2$ with $C_{24}$, so we must work a little harder to fix our
labeling of the walls of $P_2$.  We choose the identification of $P_2$
with $C_2$ that identifies $P_{13}\cap P_{11}\sset P_2$ with
$C_{24}\cap C_{21}$, and label the walls $P_{2k}$ of $P_2$
accordingly.  Now, $P_2$ has three discriminant walls: it shares $P_{24}$
with $P_1$, and across $P_{22}$ and $P_{26}$ lie chambers of type~$3$.
We write $P_3$ for the one across $P_{22}$ and $P_3'$ for the one
across $P_{26}$; these chambers are exchanged by the diagram
automorphism.  Label the walls of $P_3$ by $P_{3k}$ according to the
unique isometry $P_3\isomorphism C_3$, and similarly for $P_3'$.
Finally, across $P_{31}$ lies a chamber $P_4$ of type~$4$, whose walls
we name $P_{4k}$ according to the isometry $P_4\isomorphism C_4$.
Similarly, $P_3'$ shares $P_{31}'$ with a chamber $P_4'$ which the
diagram automorphism exchanges with $P_4$.  We label the walls of
$P_4'$ accordingly.  Let $Q$ be the union of all eight chambers $P_0$,
$P_0'$, $P_1$, $P_2$, $P_3$, $P_3'$, $P_4$ and $P_4'$.  The
construction of $Q$ is summarized in figure~\ref{fig-assembly-of-Q}.

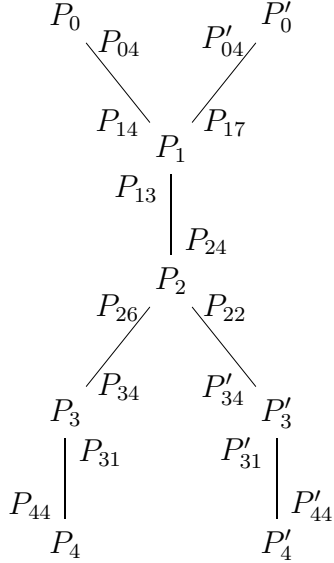
\begin{figure}
\begin{picture}(100,200)(0,0)
\put(10,200){\makebox(0,0)[c]{$P_0$}}
\put(90,200){\makebox(0,0)[c]{$P_0'$}}
\put(50,150){\makebox(0,0)[c]{$P_1$}}
\put(50,100){\makebox(0,0)[c]{$P_2$}}
\put(10,50){\makebox(0,0)[c]{$P_3$}}
\put(90,50){\makebox(0,0)[c]{$P_3'$}}
\put(10,0){\makebox(0,0)[c]{$P_4$}}
\put(90,0){\makebox(0,0)[c]{$P_4'$}}
\put(18,190){\line(4,-5){24}}
\put(82,190){\line(-4,-5){24}}
\put(22,185){\makebox(0,0)[bl]{$P_{04}$}}
\put(78,185){\makebox(0,0)[br]{$P_{04}'$}}
\put(38,165){\makebox(0,0)[tr]{$P_{14}$}}
\put(62,165){\makebox(0,0)[tl]{$P_{17}$}}
\put(50,140){\line(0,-1){30}}
\put(45,135){\makebox(0,0)[r]{$P_{13}$}}
\put(55,115){\makebox(0,0)[l]{$P_{24}$}}
\put(18,60){\line(4,5){24}}
\put(82,60){\line(-4,5){24}}
\put(22,65){\makebox(0,0)[tl]{$P_{34}$}}
\put(78,65){\makebox(0,0)[tr]{$P_{34}'$}}
\put(38,85){\makebox(0,0)[br]{$P_{26}$}}
\put(62,85){\makebox(0,0)[bl]{$P_{22}$}}
\put(10,40){\line(0,-1){30}}
\put(15,35){\makebox(0,0)[l]{$P_{31}$}}
\put(5,15){\makebox(0,0)[r]{$P_{44}$}}
\put(90,40){\line(0,-1){30}}
\put(85,35){\makebox(0,0)[r]{$P_{31}'$}}
\put(95,15){\makebox(0,0)[l]{$P_{44}'$}}
\end{picture}
\caption{Assembly of the polyhedron $Q$ from $8$ chambers.}
\label{fig-assembly-of-Q}
\end{figure}

We remark that the diagram automorphisms of $P_1$ and $P_2$ coincide,
in the sense that they are the same isometry of $H$, which we will
call $S$; this isometry preserves $Q$.  Throughout this section, ``the
diagram automorphism'' refers to $S$.

\begin{lemma}
\label{lem-Q-is-Coxeter-polyhedron}
$Q$ is a Coxeter polyhedron.
\end{lemma}

\begin{proof}
As a set, the boundary of $Q$ is the union of the Eckardt walls of the
$P_j$ and $P_j'$, together with $P_{37}$, $P_{46}$, $P_{37}'$ and
$P_{46}'$.  Suppose $W$ is an Eckardt wall of one of the $P_j$ or $P_j'$
and $H_W^3$ is the hyperplane in $H$ that it spans.  Then $Q$ lies
entirely in one of the closed half-spaces bounded by $H_W^3$, because 
$\PGamma^\R$ contains the reflection across $H_W^3$, while
no
point in the interior of $Q$ can be stabilized by a reflection of
$\PGamma^\R$.   We call $H_W^3\cap Q$ an Eckardt wall of $Q$.  Two
Eckardt walls of $Q$ that meet make interior angle $\pi/n$ for some
integer $n$, for otherwise some point in the interior of $Q$ would be
stabilized by a reflection.  

Now we claim that for $W=P_{37}$, $P_{46}$, $P_{37}'$ or $P_{46}'$,
the wall of $Q$ containing $W$ coincides with $W$, and its only
meetings with other walls of $Q$ are orthogonal intersections with
Eckardt walls.  We verify this for $W=P_{37}$; the key point is that
$P_{37}$ is orthogonal to all the walls of $P_3$ that it meets, namely
$P_{35}$, $P_{32}$, $P_{33}$ and $P_{36}$, and all these walls are
Eckardt walls of $P_3$.  By the above, we know that $Q$ lies in the
region bounded by the $H^3$'s containing $P_{35}$, $P_{32}$, $P_{33}$
and $P_{36}$, so the only walls of $Q$ which could meet $W$ are these
walls (or rather their extensions to walls of $Q$).  More precisely,
there is a neighborhood of $P_{37}$ in $H$ whose intersection with $Q$
coincides with its intersection with $P_3$.  All our claims follow
from this.  The same argument applies to $P_{46}$, and for the
remaining two walls we appeal to symmetry.
\end{proof}

\subsection{Simple roots for $Q$}
Since $Q$ is a Coxeter polyhedron, it may be described as the image in
$H^4$ of the set of vectors having $x\cdot s\leq0$ where $s$ varies
over a set of simple roots for $Q$.  There is one simple root for each
wall of $Q$, so we may find simple roots for $Q$ by taking all the
simple roots for the $P_j$ and $P_j'$, and discarding the ones
associated to the walls along which the $P_j$ and $P_j'$ meet.  We
will also discard duplicates, which occur when walls of two different $P_j$
or $P_j'$ lie in the same wall of $Q$.  

Therefore we will need to know simple roots for all the $P_j$ and
$P_j'$.  We identify $H$ with $H_1^4\sset\ch^4$, such that $P_0$ is
$C_0$.  Then $P_1$ is the image of $C_1\sset H_1^4\sset\ch^4$ under
the map
$$
T_1:(x_0,\dots,x_4)\mapsto(x_0,x_1,x_2,x_3,i x_4)\;,
$$
which is an isometry of $\ch^4$ (although not an element of $\PGamma$).  This
uses the facts that $C_{04}$ and $C_{14}$ coincide as subsets of
$\ch^4$ and $T_1$ carries $r_{14}$ to a negative multiple of $r_{04}$.
Similarly, using the intersections of $P_1$ with $P_2$, $P_2$ with
$P_3$, and $P_3$ with $P_4$ described in lemma~\ref{lem-intersections-in-ch4}, we find
\begin{align*}
P_2=T_2(C_2)\hbox{ where }T_2:(x_0,\dots,x_4)&\mapsto(x_0,\phantomi x_1,\phantomi x_2,i
x_3, ix_4),\\
P_3=T_3(C_3)\hbox{ where }T_3:(x_0,\dots,x_4)&\mapsto(x_0,\phantomi x_1,i x_2,i
x_3, ix_4)\hbox{, and}\\
P_4=T_4(C_4)\hbox{ where }T_4:(x_0,\dots,x_4)&\mapsto(x_0,i x_1,i x_2,i
x_3, ix_4)\;.
\end{align*}
For uniformity of notation we define $T_0$ to be the identity map.  In
all cases we have $P_{jk}=T_j(C_{jk})$; we selected our labelings of
the walls of the $P_j$ so that this would hold.  We write $s_{jk}$ for
$T_j(r_{jk})$, yielding simple roots for the $P_j$.  Given $r_{jk}$
from figure~\ref{fig-large-Coxeter-diagrams}, $s_{jk}$ is got by
replacing $\theta$ by $-\sqrt3$ wherever it appears.  

Since simple
roots for $P_1$ are now known, the matrix for the diagram automorphism
$S$ can be worked out, yielding
\begin{equation}
\label{eq-S}
S=
\begin{pmatrix}
 {3}&{2}&{1}&{0}&{-\sqrt3}\cr
 {-2}&{-1}&{-1}&{0}&{\sqrt3}\cr
 {-1}&{-1}&{-1}&{0}&{0}\cr
 {0}&{0}&{0}&{1}&{0}\cr
 {\sqrt3}&{\sqrt3}&{0}&{0}&{-1}\cr
\end{pmatrix}\;.
\end{equation}
Since $P_0'$, $P_3'$ and $P_4'$ are the images of $P_0$, $P_3$ and
$P_4$ under $S$, they are described by simple roots
$s_{jk}'=S\cdot s_{jk}$.  We now have explicit simple roots for all
eight chambers comprising $Q$.

To obtain simple roots for $Q$, we take all the $s_{jk}$ and $s_{jk}'$
and discard those involved in the gluing of
figure~\ref{fig-assembly-of-Q}, namely $s_{04}$, $s_{04}'$, $s_{14}$,
$s_{17}$, $s_{13}$, $s_{24}$, $s_{26}$, $s_{22}$, $s_{34}$, $s_{34}'$,
$s_{31}$, $s_{31}'$, $s_{44}$ and $s_{44}'$.  This leaves us with $36$
simple roots.  There is a great deal of duplication, for example
$s_{01}$ and $s_{43}$ are positive scalar multiples of each other.
After eliminating duplicates, only~$10$ remain, given in
table~\ref{tab-simple-roots-for-Q}.  We will indicate the walls of $Q$
by $A,\dots,E,E',\dots,A'$ and  corresponding simple roots by
$s_A,\dots,s_E,s_E',\dots,s_A'$.  We have scaled them so that $s_A$,
$s_B$, $s_B'$ and $s_A'$ have norm~$1$ and the rest have norm~$2$.  In
the table we also indicate which $P_{jk}$ and $P_{jk}'$ lie in each
wall of $Q$.  The diagram automorphism acts by exchanging primed and
unprimed letters.  With simple roots in hand, one can work out $Q$'s
dihedral angles, yielding figure~\ref{fig-gluedCoxeterDiagram} as the
Coxeter diagram of $Q$.

\begin{table}
\centeroverfull{%
\begin{tabular}{cll}
{\bf root}&{\bf coordinates}&{\bf $\hbox{\bf root}^\perp$ contains}\\
$s_A$&$(3,-1,\sqrt3,\sqrt3,\sqrt3)$&$P_{37}$\\
$s_B$&$(\sqrt3,1,1,1,1)$&$P_{46}$\\
$s_C$&$(1,-1,-1,-1,0)$&$P_{05}$, $P_{16}$, $P_{02}'$, $P_{33}'$, $P_{42}'$\\
$s_D$&$(\sqrt3,-\sqrt3,0,1,1)$&$P_{27}$, $P_{36}$, $P_{03}'$,
$P_{32}'$,  $P_{41}'$\\
$s_E$&$(0,1,-1,0,0)$&$P_{01}$, $P_{11}$, $P_{21}$, $P_{43}$, $P_{35}'$, $P_{45}'$\\
$s_E'$&$(1,0,0,0,\sqrt3)$&$P_{15}$, $P_{25}$, $P_{35}$, $P_{45}$,
$P_{01}'$, $P_{43}'$\\
$s_D'$&$(0,0,0,1,-1)$&$P_{03}$, $P_{23}$, $P_{32}$, $P_{41}$, $P_{36}'$\\
$s_C'$&$(0,0,1,-1,0)$&$P_{02}$, $P_{12}$, $P_{33}$, $P_{42}$, $P_{05}'$\\
$s_B'$&$(3+2\sqrt3,-2-\sqrt3,-2-\sqrt3,1,2+\sqrt3)$&$P_{46}'$\\
$s_A'$&$(4+\sqrt3,-2-\sqrt3,-2-\sqrt3,\sqrt3,\sqrt3)$&$P_{37}'$\\
\end{tabular}}
\smallskip\caption{Simple roots for the polyhedron $Q$.}
\label{tab-simple-roots-for-Q}
\end{table}

\setbox\Wzerobox=\hbox{%
\begin{picture}(0,0)
\largediagrams
\setlength{\unitlength}{1sp}
%
%
  \Ax=- .9pt   \Ay=  .5pt
  \Bx=- .9pt   \By=- .5pt
  \Cx=-1.5pt   \Cy= 1  pt
  \Dx=-1.5pt   \Dy=-1  pt
  \Ex=- .15pt  \Ey=- .7pt
  \multiply\Ax by\edgelengthC
  \multiply\Ay by\edgelengthC
  \multiply\Bx by\edgelengthC
  \multiply\By by\edgelengthC
  \multiply\Cx by\edgelengthC
  \multiply\Cy by\edgelengthC
  \multiply\Dx by\edgelengthC
  \multiply\Dy by\edgelengthC
  \multiply\Ex by\edgelengthC
  \multiply\Ey by\edgelengthC
  \heavybond{\Ax}{\Ay}{\Bx}{\By}
  \heavybond{-\Ax}{-\Ay}{-\Bx}{-\By}
  \bond{\Cx}{\Cy}{\Dx}{\Dy}
  \bond{-\Cx}{-\Cy}{-\Dx}{-\Dy}
  \bond{\Cx}{\Cy}{-\Dx}{-\Dy}
  \bond{-\Cx}{-\Cy}{\Dx}{\Dy}
  \bond{\Cx}{\Cy}{-\Ex}{-\Ey}
  \bond{-\Cx}{-\Cy}{\Ex}{\Ey}
  \dashedbond{\Ax}{\Ay}{\Cx}{\Cy}
  \dashedbond{-\Ax}{-\Ay}{-\Cx}{-\Cy}
  \dashedbond{\Bx}{\By}{\Cx}{\Cy}
  \dashedbond{-\Bx}{-\By}{-\Cx}{-\Cy}
  \dashedbond{\Bx}{\By}{\Dx}{\Dy}
  \dashedbond{-\Bx}{-\By}{-\Dx}{-\Dy}
  \dashedbond{\Ax}{\Ay}{\Ex}{\Ey}
  \dashedbond{-\Ax}{-\Ay}{-\Ex}{-\Ey}
  \dashedbond{\Ax}{\Ay}{-\Ax}{-\Ay}
  \dashedbond{\Bx}{\By}{-\Bx}{-\By}
  \dashedbond{\Ax}{\Ay}{-\Bx}{-\By}
  \dashedbond{-\Ax}{-\Ay}{\Bx}{\By}
  \triplebond{\Dx}{\Dy}{\Ex}{\Ey}{-.217pt}{.976pt}
  \triplebond{-\Dx}{-\Dy}{-\Ex}{-\Ey}{-.217pt}{.976pt}
  \solidnode{\Ax}{\Ay}
  \solidnode{-\Ax}{-\Ay}
  \solidnode{\Bx}{\By}
  \solidnode{-\Bx}{-\By}
  \hollownode{\Cx}{\Cy}
  \hollownode{-\Cx}{-\Cy}
  \hollownode{\Dx}{\Dy}
  \hollownode{-\Dx}{-\Dy}
  \hollownode{\Ex}{\Ey}
  \hollownode{-\Ex}{-\Ey}
  \nearnode\Ax\Ay{-70}{-70}{-2pt}{0pt}{tr}{$A$}
  \nearnode\Bx\By{-100}{0}{-2pt}{0pt}{r}{$B$}
  \nearnode\Cx\Cy{-100}{0}{-2pt}{0pt}{r}{$C$}
  \nearnode\Dx\Dy{-100}{0}{-2pt}{0pt}{r}{$D$}
  \nearnode\Ex\Ey{0}{-140}{0pt}{0pt}{t}{$E$}
  \nearnode{-\Ax}{-\Ay}{70}{70}{2pt}{0pt}{bl}{$A'$}
  \nearnode{-\Bx}{-\By}{100}{0}{2pt}{0pt}{l}{$B'$}
  \nearnode{-\Cx}{-\Cy}{100}{0}{2pt}{0pt}{l}{$C'$}
  \nearnode{-\Dx}{-\Dy}{100}{0}{2pt}{0pt}{l}{$D'$}
  \nearnode{-\Ex}{-\Ey}{0}{140}{0pt}{0pt}{b}{$E'$}
\end{picture}
}
\begin{figure}
\setlength{\unitlength}{1bp}
\begin{picture}(238,150)(-119,-75)
\unhbox\Wzerobox
\end{picture}
\caption{The polyhedron $Q$.}
\label{fig-gluedCoxeterDiagram}
\end{figure}

\subsection{$Q$ as a fundamental domain}
We already know that $\PGamma^\R$ contains the reflections across $C$,
$D$, $E$, $E'$, $D'$ and $C'$.  By lemma~\ref{lem-intersections-in-ch4}, $P_{37}$ and $P_{46}$
are identified in $\moduli_s^\R$, so there exists an element $\tau$ of
$\PGamma^\R$ carrying $A=P_{37}$ to $B=P_{46}$.  This transformation
must carry $P_3$ to the type~$3$ chamber on the other side of $P_{46}$
from $P_4$, and so it carries $s_A$ to $-s_B$.  By considering how the
walls of $Q$ meet $A$ and $B$, one sees that $\tau$ must fix each of $s_E'$,
$s_D'$ and $s_C'$, and carry $s_D$ to $s_E$.  
This determines $\tau$ uniquely:
\begin{equation}
\label{eq-translationmatrix}
\tau=
\begin{pmatrix}
 {{7} + {3}\sqrt3}&{{3} + \sqrt3}&{{-3} - {2}\sqrt3}&{{-3} - {2}\sqrt3}&{{-3} - {2}\sqrt3}\cr
 {{3} + \sqrt3\phantom{0}}&{1}&{{-1} - \sqrt3\phantom{0}}&{{-1} - \sqrt3\phantom{0}}&{{-1} - \sqrt3\phantom{0}}\cr
 {{3} + {2}\sqrt3}&{{1} + \sqrt3}&{{-1} - \sqrt3\phantom{0}}&{{-2} - \sqrt3\phantom{0}}&{{-2} - \sqrt3\phantom{0}}\cr
 {{3} + {2}\sqrt3}&{{1} + \sqrt3}&{{-2} - \sqrt3\phantom{0}}&{{-1} - \sqrt3\phantom{0}}&{{-2} - \sqrt3\phantom{0}}\cr
 {{3} + {2}\sqrt3}&{{1} + \sqrt3}&{{-2} - \sqrt3\phantom{0}}&{{-2} - \sqrt3\phantom{0}}&{{-1} - \sqrt3\phantom{0}}\cr
\end{pmatrix}\;.
\end{equation}
Of course, $\PGamma^\R$ also contains $\tau'=S\tau S$, which carries
$A'$ to 
$B'$. We define $\frac12\PGamma^\R$ to be the subgroup of $\PGamma^\R$
generated by $\tau$, $\tau'$ and the reflections in $C$, $D$, $E$, $E'$, $D'$ and $C'$.

\begin{lemma}
\label{lem-fundamental-domain-for-half-PGammaR}
$Q$ is a fundamental domain for $\frac12\PGamma^\R$.  More precisely,
the $\frac12\PGamma^\R$-images of $Q$ cover $H\isomorphism H^4$, and
the only identifications among points of $Q$ under
$Q\to\frac12\PGamma^\R\backslash H$ are that $A$ (resp. $A'$) is
identified with $B$ (resp. $B'$) by the action of $\tau$
(resp. $\tau'$). 
\end{lemma}

\begin{proof}
  All our claims follow from Poincar\'e's polyhedron theorem, as
  formulated in \cite[sec. IV.H]{maskit}.  There are 7 conditions to
  verify.  The key points are that any two Eckardt walls that
  intersect make an angle of the form $\pi/(\hbox{an integer})$, and
  that the 4 discriminant walls are disjoint from each other and
  orthogonal to the Eckardt walls that they meet.  These properties
  dispose of Maskit's conditions (i)--(vi).  Condition (vii) is that
  $Q$ modulo the identifications induced by $\tau$ and $\tau'$ is
  metrically complete.  This follows because we already know from
  theorem~\ref{thm-MsR-uniformization} that $\PGamma^\R\backslash H^4$ is complete, and
  $\frac{1}{2}\PGamma^\R\sset\PGamma^\R$.
\end{proof}

The main theorem of this section is now an easy consequence:

\begin{theorem}
\label{thm-description-of-PGammaR}
$\PGamma^\R=\bigl(\frac12\PGamma^\R\bigr)\semidirect\Z/2$, the $\Z/2$
being the diagram automorphism $S$.
\end{theorem}

\begin{proof}
Because $S$ sends $Q$ to itself, and $Q$ is a fundamental domain for
$\frac{1}{2}\PGamma^\R$, $S\notin\frac12\PGamma^\R$.  Since $S$
normalizes $\frac12\PGamma^\R$, we have 
$$
\bigl\langle\textstyle{\frac12\PGamma^\R}, S\bigr\rangle
=\frac12\PGamma^\R\semidirect\langle S\rangle\;.
$$
Since this larger group lies in $\PGamma^\R$ and has the same covolume
as $\PGamma^\R$, it equals $\PGamma^\R$.
\end{proof}

\begin{remark} Poincar\'e's polyhedron theorem readily gives a
  presentation for $\PGamma^\R$: there are generators $C,C',D,D',E,E'$
  (the reflections in the Eckardt walls of $Q$), $\tau,\tau'$ (the
  maps identifying $A$ with $B$, respectively $A'$ with $B'$), and $S$
  (the diagram automorphism) with the following relations.  (1) The
  subgroup generated by $C,D,E,C',D',E'$ has the Coxeter presentation
  indicated in the diagram. 
(2)
  $\tau$ commutes with $C',D',E'$ while $\tau D = E \tau$. 
(3) The
  relations obtained from (2) by interchanging the primed and unprimed
  letters.  
(4) $S$ is an involution and conjugation
  by it interchanges all the primed and unprimed generators. 
A presentation for $\frac12\PGamma^\R$ is obtained by
  deleting the generator $S$ and the relations (4).
\end{remark}

\subsection{The discriminant in
  $\moduli_s^\R\isomorphism\PGamma^\R\backslash H^4$}
We have now established
theorem~\ref{thm-main-theorem-stable}, except for the nonarithmeticity
and the fact that $\moduli_0^\R\sset\moduli_s^\R$ corresponds to
$\PGamma^\R\backslash(H^4-\H')$ where $\H'$ is a union of $H^2$'s and
$H^3$'s.  We will now address $\H'$; see the next section for the
nonarithmeticity.  The part of $Q$ that lies over the discriminant in
$\moduli_s^\R$ consists of (1) the walls $A$, $B$, $A'$ and $B'$, (2)
the faces corresponding to triple bonds in
figure~\ref{fig-gluedCoxeterDiagram}, and (3) the walls of the $P_j$
and $P_j'$ along which we
glued the 8 chambers to obtain $Q$.  We will refer to a wall of case
(3) as an `interior wall'.  Setting $\H'$ to be the preimage of
$\moduli_s^\R-\moduli_0^\R$ in $H^4$, we see that $\H'$ is the union
of the $\frac12\PGamma^\R$-translates of these three parts of $Q$.
The wall $A$ is orthogonal to all the walls of $Q$ that it meets, all
of which are Eckardt walls, so it's easy to see that the $H^3$
containing $A$ is covered by the $\frac12\PGamma^\R$-translates of
$A$.  The same argument applies with $B$, $A'$ or $B'$ in place of
$A$, and also applies in case (2), yielding $H^2$'s.

The essential facts for treating
case (3) are the following.  If $I$ is an interior wall, then every
wall $w$ of $Q$ with which $I$ has 2-dimensional intersection  is an
Eckardt wall of $Q$, and is either orthogonal to $I$ or makes angle
$\pi/4$ with it.  In the orthogonal case, it is obvious that $\H'$
contains the image of $I$ under reflection across $w$.  In the $\pi/4$ case,
one can check that there is another interior wall $I'$ with
$I'\cap w=I\cap w$, $\angle(w,I')=\pi/4$ and $I\perp I'$.  Then the
image of $I'$ under reflection across $w$ lies in the same $H^3$ as $I$ does.
Repeating this process, we see that the $H^3$ containing $I$ is tiled
by $\frac12\PGamma^\R$-translates of interior walls of $Q$.  It
follows that $\H'$ is a union of $H^2$'s and $H^3$'s.

We remark that the $H^3$ tiled by translates of interior walls can be viewed as a
3-dimensional analogue of our gluing process, describing moduli of
real 6-tuples in $\C P^1$; see \cite{allcock-lecnotes} and
\cite{toledo-lecnotes} for details.  In particular,
its stabilizer in $\PGamma^\R$ is the nonarithmetic group discussed
there.  Also, see \cite{toledo-lecnotes} for the 2-dimensional analogue.

\section{Nonarithmeticity}
\label{sec-nonarithmeticity}

This section is devoted to proving the following result:

\begin{theorem}
\label{thm-nonarithmetic}
$\PGamma^\R$ is a nonarithmetic lattice in $PO(4,1)$.
\end{theorem}

Our main tool is
Corollary~12.2.8 of \cite{DeligneMostow}.   We recall the context:
$G$ is an adjoint connected absolutely simple non compact real Lie
group, ${\bf G}$ is an adjoint connected simple algebraic group over
$\R$ so that $G$ is the identity component of ${\bf G}(\R)$, and
$L$ is a lattice in $G$.  Let $E=\Q[\Trace\Ad L]$, the
field generated over $\Q$ by
$\{\Trace\Ad\gamma:\gamma\in L\}$.  Assume that there is a
totally real number field $F$ and a form ${\bf G}_F$
of ${\bf G}$ over $F$ so that a subgroup of finite index of $L$
is contained in ${\bf G}(\EuScript{O}_F)$, where
$\EuScript{O}_F$ is the ring of integers in $F$.  It
follows that
$E\subset F$.  With this context in mind, the statement 
we will use is:

\begin{theorem}[\protect{\cite[Corollary~12.2.8]{DeligneMostow}}]
\label{thm-Deligne-Mostow}
A lattice $L\subset G$ is arithmetic
in $G$ if and only if for each embedding $\sigma$ of $F$ in
$\R$, not inducing the identity embedding of $E$ in $\R$, the real
group  ${\bf G}_F\tensor_{F,\sigma}\R$ is compact.
\qed
\end{theorem}

\begin{proof}[Proof of theorem~\ref{thm-nonarithmetic}:]
  To apply theorem~\ref{thm-Deligne-Mostow}, we take $G$ to be the
  connected component of $SO(4,1)$ and $L$ to be the subgroup of
  $\PGamma^\R$ that acts on $H^4$ by orientation-preserving
  isometries.  Note that $\Isom H^4=\PO(4,1)=\SO(4,1)$, so that $L$
  is indeed a subgroup of $G$.  We take ${\bf G}$ to be the special
  orthogonal group of the form $\diag\{-1,1,1,1,1\}$ and $F =
  \Q(\sqrt{3})$.  Then $\EuScript{O}_F=\Z[\sqrt3]$.  Note that ${\bf
    G}$ is defined over $\Q$, hence over $F$, and that $L\subset{\bf
    G}(\Z[\sqrt3])$.  For the last statement we need two observations.
  First, the matrices \eqref{eq-S} and \eqref{eq-translationmatrix} of
  $S$ and $\tau$ have entries in $\Z[\sqrt3]$.  Second, each root
  $s_A,\dots,s_E,s_E',\dots,s_A'$ in table~\ref{tab-simple-roots-for-Q} has coordinates in
  $\Z[\sqrt3]$ and norm $1$ or~$2$; it follows that the matrix of its
  reflection has entries in $\Z[\sqrt3]$.

Next we show that 
$E=\Q(\sqrt3)$ = $F$.  It is clear that $E$ is either $\Q$ or
$\Q(\sqrt{3})$.  To prove that  
$E=\Q(\sqrt3)$ it suffices to exhibit a single $\gamma\in L$
with $\Trace\Ad\gamma\notin\Q$.  Almost any $\gamma$ will do; we
take $\gamma=(R_CR_{D'}R_{E'})^2$, where the $R$'s are the reflections
in the corresponding simple roots from table~\ref{tab-simple-roots-for-Q}.  One can compute a matrix
for $\gamma$ and its square and compute their traces, yielding
$\Trace(\gamma)=13+6\sqrt3$ and $\Trace(\gamma^2)=209+120\sqrt3$.
Since the adjoint representation of $\O(4,1)$ is the exterior
square of the standard one, we can use the formula
$$
\Trace\Ad(\gamma)=\frac{1}{2}\left((\Trace(\gamma))^2-\Trace(\gamma^2)\right)
=34+18\sqrt3\notin\Q\;.
$$
This proves  $E=\Q(\sqrt3)$.

Finally, if $\sigma$ denotes the non-identity embedding of $F$ in
$\R$, then since $E=F$ it does not induce the identity embedding of
$E$.  Since the form $\diag\{-1,1,1,1,1\}$ defining ${\bf G}$ is fixed
by $\sigma$, the group ${\bf G}_F\tensor_{F,\sigma}\R$ is again the
non-compact group $SO(4,1)$.  Thus $L$ is not arithmetic.
\end{proof} 

\medskip

\begin{remarks}

  (1) In the introduction we said that our gluing construction is
  philosophically that of Gromov and Piatetski-Shapiro \cite{GPS}.
  But in technical detail it is quite different.  They glue hyperbolic
  manifolds with boundary whose fundamental groups are Zariski dense
  in $PO(n,1)$ and lie in non-commensurable arithmetic lattices.  We
  glue orbifolds with boundary and corners whose orbifold fundamental
  groups 
  are not Zariski dense in $PO(4,1)$ (but do lie in arithmetic lattices that are   not all commensurable).  In particular, we cannot
  directly apply their methods to prove non-arithmeticity.

  (2) We wonder whether the unimodular lattice $\diag\{-1,\breakok1,\breakok1,\breakok1,\breakok1\}$
  over $\Z[\sqrt3]$ plays some deeper geometric or arithmetic role.
  For example, $\PGamma^\R$ maps to $\PO(5,\F_3)\isomorphism W(E_6)$ by
  reduction modulo $\sqrt3$.  On each component of the smooth
  moduli space, the action on this $\F_3$ vector space is the same as
  the action on $V$ from sections~\ref{sec-five-families}
  and~\ref{sec-Segre}.  But it is not clear what this really means.
  
  (3) The group generated by reflections in the facets of $Q$, while
  being quite different from $\PGamma^\R$, also preserves this
  $\Z[\sqrt3]$-lattice and is also nonarithmetic.
\end{remarks}

\end{document}